\documentclass[11pt]{amsart}

\usepackage{amsmath, amssymb, amsfonts}
\usepackage{graphicx, overpic}
\usepackage{fullpage}
\usepackage[usenames,dvipsnames,svgnames,table]{xcolor}
\usepackage{enumerate}
\usepackage{hyperref}

\theoremstyle{plain}
\newtheorem{thm}{Theorem}[section]
\newtheorem{lemma}[thm]{Lemma}
\newtheorem{prop}[thm]{Proposition}
\newtheorem{cor}[thm]{Corollary}

\numberwithin{equation}{section}
\numberwithin{figure}{section}

\theoremstyle{definition}
\newtheorem{definition}[thm]{Definition}
\newtheorem{example}[thm]{Example}
\newtheorem{remark}[thm]{Remark}

\newtheorem{obs}[thm]{Observation}

\newtheorem{assumption}[thm]{Assumption}

\newtheorem{question}[thm]{Question}

\newcommand{\G}{\Gamma}
\newcommand{\Z}{\mathbb{Z}}

\newcommand{\N}{\mathbb{N}}
\newcommand{\Q}{\mathbb{Q}}

\newcommand{\cC}{\mathcal{C}}
\newcommand{\cG}{\mathcal{G}}
\newcommand{\cH}{\mathcal{H}}
\newcommand{\cO}{\mathcal{O}}
\newcommand{\cP}{\mathcal{P}}
\newcommand{\cR}{\mathcal{R}}
\newcommand{\la}{\langle}
\newcommand{\ra}{\rangle}

\newcommand{\half}{\mathcal{H}}

\newcommand{\cS}{\mathcal{S}}
\newcommand{\cA}{\mathcal{A}}

\newcommand{\cT}{\mathcal T}

\newcommand{\cQ}{\mathcal{Q}}
\newcommand{\cV}{\mathcal{V}}

\newcommand{\cX}{\mathcal{X}}
\newcommand{\cY}{\mathcal{Y}}
\newcommand{\cM}{\mathcal{M}}
\newcommand{\cW}{\mathcal{W}}
\newcommand{\cZ}{\mathcal{Z}}

\newcommand{\wv}{\widetilde{v}}
\newcommand{\ww}{\widetilde{w}}
\newcommand{\hw}{\widehat{w}}
\newcommand{\hwn}{\widehat{w_0}}
\newcommand{\hz}{\widehat{z}}

\newcommand{\cOg}{\mathcal{O}_{\Gamma}}
\newcommand{\cOgp}{\mathcal{O}_{\Gamma'}}

\newcommand{\gL}{\Lambda}

\newcommand{\intcap}{\, \footnotesize{\stackrel{\circ}{\cap}} \,}

\newcommand{\hide}[1]{}

\newcommand{\card}{\operatorname{card}}

\def\polhk#1{\setbox0=\hbox{#1}{\ooalign{\hidewidth
    \lower1.0ex\hbox{$\,\lhook$}\hidewidth\crcr\unhbox0}}}

\title[Commensurability for RACG's and geometric amalgams]{Commensurability for certain right-angled Coxeter groups and geometric amalgams of free groups}
\author{Pallavi Dani, Emily Stark and Anne Thomas}
\begin{document}
\maketitle

\begin{abstract}  We give explicit necessary and sufficient conditions for the abstract commensurability of certain families of $1$-ended, hyperbolic groups, namely right-angled Coxeter groups defined by generalized $\Theta$-graphs and cycles of generalized $\Theta$-graphs, and geometric amalgams of free groups whose JSJ graphs are trees of diameter $\le 4$.  We also show that if a geometric amalgam of free groups has JSJ graph a tree, then it is commensurable to a right-angled Coxeter group, and give an example of a geometric amalgam of free groups which is not quasi-isometric (hence not commensurable) to any group which is finitely generated by torsion elements. Our proofs involve a new geometric realization of the right-angled Coxeter groups we consider, such that covers corresponding to  torsion-free, finite-index subgroups are surface amalgams. \end{abstract}

\section{Introduction}

Two groups $G$ and $H$ are \emph{(abstractly) commensurable} if they contain finite-index subgroups $G' < G$ and $H' < H$ so that $G'$ and $H'$ are abstractly isomorphic.  There are related but stronger notions of commensurability for subgroups of a given group.  Commensurability in the sense we consider is an equivalence relation on abstract groups which implies quasi-isometric equivalence (for finitely generated groups).  In this paper, we give explicit necessary and sufficient conditions for commensurability within certain families of Coxeter groups and amalgams of free groups.   

Let $\G$ be a finite simplicial graph with vertex set $S$.  The \emph{right-angled Coxeter group} $W_\G$ with \emph{defining graph $\G$}
has generating set $S$ and relations $s^2 = 1$ for all $s \in S$, and $st = ts$ whenever $s, t \in S$ are adjacent vertices.  We assume throughout that  $W_\G$ is infinite, or equivalently, that $\G$ is not a complete graph. The graph $\G$ is \emph{3-convex} if every path between a pair of vertices of valence at least three in $\G$ has at least three edges.  For each induced subgraph $\Lambda$ of $\G$, with vertex set $A$, the corresponding \emph{special subgroup} of $W_\G$ is the right-angled Coxeter group $W_\Lambda = W_A = \langle A \rangle$.   See Section~\ref{sec:graphs} for additional terminology for graphs, and Figure~\ref{GraphsAmalgamsIntro} for some examples of $3$-convex defining graphs.

\begin{center}
\begin{figure}[htp]
\begin{overpic}[width=125mm]
{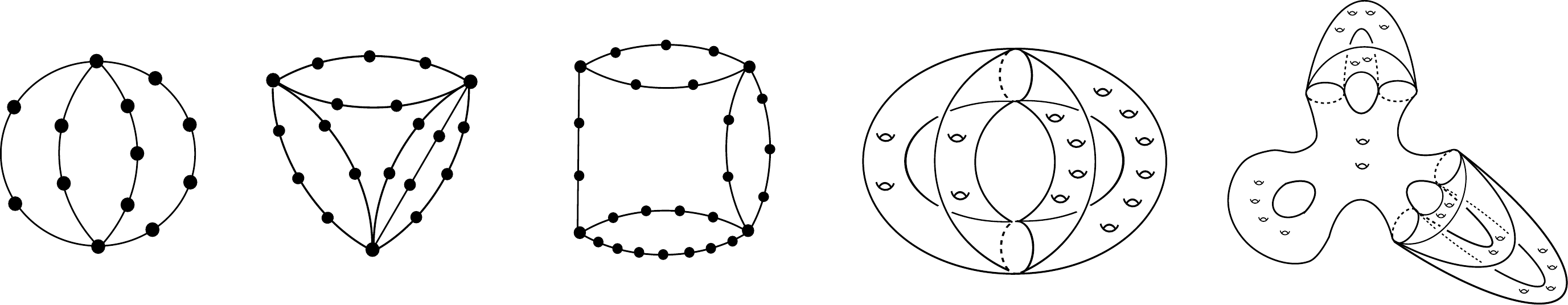}
\put(0,16){$\G$}
\put(15,16){$\G'$}
\put(33,16){$\G''$}
\put(55,16){$\cX$}
\put(78,16){$\cX'$}
\end{overpic}
\caption{\small{The graphs $\G$, $\G'$, and $\G''$ are examples of $3$-convex defining graphs, and the spaces $\cX$ and $\cX'$ are examples of surface amalgams.}}
\label{GraphsAmalgamsIntro}
\end{figure}
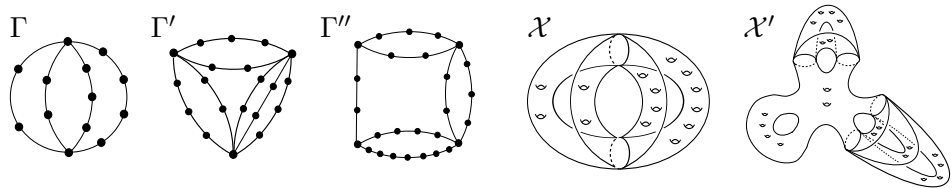
\end{center}

\emph{Geometric amalgams of free groups} were introduced by Lafont in~\cite{lafont}. These are fundamental groups of spaces called \emph{surface amalgams},  which, roughly speaking, consist of surfaces with boundary glued together along their boundary curves.  See 
Section~\ref{sec:lafont} for details, and Figure~\ref{GraphsAmalgamsIntro} for some examples. 

We now outline our main results and the ideas of their proofs, including some new constructions which may be of independent interest, and discuss previous work on related questions.  We defer precise definitions and theorem statements to Section~\ref{sec:defs-statements}.

In Theorems~\ref{thm:GenTheta} and~\ref{thm:CycleGenTheta}, we give explicit necessary and sufficient conditions for commensurability of right-angled Coxeter groups defined by two families of graphs.  Theorem~\ref{thm:GenTheta} classifies those defined by 
3-convex \emph{generalized $\Theta$-graphs} (see Definition~\ref{def:genTheta} and the graph~$\G$ in 
Figure~\ref{GraphsAmalgamsIntro}), and Theorem~\ref{thm:CycleGenTheta} classifies those defined by 
3-convex \emph{cycles of generalized $\Theta$-graphs} (see Definition~\ref{def:cyclegenTheta} and the graphs $\G'$ and $\G''$ in Figure~\ref{GraphsAmalgamsIntro}).  We prove 
Theorem~\ref{thm:GenTheta} in Section~\ref{sec:GenTheta}, and the necessary and sufficient conditions of Theorem~\ref{thm:CycleGenTheta} in Sections~\ref{sec:CycleGenThetaNec} and~\ref{sec:CycleGenThetaSuff}, respectively.
 Our commensurability criteria are families of equations involving the Euler characteristics of certain special subgroups, and we express these criteria using commensurability of vectors with entries determined by these Euler characteristics.  
 
In~\cite{crisp-paoluzzi}, Crisp and Paoluzzi classified up to commensurability the right-angled Coxeter groups defined by a certain family of three-branch generalized $\Theta$-graphs, and in Remark~\ref{rem:doubling} we recover their result using Theorem~\ref{thm:GenTheta}.  If $\G$ is a $3$-convex generalized $\Theta$-graph and $\G'$ a $3$-convex cycle of generalized $\Theta$-graphs, we can also determine the commensurability of $W_\G$ and $W_{\G'}$, as explained in Remark~\ref{rem:GenThetaCycleGenTheta}.  

The results described in the previous two paragraphs fit into a larger program of classifying 1-ended, hyperbolic right-angled Coxeter groups up to commensurability.  
Since groups that are commensurable are quasi-isometric, a step in this program is provided by Dani and Thomas~\cite{dani-thomas-jsj}, who considered Bowditch's JSJ tree~\cite{bowditch}, a quasi-isometry invariant for $1$-ended hyperbolic groups which are not cocompact Fuchsian.  If $G$ is such a group then $G$ acts cocompactly on its Bowditch JSJ tree $\cT_G$ with edge stabilizers maximal $2$-ended subgroups over which $G$ splits.  The quotient graph for the action of $G$ on $\cT_G$ is called the \emph{JSJ graph of $G$}, and the induced graph of groups is the \emph{JSJ decomposition for $G$}.  In Section~\ref{sec:jsj}, we recall results from~\cite{dani-thomas-jsj} that give an explicit ``visual" construction of the JSJ decomposition for right-angled Coxeter groups~$W_\G$ satisfying the following.

\begin{assumption} \label{assumptions}
The graph $\G$ has no triangles ($W_\G$ is 2-dimensional); 
$\G$ is connected and 
no vertex or (closed) edge separates $\G$ into two or more components
($W_\G$ is $1$-ended);
$\G$  has no squares ($W_\G$ is hyperbolic);
$\G$  is not a cycle of length $\ge 5$ ($W_\G$ is not cocompact Fuchsian); and $\G$ 
has a cut pair of vertices~$\{a, b\}$ ($W_\G$ splits over a $2$-ended subgroup).
\end{assumption}
\noindent Moreover, Theorem~1.3 of~\cite{dani-thomas-jsj} says that Bowditch's tree is a \emph{complete} quasi-isometry invariant for the family of groups satisfying, in addition, 
 \begin{assumption}\label{K4assumption}
 $\G$ has no induced subgraph which is a subdivided copy of $K_4$, the complete graph on four vertices. 
 \end{assumption}
 
We denote by $\cG$ the family of graphs satisfying Assumptions~\ref{assumptions} and~\ref{K4assumption}.  Generalized~$\Theta$-graphs and cycles of generalized $\Theta$-graphs are two infinite families of graphs in $\cG$.  Thus Theorems~\ref{thm:GenTheta} and~\ref{thm:CycleGenTheta} and Remarks~\ref{rem:doubling} and~\ref{rem:GenThetaCycleGenTheta} 
provide a finer classification (up to commensurability) within some quasi-isometry classes determined by graphs in $\cG$.   In  Section~\ref{sec:discussion} we discuss the obstructions to extending our results to other families of graphs in $\cG$.

Our proofs of the necessary conditions in Theorems~\ref{thm:GenTheta} 
and~\ref{thm:CycleGenTheta} follow the same general strategy used by Crisp--Paoluzzi~\cite{crisp-paoluzzi} and Stark~\cite{stark} on commensurability of certain geometric amalgams of free groups (discussed further below).  Given two groups which are commensurable, the first step in both these papers is to consider covers corresponding to 
isomorphic (torsion-free) finite-index subgroups.  In both cases such covers are surface amalgams, and a 
 crucial ingredient is Lafont's topological rigidity result from~\cite{lafont},
which says that any isomorphism between a pair of geometric amalgams of free groups is induced by a homeomorphism of the corresponding surface amalgams.  This homeomorphism between the covers is then analyzed to obtain the necessary conditions.  

The natural spaces to apply this strategy to in our setting are quotients $\Sigma_\G/G$ where $\Sigma_\G$ is the Davis complex for $W_\G$, and $G$ is a torsion-free, finite-index subgroup of $W_\G$.  However, Stark proves in~\cite{stark-rigidity} that topological rigidity fails for such quotients, by constructing an example where $G$ and $G'$ are isomorphic torsion-free, finite-index subgroups of $W_\G$, but $\Sigma_\G/G$ and $\Sigma_\G/G'$ are not homeomorphic.  The graph $\G$ in this example is a $3$-convex cycle of generalized $\Theta$-graphs.

In light of the result of~\cite{stark-rigidity}, in Section~\ref{sec:orbicomplex} we introduce a new geometric realization for  right-angled Coxeter groups~$W_\G$ with  $3$-convex $\G \in \cG$, by constructing a piecewise hyperbolic orbicomplex $\cO_\G$ with fundamental group $W_\G$.  The orbicomplex $\cO_\G$ has underlying space obtained by gluing together right-angled hyperbolic polygons, and each edge of $\cO_\G$ which is contained in only one such polygon is a reflection edge.  It follows that any cover of $\cO_\G$ corresponding to a torsion-free, finite-index subgroup of $W_\G$ is a surface amalgam (tiled by right-angled polygons).  Thus we can apply Lafont's result to these spaces.  With a view to generalizing the commensurability classification, we give the construction of the orbicomplex $\cO_\G$ for all 
$3$-convex graphs in $\cG$, not just for generalized $\Theta$-graphs and cycles of generalized $\Theta$-graphs.  

Our construction of $\cO_\G$ makes heavy use of the JSJ decomposition from~\cite{dani-thomas-jsj}.  We restrict to $3$-convex defining graphs in this paper so that the correspondence between $\G$ and the JSJ decomposition of $W_\G$ is more straightforward than the general case in~\cite{dani-thomas-jsj}.  A reference for orbifolds is Kapovich \cite{kapovich}; we view orbicomplexes as complexes of groups, and use the theory of these from Bridson--Haefliger~\cite{bridson-haefliger}.

The proofs of the necessary conditions in Theorems~\ref{thm:GenTheta} 
and~\ref{thm:CycleGenTheta} then involve a careful analysis of the homeomorphic finite covers guaranteed by topological rigidity. For generalized $\Theta$-graphs, we adapt Stark's proof 
in~\cite{stark} to the setting where the orbicomplexes considered do not have the same Euler characteristic. The proof of the necessary conditions for cycles of generalized $\Theta$-graphs is considerably more difficult.  Here, the groups $W_\G$ and $W_{\G'}$ are fundamental groups of orbicomplexes $\cO_\G$ and $\cO_{\G'}$ with ``central" orbifolds $\cA$ and $\cA'$ that have  many branching edges along which other orbifolds are attached (see the lower right of 
Figure~\ref{CycleGenTheta}).  A key ingredient in the proof of Theorem~\ref{thm:CycleGenTheta} is a careful argument to show that the homeomorphism $f:\cX \to \cX'$ between finite covers $\pi:\cX \to \cO_\G$ and $\pi':\cX' \to \cO_{\G'}$ guaranteed by Lafont's topological rigidity result can be modified so that either $f(\pi^{-1}(\cA)) = \pi'^{-1}(\cA')$, or $f(\pi^{-1}(\cA))$ has no surfaces in common with $\pi'^{-1}(\cA')$.

To prove the sufficient conditions in Theorem~\ref{thm:GenTheta} and~\ref{thm:CycleGenTheta}, 
given any pair of groups $W_\G$ and $W_{\G'}$ which satisfy the putative sufficient conditions, we construct finite-sheeted covers of $\cO_\G$ and $\cO_{\G'}$ with isomorphic fundamental groups.  It follows that $W_\G$ and $W_{\G'}$ have isomorphic finite-index subgroups.  In the case of generalized $\Theta$-graphs, these finite-sheeted covers are orbicomplexes whose construction is an immediate generalization of Crisp and Paoluzzi's.  The finite covers in the case of cycles of generalized $\Theta$-graphs are homeomorphic surface amalgams, and their construction is quite  delicate.  

In order to explain our constructions of surface amalgams covering $\cO_\G$, we introduce the notion of a \emph{half-covering} of graphs (see Section~\ref{sec:half-coverings}).  The idea is that if a surface amalgam~$\cY$ is a finite-sheeted cover of another surface amalgam $\cX$, or of an orbicomplex $\cO_\G$, then the induced map of JSJ graphs is a half-covering.  For the proofs of the sufficient conditions in Theorem~\ref{thm:CycleGenTheta}, we first construct the JSJ graphs for the homeomorphic finite-sheeted covers, together with the associated half-coverings.   We then construct suitable surfaces and glue them together along their boundaries, guided by adjacencies in the JSJ graphs, to obtain a surface amalgam covering $\cO_\G$.

One construction in our proof of sufficient conditions for Theorem~\ref{thm:CycleGenTheta} may be of independent interest.  In Section~\ref{sec:tfcover}, given any orbicomplex $\cO_\G$ (with $\G \in \cG$ and $\G$ being $3$-convex), we construct a particularly nice degree 16 cover $\cX$ which is a surface amalgam. Hence $W_\G$ has an index 16 subgroup which is a geometric amalgam of free groups (this result is stated below as  Theorem~\ref{thm:Degree16}).  Our construction of $\cX$ uses tilings of surfaces similar to those in Futer and Thomas~\cite[Section 6.1]{futer-thomas}.   Given two groups $W_\G$ and $W_{\G'}$ which satisfy the sufficient conditions from Theorem~\ref{thm:CycleGenTheta}, we first pass to our degree 16 covers $\cX$ and $\cX'$ of $\cO_\G$ and $\cO_{\G'}$, then construct homeomorphic finite-sheeted covers of $\cX$ and $\cX'$.  

Theorem~\ref{thm:GeomAmalgamsRACGs}, proved in Section~\ref{sec:GeomAmalgamsRACGs}, says that for all geometric amalgams of free groups $G$, if the JSJ graph of $G$ is a tree, then $G$ is commensurable to some $W_\G$ (with $\G \in \cG$).  In Section~\ref{sec:GeomAmalgamsRACGs} we also give an example of a geometric amalgam of free groups which is not quasi-isometric (and hence not commensurable) to any $W_\G$, or indeed to any group finitely generated by torsion elements.
This leads to the question:

\begin{question}  Which geometric amalgams of free groups are commensurable to right-angled Coxeter groups (with defining graphs in $\cG$)? 
\end{question}

The proof of Theorem~\ref{thm:GeomAmalgamsRACGs} uses the torsion-free degree 16 cover of the orbicomplex $\cO_\G$ that we construct in Section~\ref{sec:tfcover}.  We show that any surface amalgam whose JSJ graph is a tree admits a finite-sheeted cover which ``looks like'' our torsion-free cover.  Then we follow our construction in Section~\ref{sec:tfcover} backwards to obtain an orbicomplex $\cO_\G$ as a finite quotient, with fundamental group the right-angled Coxeter group $W_\G$.

As a corollary to Theorems~\ref{thm:GenTheta},~\ref{thm:CycleGenTheta}, and~\ref{thm:GeomAmalgamsRACGs}, we obtain the commensurability classification of geometric amalgams of free groups whose JSJ graph is a tree with diameter at most~4 (see Corollary~\ref{cor:GeomAmalgamsCommens}).  This recovers part of Theorem 3.31 of Stark \cite{stark}, which gives the commensurability classification of fundamental groups of surface amalgams obtained by identifying two closed surfaces along an essential simple closed curve in each.  We remark that Malone \cite{malone} provides a complete quasi-isometry classification within the class of geometric amalgams of free groups; in particular, he proves that the isomorphism type of Bowditch's JSJ tree for such a group is a complete quasi-isometry invariant, and supplies an algorithm to compute this tree.  

We conclude this part of the introduction by mentioning some earlier work on commensurability and on the relationship between commensurability and quasi-isometry for groups closely related to the ones we study.  We refer to the surveys by Paoluzzi~\cite{paoluzzi} and Walsh~\cite{walsh} for more comprehensive accounts.  
In~\cite{jkrt} 
and~\cite{gjk},   
the related notion of wide commensurability is studied for Coxeter groups generated by reflections in the faces of polytopes in $n$-dimensional real hyperbolic space.  Apart from these two papers and~\cite{crisp-paoluzzi}, we do not know of any other work on commensurability for (infinite non-affine) Coxeter groups.  

In~\cite{crisp}, Crisp investigated commensurability in certain 2-dimensional Artin groups, while Huang has studied quasi-isometry and commensurability in right-angled Artin groups with finite outer automorphism groups in~\cite{huang-qi} and~\cite{huang-commens}.  Huang's combined results show that 
within a class of right-angled Artin groups defined by a few additional conditions on the defining graph,
the quasi-isometry, commensurability and isomorphism classes are the same.  We note that none of the groups we consider is quasi-isometric to a right-angled Artin group, since the groups we consider are all $1$-ended and hyperbolic, and a right-angled Artin group is hyperbolic if and only if it is free.

The above results of Huang are in contrast to our settings of right-angled Coxeter groups and geometric amalgams of free groups:
Theorems~\ref{thm:GenTheta},~\ref{thm:CycleGenTheta}, and~\ref{thm:GeomAmalgamsRACGs} above, together with the descriptions of Bowditch's JSJ tree from~\cite{dani-thomas-jsj} and~\cite{malone}, show that each quasi-isometry class containing one of the groups considered in our theorems contains infinitely many commensurability classes.  For geometric amalgams of free groups, Malone~\cite{malone} and Stark~\cite{stark} had both given examples to show that commensurability and quasi-isometry are different.

\subsection{Definitions and statements of results}\label{sec:defs-statements}  We now give precise definitions and statements for our main results.

First, we recall the definition of the Euler characteristic of a Coxeter group in the case we will need.  A reference for the general definition is pp. 309--310 of~\cite{davis-book}.

\begin{definition}[Euler characteristic of $W_\G$]\label{def:euler} Let $W_\G$ be a right-angled Coxeter group with defining graph $\G$ having vertex set $V(\G)$ and edge set $E(\G)$.  Assume that $\G$ is triangle-free.  Then the \emph{Euler characteristic} of $W_\G$ is the rational number $\chi(W_\G)$ given by:
\[
\chi(W_\G) = 1 - \frac{\card V(\G)}{2}  + \frac{\card E(\G)}{4}.
\]
\end{definition}

\noindent We remark that $\chi(W_\G)$ is equal to the Euler characteristic of the orbicomplex induced by the action of $W_\G$ on its Davis complex (we will not need this interpretation).

To each group we consider, we associate a collection of vectors involving the Euler characteristics of certain special subgroups.  We will use the following notions concerning commensurability of vectors.

 \begin{definition}[Commensurability of vectors]
 Let $k, \ell \geq 1$.  Vectors $v \in \Q^k$ and $w  \in \Q^{\ell}$ are {\it commensurable} if $k = \ell$ and there exist non-zero integers $K$ and $L$ so that $Kv = Lw$.   Given a nontrivial commensurability class $\cV$ of vectors in $\Q^k$, the \emph{minimal integral element} of $\cV$ is the unique vector $v_0 \in \cV \cap \Z^k$ so that the entries of $v_0$ have greatest common divisor~$1$ and the first nonzero entry of $v_0$ is a positive integer.  Then for each $v \in \cV$, there is a unique rational $R = R(v)$ so that $v = Rv_0$.
 \end{definition}
 
Our first main result, Theorem~\ref{thm:GenTheta}, classifies the commensurability classes among right-angled Coxeter groups with defining graphs  
generalized $\Theta$-graphs, defined as follows.

\begin{definition}[Generalized $\Theta$-graph]\label{def:genTheta}
For $k\ge 1$, let $\Psi_k$ be the graph with two vertices $a$ and~$b$, each of valence $k$, and $k$ edges $e_1, \dots, e_k$ connecting $a$ and $b$.  For integers $0 \leq n_1 \leq \dots \leq n_k$, the \emph{generalized $\Theta$-graph} $\Theta(n_1, n_2, \dots, n_k)$ is obtained by subdividing the edge $e_i$ of $\Psi_k$ into $n_i+1$ edges by inserting $n_i$ new vertices along $e_i$, for $1 \leq i \leq k$.   
\end{definition}

\noindent For example, the graph $\G$ in Figure~\ref{GraphsAmalgamsIntro}  is $\Theta(2,2,3,4)$.  
 We write $\beta_i$ for the induced subgraph of $\Theta = \Theta(n_1, n_2, \dots, n_k)$ which was obtained by subdividing the edge $e_i$ of $\Psi_k$, and call $\beta_i$ the \emph{$i$th branch} of~$\Theta$.  A generalized $\Theta$-graph is $3$-convex if and only if $n_i \ge 2$ for all $i$, equivalently $n_1 \geq 2$. 
Note that if $\Theta$ is $3$-convex and has at least 3 branches then $\Theta$ satisfies 
Assumptions~\ref{assumptions} and~\ref{K4assumption}.

To each generalized $\Theta$-graph, we associate the following Euler characteristic vector. 

 \begin{definition}[Euler characteristic vector for generalized $\Theta$-graphs]\label{def:theta-euler}
   Let $W_{\Theta}$ be the right-angled Coxeter group with defining graph 
   $\Theta = \Theta(n_1, \ldots, n_k)$.  The {\it Euler characteristic vector of $W_{\Theta}$} is $v= (\chi(W_{\beta_1}), \dots, \chi(W_{\beta_k}))$.  \end{definition}

\noindent Note that, by definition, $\chi(W_{\beta_1}) \geq \ldots \geq \chi(W_{\beta_k})$.  For example, if $\G = \Theta(2,2,3,4)$ as in Figure~\ref{GraphsAmalgamsIntro}, then the Euler characteristic vector of $W_\G$ is $v = (-\frac{1}{4},-\frac{1}{4},-\frac{1}{2},-\frac{3}{4})$.

The commensurability classification for the corresponding right-angled Coxeter groups is then as follows.  

\begin{thm}\label{thm:GenTheta} Let $\Theta$ and $\Theta'$ be $3$-convex generalized $\Theta$-graphs with at least 3 branches, and let $v$ and~$v'$ be the Euler characteristic vectors of the right-angled Coxeter groups $W_{\Theta}$ and $W_{\Theta'}$, respectively. Then $W_{\Theta}$ and $W_{\Theta'}$ are abstractly commensurable if and only if $v$ and $v'$ are commensurable.
\end{thm}

\begin{remark}\label{rem:doubling}
We now explain how Theorem~\ref{thm:GenTheta} can be used to classify all right-angled Coxeter groups defined by generalized $\Theta$-graphs satisfying   Assumptions~\ref{assumptions} (so that the corresponding groups are $1$-ended and hyperbolic), with $\Theta$ not required to be $3$-convex.   

If $\Theta = \Theta(n_1,\dots,n_k)$ satisfies every condition in Assumptions~\ref{assumptions}, it is easy to check that $k \geq 3$, $n_1 \geq 1$, and $n_i \geq 2$ for all $i \geq 2$.  Theorem~\ref{thm:GenTheta} covers the case $n_1 \geq 2$, so we just need to discuss the case $n_1=1$.  

Let $c$ be the unique vertex of valence $2$ on the first 
branch of $\Theta$.  Then we may form the \emph{double of~$\Theta$ over $c$}, defined by
$D_c(\Theta) = 
\Theta(n_2, n_2, n_3, n_3, \dots, n_k, n_k)$.
This is a $3$-convex generalized $\Theta$-graph with $2(k-1) \geq 4$ branches, obtained from $\Theta$ by deleting the open star of~$c$, and then identifying two 
copies of the resulting graph along $a$ and $b$.  The group $W_{D_c(\Theta)}$ is isomorphic to the kernel of the map 
$W_{\Theta} \to \Z/2\Z$ which sends $c$ to 1 and all other generators to 0.  In particular, $W_{D_c(\Theta)}$ is 
commensurable to $W_\Theta$.  This, together with Theorem~\ref{thm:GenTheta},
can be used to extend our classification result above to all generalized 
$\Theta$-graphs 
satisfying Assumptions~\ref{assumptions}.  In this way we recover Crisp and Paoluzzi's result 
from~\cite{crisp-paoluzzi}, as they  considered the family of graphs $\Theta(1, m+1,n+1)$ with $m,n \ge 1$.
\end{remark}

Next, in Theorem~\ref{thm:CycleGenTheta} we consider groups $W_\G$ where $\G$ is a cycle of generalized $\Theta$-graphs, defined as follows.  

\begin{definition}
[Cycle of generalized $\Theta$-graphs]\label{def:cyclegenTheta}
Let $N\ge 3$ and let $r_1, \dots, r_N$ be positive integers so that for each $i$, at most 
one of $r_i$ and $r_{i+1}$ (mod $N$) is equal to 1.  Now for $1 \le i \le N$, let $\Psi_{r_i}$ be the graph from Definition~\ref{def:genTheta}, with $r_i$ edges between $a_i$ and $b_i$.  Let $\Psi$ be the graph obtained by identifying $b_i$ with $a_{i+1}$ for all $i$ (mod $N$). 
A \emph{cycle of $N$ generalized $\Theta$-graphs}  is a graph obtained from $\Psi$ 
 by (possibly) subdividing edges of $\Psi$.   \end{definition}

\noindent For example, the graph $\G'$ (respectively, $\G''$) in Figure~\ref{GraphsAmalgamsIntro} is a cycle of three (respectively, four) generalized $\Theta$-graphs. If $\G$ is a cycle of generalized $\Theta$-graphs, we denote by $\Theta_i$ the $i$th generalized $\Theta$-graph of $\G$, that is, 
the subdivided copy of $\Psi_{r_i}$ inside $\G$, and we say that $\Theta_i$ is
\emph{nontrivial} if~$r_i>1$.  
Observe that the condition on the $r_i$ guarantees that if $\Theta_i$ is 
trivial, then $\Theta_{i-1}$ and $\Theta_{i+1}$ are not, hence the vertices 
$a_1, \dots, a_N$ have valence at least three in $\G$. 
It follows that a cycle of generalized $\Theta$-graphs $\G$ is $3$-convex if and only if each edge of $\Psi$ is subdivided into at least three edges, by inserting at least two vertices.  

We now use this notation to define Euler characteristic vectors associated to cycles of generalized $\Theta$-graphs.

\begin{definition}\label{def:cycleEuler}
[Euler characteristic vectors for cycles of generalized $\Theta$-graphs]
Let $\G$ be a 3-convex cycle of $N$ generalized $\Theta$-graphs and let 
$I\subset \{1, \dots, N\}$ be the set of indices with 
 $r_i>1$ (so~$I$ records the indices $i$ for which $\Theta_i$ is nontrivial).
 \begin{enumerate}  
\item For each $i \in I$,  we define the vector $v_i \in \Q^{r_i}$ to be 
the Euler characteristic vector of $W_{\Theta_i}$ (from Definition~\ref{def:theta-euler}).  Thus if 
$\Theta_i$ has  branches $\beta_{i1}, \dots, \beta_{ir_i}$, with $r_i > 1$, then \[v_i = (\chi(W_{\beta_{i1}}), \ldots, \chi(W_{\beta_{ir_i}})) =: (\chi_{i1}, \ldots, \chi_{ir_i}).\]

\item If there is some $r \geq 2$ so that each nontrivial $\Theta_i$ has exactly $r$ branches, we define another vector $w$ associated to $\G$ as follows.  Let $A$  be the union of $\{a_1,\dots,a_N\}$ with the vertex sets of all trivial $\Theta_i$.  
Recall from the beginning of this section that $W_A$ is the special subgroup of $W_\G$ defined by the subgraph of~$\G$ induced by $A$.  
Then $w \in \Q^{r+1}$ is the reordering of the vector 
 \[ \left( \sum_{i \in I} \chi_{i1}, \,\sum_{i  \in I} \chi_{i2}, \, \ldots, \sum_{i \in I} \chi_{ir}, \, \chi(W_A)\right) \]
 obtained by putting its entries in non-increasing order. 
 \end{enumerate}
\end{definition}

\noindent For example, the graph $\G'$ in Figure~\ref{GraphsAmalgamsIntro} has $I = \{1,2,3\}$, with say $v_1 = v_2 = (-\frac{1}{4},-\frac{1}{2})$ and $v_3 = (-\frac{1}{4},-\frac{1}{4},-\frac{1}{2})$.  If $\G''$ is as in Figure~\ref{GraphsAmalgamsIntro}, then $r = 2$, we can choose $I = \{1,2,3\}$, and then $v_1 = v_2 = (-\frac{1}{4},-\frac{1}{2})$, $v_3 = (-\frac{3}{4},-\frac{3}{2})$, $\chi(W_A) = -\frac{5}{4}$, and $w = (-\frac{5}{4},-\frac{5}{4},-\frac{5}{2})$.

We can now state the commensurability classification of right-angled Coxeter groups defined by 3-convex cycles of generalized $\Theta$-graphs.  

\begin{thm}\label{thm:CycleGenTheta}  
 Let $\Gamma$ and $\Gamma'$ be 3-convex cycles of $N$ and $N'$ generalized $\Theta$-graphs, 
respectively (with $N, N' \geq 3$).  Let $r_i$ be the number of branches of the $i$th generalized $\Theta$-graph $\Theta_i$ in $\G$, and let 
$I$ be the set of indices with $r_i>1$.  Let $\{v_i \mid i \in I\}$ and $w$ be the vectors from 
Definition~\ref{def:cycleEuler}, and let $W_A$ be the special subgroup from 
Definition~\ref{def:cycleEuler}(2).  Here, $\{ v_i \mid i \in I \}$ denotes a multiset of vectors (since the $v_i$ may not all be distinct).  Let $r_k'$, $\Theta'_k$, $I', \{v_k'\mid k \in I'\}$, $w'$, and $W_{A'}$ be the corresponding objects for $\G'$.  

The right-angled Coxeter groups $W = W_{\Gamma}$ and $W'= W_{\Gamma'}$ are abstractly commensurable if and only if at least one of (1) or (2) below holds.

\medskip
\begin{enumerate}
\item 
\begin{enumerate}
    \item The set of commensurability classes of the vectors  $\{v_i \mid i\in I\}$ coincides with the set of commensurability classes of the vectors $\{v_k' \mid k \in I'\}$; and

\medskip    
\item 
given a commensurability class of vectors $\cV$, if $I_{\cV}\subset I$ is the set of indices of the vectors in $\{v_i \mid i \in I\}\cap \cV$, and $I'_{\cV} \subset I'$  is the set of indices of the vectors in 
 $\{v_k' \mid k\in I'\}\cap \cV$, 
then
$$
\chi(W_{A'})\cdot \left( \displaystyle \sum_{i \in {I_{\cV}}} \chi(W_{\Theta_i}) \right)=
\chi(W_A)\cdot \left( \displaystyle  \sum_{k \in {I'_{\cV}}}  \chi(W_{\Theta'_k})\right).
$$
   \end{enumerate}  

\item There exists $r \geq 2$ such that each nontrivial generalized $\Theta$-graph in $\Gamma$ and in $\Gamma'$ has~$r$ branches, and:
  \begin{enumerate}
   \item the vectors $v_i$ for $ i \in I$ are contained in a single commensurability class; likewise, the vectors $v_k'$ for $k\in I'$ are contained in a single commensurability class; and
   \item the vectors $w$ and $w'$ are commensurable. 
  \end{enumerate}
\end{enumerate}
\end{thm}

\begin{remark}\label{rem:Conditions}  For a pair of graphs $\G$ and $\G'$ as in the statement of 
Theorem~\ref{thm:CycleGenTheta}, with $W_\G$ and $W_{\G'}$ abstractly commensurable, it is possible that (1) holds but not (2), that (2) holds but not (1), or that both (1) and (2) hold.
See Figure~\ref{fig:Conditions} for examples of this. 

\begin{center}
\begin{figure}[htp]
\begin{overpic}[width=125mm]
{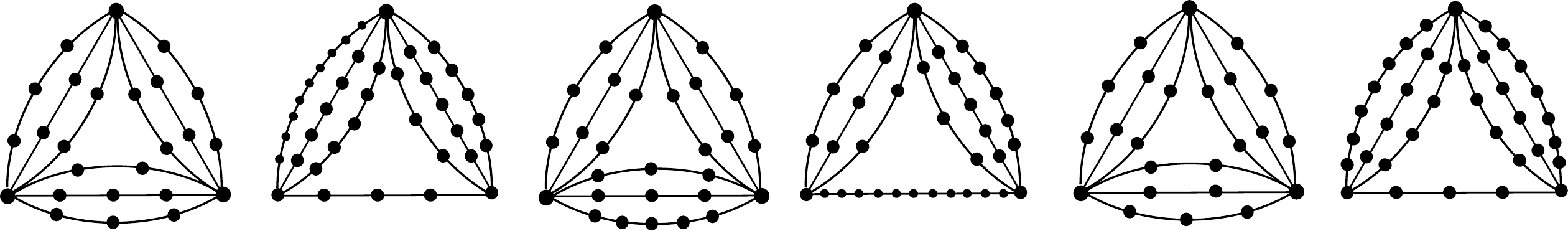}
\put(0,15){\footnotesize{$\G_1$}}
\put(16,15){\footnotesize{$\G_1'$}}
\put(33,15){\footnotesize{$\G_2$}}
\put(50,15){\footnotesize{$\G_2'$}}
\put(67,15){\footnotesize{$\G_3$}}
\put(85,15){\footnotesize{$\G_3'$}}
\put(0.5,10){\footnotesize{$u$}}
\put(12,10){\footnotesize{$u$}}
\put(6.5,-2){\footnotesize{$v$}}
\put(16.5,10){\footnotesize{$4u$}}
\put(30,10){\footnotesize{$2v$}}
\put(35,10){\footnotesize{$u$}}
\put(46.5,10){\footnotesize{$u$}}
\put(40,-2){\footnotesize{$2u$}}
\put(52,10){\footnotesize{$v$}}
\put(63.5,10){\footnotesize{$2v$}}
\put(69,10){\footnotesize{$u$}}
\put(81,10){\footnotesize{$u$}}
\put(75,-2){\footnotesize{$u$}}
\put(85,10){\footnotesize{$3u$}}
\put(98,10){\footnotesize{$3u$}}
\put(6.5,15){\footnotesize{$a$}}
\put(-1.5,1){\footnotesize{$b$}}
\put(14,0.5){\footnotesize{$c$}}
\put(24,15){\footnotesize{$a'$}}
\put(23,-1){\footnotesize{$\beta'$}}
\end{overpic}
\vspace{3mm}
\caption{\small We illustrate Remark~\ref{rem:Conditions}.  Let $u= (-\frac 14, -\frac {1}4, -\frac {1}2)$ and $v = (-\frac {1}4, -\frac {1}2,- \frac {1}2)$. 
By Definition~\ref{def:euler}, the Euler characteristic vectors associated to the nontrivial generalized $\Theta$-subgraphs of the  graphs above are the multiples of $u$ and $v$ depicted.  
Consider first the pair $\G_1, \G_1'$.   Condition (1a) in Theorem~\ref{thm:CycleGenTheta} clearly holds.   Now $W_A = \langle a, b, c\rangle <W_{\G_1}$ 
satisfies $\chi(W_A)= -\frac{1}2$, and $W_{A'}<W_{\G_1'}$ is the group defined by the branch $\beta'$ together with the vertex $a'$, so $\chi(W_{A'}) = -1$.  
It is easy to check that for the commensurability class of~$u$, both sides of the equation in (1b) are 
$2$ and for that of $v$, both sides of the equation in (1b) are $\frac54$.  Thus condition (1) is satisfied for $\G_1, \G_1'$.  However, since $u$ and $v$ are not commensurable, (2a) fails for this pair.  Now consider the graphs $\G_2, \G_2'$.  It is clear that condition (2a) holds, but since $u$ and $v$ are not commensurable, (1a) fails.  For this pair, we have $w = (-\frac 12, -1, -1, -2)$ and $w' = (-\frac 34, -\frac 32, -\frac 32, -3)$, so (2b) holds as well.  Thus condition (2) but not condition~(1) is satisfied for $\G_2, \G_2'$.  Finally, we leave it to the reader to check that both (1) and (2) hold for the pair $\G_3, \G_3'$.  
}
\label{fig:Conditions}
\end{figure}
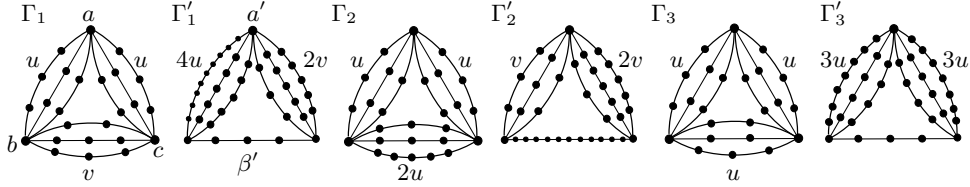
\end{center}
\end{remark}

\begin{remark}\label{rem:GenThetaCycleGenTheta}
If $\G$ is a generalized $\Theta$-graph as in the statement of Theorem~\ref{thm:GenTheta}, then the result of doubling $\G$ along a vertex of valence 2 which is adjacent to a vertex of valence at least $3$ is a cycle of three generalized $\Theta$-graphs, and by doubling again if necessary, we can obtain a $3$-convex cycle of generalized $\Theta$-graphs $\G'$ such that the groups $W_\G$ and $W_{\G'}$ are commensurable (compare Remark~\ref{rem:doubling}).  Hence we can determine which groups from Theorems~\ref{thm:GenTheta} and~\ref{thm:CycleGenTheta} are commensurable to each other.
\end{remark}

We then turn our attention to the relationship between right-angled Coxeter groups and geometric amalgams of free groups.  We prove the following two theorems.  Recall that $\cG$ is the class of graphs satisfying Assumptions~\ref{assumptions} and~\ref{K4assumption}.

\begin{thm}\label{thm:Degree16} If $\G \in \cG$ and $\G$ is $3$-convex, then $W_\G$ has an index 16 subgroup which is a geometric amalgam of free groups.
\end{thm}

\begin{thm}\label{thm:GeomAmalgamsRACGs}  If a geometric amalgam of free groups has JSJ graph which is a tree, then it is abstractly commensurable to a right-angled Coxeter group (with defining graph in $\cG$). \end{thm}

In addition, we show that if the diameter of the JSJ graph of the geometric amalgam of free groups is at most 4, then the defining graph of the corresponding right-angled Coxeter group guaranteed by 
Theorem~\ref{thm:GeomAmalgamsRACGs} is either a generalized $\Theta$-graph or a cycle of generalized $\Theta$-graphs.   Thus we have the following corollary.

\begin{cor}\label{cor:GeomAmalgamsCommens}  Geometric amalgams of free groups whose JSJ graphs are trees of diameter at most 4 can be classified up to abstract commensurability.
\end{cor}

\subsection*{Acknowledgements}

We thank the Forschungsinstitut f\"ur Mathematik at ETH Z\"urich, for hosting Thomas in Spring Semester 2016 and supporting a visit by Dani in June 2016, during which part of this research was carried out.   This research was also supported by the National Science Foundation under 
Grant No.~DMS-1440140 while Dani and Thomas were in residence at the Mathematical Sciences Research Institute (MSRI) in Berkeley, California during the Fall 2016 semester. We also thank MSRI for supporting a visit by Stark in Fall 2016. 
Dani was partially supported by NSF Grant No.~DMS-1207868 and this
work was supported by a grant from the Simons Foundation (\#426932, Pallavi Dani).  Thomas was partially supported by an Australian Postdoctoral Fellowship and this research was supported in part by Australian Research Council Grant 
No.~DP110100440.  Finally, we thank the anonymous referee for careful and swift reading, which included spotting a mistake in the proof of Theorem~\ref{thm:GeomAmalgamsRACGs} in an earlier version, and numerous other helpful suggestions.

\section{Preliminaries}\label{sec:background}

We recall relevant graph theory in Section~\ref{sec:graphs}.  Section~\ref{sec:jsj} states the results on JSJ decompositions from~\cite{dani-thomas-jsj} that we will need, and establishes some technical lemmas.  In Section~\ref{sec:lafont} we recall from~\cite{lafont} the definitions of geometric amalgams of free groups and surface amalgams, and state Lafont's topological rigidity result.  Section~\ref{sec:coverings} contains some well-known results on coverings of surfaces with boundary, and Section~\ref{sec:euler} recalls the definition of Euler characteristic for orbicomplexes.

\subsection{Graph theory}\label{sec:graphs}

In this paper, we work with several different kinds of graphs.
We now recall some graph-theoretic terminology and establish notation. 

We will mostly consider unoriented graphs, and so refer to these as just \emph{graphs}.  As in~\cite{behrstock-neumann}, a \emph{graph}~$\Lambda$ consists of a vertex set $V(\Lambda)$, an edge set $E(\Lambda)$, and a map $\epsilon:E(\Lambda) \to V(\Lambda)^2/C_2$ from the edge set to the set of unordered pairs of elements of $V(\Lambda)$.  For an edge $e$, we write $\epsilon(e) = [x,y] = [y,x]$, where $x, y \in V(\Lambda)$.  An edge $e$ is a \emph{loop} if $\epsilon(e) = [x,x]$ for some $x \in V(\Lambda)$.  If $\epsilon(e) = [x,y]$ we say that $e$ is \emph{incident} to $x$ and $y$, and when $x \neq y$, that $x$ and $y$ are \emph{adjacent} vertices.  The \emph{valence} of a vertex $x$ is the number of edges incident to $x$, counting $2$ for each loop $e$ with $\epsilon(e) = [x,x]$.  A vertex is \emph{essential} if it has valence at least $3$.  

Identifying $\Lambda$ with its realization as a $1$-dimensional cell complex, a \emph{cut pair} in $\Lambda$ is a pair of vertices $\{x,y\}$ so that $\Lambda \setminus \{x,y\}$ has at least two components, where a \emph{component} by definition contains at least one vertex.  A cut pair $\{x,y\}$ is \emph{essential} if $x$ and $y$ are both essential vertices.  A \emph{reduced path} in $\Lambda$ is a path which does not self-intersect.  A \emph{branch} of $\Lambda$ is a subgraph of $\Lambda$ consisting of a (closed) reduced path between a pair of essential vertices, which does not contain any essential vertices in its interior.

 A graph $\Lambda$ is \emph{bipartite} if $V(\Lambda)$ is the disjoint union of two nonempty subsets $V_1(\Lambda)$ and $V_2(\Lambda)$, such that every edge of $\Lambda$ is incident to exactly one element of $V_1(\Lambda)$ and exactly one element of $V_2(\Lambda)$.  In this case, we sometimes refer to the vertices in $V_1(\Lambda)$ as the \emph{Type~1} vertices and those in $V_2(\Lambda)$ as the \emph{Type~2} vertices.

An \emph{oriented graph} $\Lambda$ consists of a vertex set $V(\Lambda)$, an edge set $E(\Lambda)$, and maps $i:E(\Lambda) \to V(\Lambda)$ and $t:E(\Lambda) \to V(\Lambda)$.  For each edge $e \in E(\Lambda)$, we refer to $i(e)$ as the \emph{initial vertex} of~$e$ and $t(e)$ as the \emph{terminal vertex} of $e$.  Other definitions are similar to the unoriented case.

Throughout this paper, we reserve the notation $\G$ and $\G'$ for defining graphs of right-angled Coxeter groups, and we assume throughout that $\G$ and $\G'$ are finite, simplicial graphs in $\cG$, that is, they satisfy Assumptions~\ref{assumptions} and~\ref{K4assumption}.  (A graph $\G$ is \emph{simplicial} if it has no loops and the map $\epsilon$ is injective, that is, $\G$ does not have any multiple edges.)   

\subsection{JSJ decomposition of right-angled Coxeter groups}\label{sec:jsj}

In this section we recall the results we will need from~\cite{dani-thomas-jsj}.  We also establish some technical lemmas needed for our constructions in Section~\ref{sec:orbicomplex}, which use 
similar arguments to those in~\cite{dani-thomas-jsj}.

\subsubsection{JSJ decomposition}
Let $W = W_\G$.  The main result of~\cite{dani-thomas-jsj} gives an explicit description of the $W$-orbits in Bowditch's JSJ tree $\cT = \cT_{W_\G}$ and the stabilizers for this action.  From this, we can obtain an explicit description of the canonical graph of groups induced by the action of $W$ on $\cT$.  Recall from the introduction that the quotient graph $\Lambda = W \backslash \cT$ is the \emph{JSJ graph of $W$} and the canonical graph of groups over $\Lambda$ is the \emph{JSJ decomposition of $W$}.  The JSJ graph $\Lambda$ is in fact a tree, since $W$ is not an HNN extension (like any Coxeter group,~$W$ is generated by torsion elements hence does not surject to $\Z$). The next result follows from Theorem~3.36 of~\cite{dani-thomas-jsj}, and is illustrated by the examples in Figure~\ref{def_graph_JSJ}.

\begin{cor}\label{cor:jsj}  
Let $\G$ be a finite, simplicial graph satisfying Assumptions~\ref{assumptions} and~\ref{K4assumption}, so that $\G$ is $3$-convex.  The JSJ decomposition for $W = W_\G$ is as follows:
\begin{enumerate}
\item For each pair of essential vertices $\{a, b\}$ of $\G$ such that $\G \setminus \{a, b\}$ has $k \ge 3$ components, the JSJ graph $\Lambda$ has a vertex of valence $k$, with local group $\langle a, b \rangle$.
\item For each set $A$ of vertices of $\G$ satisfying the following conditions:
\begin{enumerate}
\item[\emph{($\alpha_1$)}] elements of $A$ pairwise separate the geometric realization $|\G|$ of $\G$, that is, given $a, b \in A$ with $a \neq b$, the space $|\G| \setminus \{a, b\}$ has at least two components; 
\item[\emph{($\alpha_2$)}] the set $A$ is maximal among all sets satisfying \emph{($\alpha_1$)}; and
\item[\emph{($\alpha_3$)}] $\langle A \rangle$ is infinite but not $2$-ended;
\end{enumerate}
the JSJ graph $\Lambda$ has a vertex of valence $\ell$ where $\ell \geq 1$ is the number of distinct pairs of essential vertices in $A$ which are as in (1).  The local group at this vertex is $\langle A \rangle$.  
\item Every edge of $\Lambda$ connects some vertex as in (1) above to some vertex as in (2) above.  Moreover, 
a vertex $v_1$ as in (1) and a vertex $v_2$ as in (2) are adjacent if and only if their local groups intersect, with this intersection necessarily $\langle a, b \rangle $ where $\{a,b\}$ are as in (1).  There will then be an edge with local group $\langle a, b \rangle$ between these vertices.  All maps from edge groups to vertex groups are inclusions.
\end{enumerate}
\end{cor}

For $i = 1,2$, we refer to the vertices of the JSJ graph $\Lambda$ as in part (i) of Corollary~\ref{cor:jsj} as the \emph{Type i} vertices, and denote these by $V_i(\Lambda)$.  The JSJ graph $\Lambda$ is a bipartite graph, with vertex set $V(\Lambda) = V_1(\Lambda) \sqcup V_2(\Lambda)$.

\begin{center}
\begin{figure}
\begin{overpic}[width=125mm]
{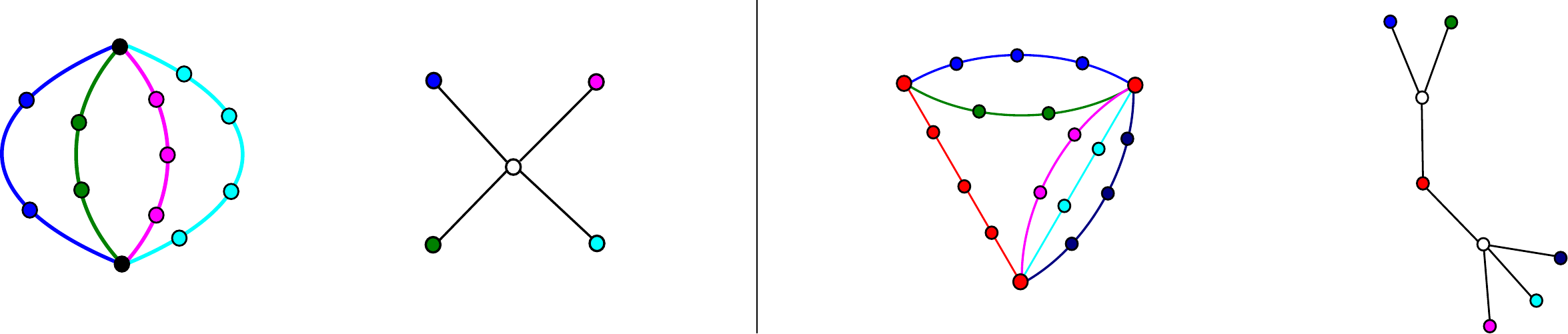}
\put(-7,10){$\Gamma = $}
\put(7,20){$a$}
\put(7,1){$b$}
\put(0,16){\footnotesize{\textcolor{blue}{$s_1$}}}
\put(0,6){\footnotesize{\textcolor{blue}{$s_2$}}}
\put(2.5,14){\footnotesize{\textcolor{ForestGreen}{$t_1$}}}
\put(3,8){\footnotesize{\textcolor{ForestGreen}{$t_2$}}}
\put(7,14){\footnotesize{\textcolor{magenta}{$u_1$}}}
\put(7.5,11){\footnotesize{\textcolor{magenta}{$u_2$}}}
\put(6.5,8){\footnotesize{\textcolor{magenta}{$u_3$}}}
\put(13,17){\footnotesize{\textcolor{cyan}{$v_1$}}}
\put(16,14){\footnotesize{\textcolor{cyan}{$v_2$}}}
\put(16,8){\footnotesize{\textcolor{cyan}{$v_3$}}}
\put(13,5){\footnotesize{\textcolor{cyan}{$v_4$}}}
\put(26,10){\footnotesize{$\langle a, b \rangle$}}
\put(35,18){\footnotesize{$\langle a, b, u_i \rangle$}}
\put(35,3){\footnotesize{$\langle a, b, v_i \rangle$}}
\put(22,18){\footnotesize{$\langle a, b, s_i \rangle$}}
\put(22,3){\footnotesize{$\langle a, b, t_i \rangle$}}
\put(50,10){$\Gamma' = $}
\put(55,17){\footnotesize{\textcolor{red}{$a_1$}}}
\put(72,17){\footnotesize{\textcolor{red}{$a_2$}}}
\put(64,1){\footnotesize{\textcolor{red}{$a_3$}}}
\put(61,5){\footnotesize{\textcolor{red}{$a_4$}}}
\put(58.5,8.5){\footnotesize{\textcolor{red}{$a_5$}}}
\put(56.5,12){\footnotesize{\textcolor{red}{$a_6$}}}
\put(86,9){\footnotesize{\textcolor{red}{$\langle a_i \rangle$}}}
\put(82,15){\footnotesize{\textcolor{black}{$\langle a_1, a_2\rangle$}}}
\put(86,4){\footnotesize{\textcolor{black}{$\langle a_2, a_3\rangle$}}}
\end{overpic}
\caption{\small{Examples of JSJ decompositions given by Corollary~\ref{cor:jsj}.  From left to right, we have a defining graph $\G$, the JSJ decomposition of $W_\G$, a defining graph~$\G'$, and the JSJ decomposition of $W_{\G'}$.  In the JSJ decompositions, the Type~1 vertices are white, all unlabelled Type 2 vertex groups are the special subgroups corresponding to the branches indicated by color, and all edge groups are equal to the groups on the adjacent Type 1 vertices.}}
\label{def_graph_JSJ} 
\end{figure}
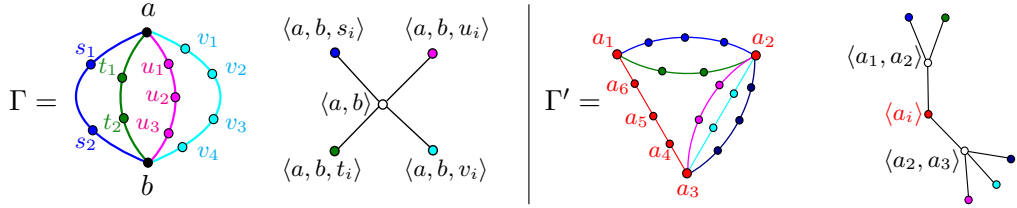
\end{center}

\subsubsection{Technical lemmas}

We will use in Section~\ref{sec:orbicomplex} the following technical lemmas related to the JSJ decomposition, where $\G$ is as in the statement of Corollary~\ref{cor:jsj}.

\begin{lemma}\label{lem:L2}  Suppose a Type~2 vertex in the JSJ decomposition of $W_\G$ has stabilizer $\langle A \rangle$.  
Let $L = L_A$ be the number of essential vertices in $A$.  Then $L \geq 2$, and $L = 2$ if and only if $A$ is equal to the set of vertices of a branch of $\G$ (including its end vertices).
\end{lemma}

\begin{proof}
The fact that $L \geq 2$ follows from the vertex $v$ having valence $\ell \geq 1$, and the description given in Corollary~\ref{cor:jsj} of the vertices adjacent to $v$.

If $A$ equals the set of vertices of a branch of $\G$, then it is clear that $L = 2$.  Conversely, if $L = 2$ let the two essential vertices of $A$ be $a$ and $b$.  Since $\la A \ra$ is not $2$-ended by ($\alpha_3$), the set $A$ must also contain a non-essential vertex, say $a'$.  Let $\beta$ be the branch of $\G$ containing~$a'$.  Using Lemma 3.20 of \cite{dani-thomas-jsj}, we obtain that $A$ contains all vertices of $\beta$, including its end vertices.  Since $L = 2$, these end vertices must be $a$ and $b$.  If $A$ contains a vertex $c$ of $\G$ which is not in the branch $\beta$, then since $L = 2$, the vertex $c$ must be non-essential and lie on another branch, say $\beta'$, between $a$ and $b$.  Now the graph $\G$ is not a cycle, so the pair $\{a',c\}$ cannot separate $| \G |$.  This contradicts ($\alpha_1$).  So if $L = 2$, the set $A$ is equal to the set of vertices of a branch $\beta$ in $\G$. 
\end{proof}

\begin{lemma}\label{lem:k3}  Suppose a Type~2 vertex in the JSJ decomposition of $W_\G$ has stabilizer $\langle A \rangle$.  Assume~$A$ contains the set of vertices of a branch $\beta$ of $\G$, including its end vertices $b$ and $b'$.  Let $k$ be the number of components of $\G \setminus \{b,b'\}$.  Then $k \geq 2$, and $k \geq 3$ if and only if $A$ equals the set of vertices of $\beta$.
\end{lemma}

\begin{proof}  First observe that $k \geq 2$ since the branch $\beta$ is not all of $\Gamma$.  Now assume $k \geq 3$, and let $\Lambda_1,\dots,\Lambda_k$ be the components of $\G \setminus \{b,b'\}$, with $\Lambda_1 = \beta \setminus \{b,b'\}$.  Suppose $A$ contains a vertex $a$ which is not on $\beta$.  Without loss of generality, $a \in \Lambda_2$.  Let $c$ be an interior vertex of $\beta$.  We will show that $\G \setminus \{a,c\}$ is connected, hence the (non-adjacent) pair $\{a,c\}$ does not separate $|\G|$.  This contradicts condition ($\alpha_1$) for $A$.  

Define $\Lambda$ to be the induced subgraph of $\G$ which is the union of $b$, $b'$, and the graphs $\Lambda_i$ for $3 \leq i \leq k$.  Since $k \geq 3$, the graph $\Lambda$ is connected, and since $c \in \Lambda_1$ and $a \in \Lambda_2$, we have that~$\Lambda$ is contained in one component of $\G \setminus \{a,c\}$.  It now suffices to show that for any vertex $d \in \Lambda_1 \cup \Lambda_2$ with $d \not \in \{a,c\}$, there is a path in $\G$ from $d$ to either $b$ or $b'$ which misses both $a$ and~$c$.  If $d \in \Lambda_1 \setminus \{a,c\} = \Lambda_1 \setminus \{c\}$, then $d$ is an interior vertex of the branch $\beta$ and the result is clear.  If $d \in \Lambda_2 \setminus \{a,c\} = \Lambda_2 \setminus \{a\}$, consider the induced graph $\overline\Lambda_2$ with vertices $b$, $b'$, and the vertex set of~$\Lambda_2$. We claim that $d$ must be in a component of $\overline\Lambda_2 \setminus\{a\}$ containing either $b$ or $b'$, hence there is a path in $\overline\Lambda_2$ from $d$ to either $b$ or $b'$ which misses $a$ (and also $c$), and we are done.  

Suppose for a contradiction that $d$ is in a component of $\overline\Lambda_2 \setminus\{a\}$ containing neither $b$ nor~$b'$.  By Assumption~\ref{assumptions}, the graph $\G$ has no separating vertices, so there is a reduced path $\eta$ in~$\G$ from $d$ to say $b$ which does not pass through $a$.  Without loss of generality, we may assume that $\eta$ does not pass through $b'$.  Then the entire path $\eta$ is contained in a component of $\overline\Lambda_2 \setminus \{a\}$.  So in fact $d$ is in a component of $\overline\Lambda_2 \setminus\{a\}$ containing $b$, and the claim holds.  Thus $\G \setminus \{a,c\}$ is connected.  We conclude that if $k \geq 3$, the set $A$ equals the vertex set of $\beta$.

Now suppose that $A$ equals the vertex set of $\beta$, and let $\ell$ be valence of the vertex of the JSJ graph which has stabilizer $\langle A \rangle$.  Then by (2) of Corollary~\ref{cor:jsj}, we have that $\ell \geq 1$ and that $A$ must contain at least one pair of vertices as in (1) of Corollary~\ref{cor:jsj}.  Since $A$ equals the vertex set of the branch $\beta$, the set $A$ contains at most one pair of vertices as in (1) of Corollary~\ref{cor:jsj}, namely the pair $\{b,b'\}$.  By the description given in (1) of Corollary~\ref{cor:jsj}, it follows that $k \ge 3$.
\end{proof}

\begin{lemma}\label{lem:k23}  Suppose a Type~2 vertex in the JSJ decomposition of $W_\G$ has stabilizer $\langle A \rangle$.
\begin{enumerate}
\item The set $A$ has a well-defined cyclic ordering $a_1,\dots,a_n$. 
\item Assume $A$ contains $L \geq 3$ essential vertices, and let $a_{i_1}, \dots, a_{i_L}$ be the essential vertices of $A$ in the induced  cyclic ordering.  For $1 \leq j \leq L$ let $k_j$ be the number of components of $\G \setminus \{a_{i_j},a_{i_{j+1}}\}$. 
\begin{enumerate}
\item For all $1 \leq j \leq L$, $k_j \geq 2$, hence $\{a_{i_j},a_{i_{j+1}}\}$ is an essential cut pair.
\item For each $1 \leq j \leq L$, there is at least one branch $\beta$ of $\G$ between $a_{i_j}$ and $a_{i_{j+1}}$.
\item Suppose $a \in A$ is a non-essential vertex, lying between $a_{i_j}$ and $a_{i_{j+1}}$ in the cyclic ordering.  Then $k_j = 2$, there is a unique branch $\beta = \beta_j$ of $\G$ between $a_{i_j}$ and~$a_{i_{j+1}}$, all vertices of $\beta$ are contained in~$A$, and $\beta$ contains the vertex $a$.
\item If there are no non-essential vertices of $A$ lying between $a_{i_j}$ and $a_{i_{j+1}}$ in the cyclic ordering on $A$, then $k_j \geq 3$.
\item Suppose $a,b \in A$ are such that $\langle a, b \rangle$ is the stabilizer of a Type 1 vertex in the JSJ decomposition of $W_\G$.  Then $a$ and $b$ are adjacent in the cyclic ordering on~$A$.
\end{enumerate}
\end{enumerate}
\end{lemma}

\begin{proof}  Part (1) is Lemma 3.14(1) of \cite{dani-thomas-jsj}.  

For (2)(a), we have $k_j \geq 2$ for all $j$ by ($\alpha_1$) in Corollary~\ref{cor:jsj} and the graph $\G$ being $3$-convex.  

For (2)(b), let $\sigma$ be an induced cycle in $\G$ containing all vertices of $A$, as guaranteed by Lemma~3.12 of~\cite{dani-thomas-jsj}.   Then since $a_{i_j}$ and $a_{i_{j+1}}$ are adjacent essential vertices in the cyclic ordering on $A$, there is an arc of $\sigma \setminus \{a_{i_j},a_{i_{j+1}}\}$ which contains no essential vertices of $A$.  To see that there is a branch $\beta$ between $a_{i_j}$ and $a_{i_{j+1}}$, it suffices to show that this arc contains no essential vertices of $\G$.  Assume there is an essential vertex $b$ of $\G$ which lies on $\sigma$ between $a_{i_j}$ and $a_{i_{j+1}}$.  Then it is not hard to see that either every reduced path from $a_{i_j}$ to $a_{i_{j+1}}$ in the component of $\G \setminus \{a_{i_j},a_{i_{j+1}}\}$ containing $b$ passes through $b$, or $\G$ contains a subdivided $K_4$ subgraph.  In the first case, by the maximality condition ($\alpha_2$) of Corollary~\ref{cor:jsj}, the vertex $b$ is in $A$, which is a contradiction.  The second case contradicts Assumption~\ref{K4assumption}.  Hence there is at least one branch $\beta$ of $\G$ between $a_{i_j}$ and $a_{i_{j+1}}$.

To prove (2)(c), let $a$ be a non-essential vertex of $A$ lying between $a_{i_j}$ and $a_{i_{j+1}}$ in the cyclic ordering, and let $\beta$ be the branch of $\G$ containing $a$. (Note that every non-essential vertex of $\G$ lies on a unique branch of $\G$.) Then all vertices of $\beta$ are in $A$, by Lemma~3.20 of~\cite{dani-thomas-jsj}.  Since the cyclic ordering on $A$ is well-defined, and $\beta$ contains $a$, it follows that the branch $\beta$ has endpoints $a_{i_j}$ and $a_{i_{j+1}}$.  Now as $L \geq 3$, the set $A$ is not equal to the vertex set of $\beta$.  Thus by Lemma~\ref{lem:k3} we get that $k_j = 2$.  So there is at most one branch of $\G$ between $a_{i_j}$ and $a_{i_{j+1}}$.  Hence $\beta = \beta_j$ is the unique branch of $\G$ between $a_{i_j}$ and $a_{i_{j+1}}$.  We have proved all claims in (2)(c).

For (2)(d), since there are no non-essential vertices of $A$ lying between $a_{i_j}$ and $a_{i_{j+1}}$ in the cyclic ordering on $A$, without loss of generality $a_{i_j} = a_1$ and $a_{i_{j+1}} = a_2$, and $\G \setminus \{ a_1, a_2\}$ has $k_j = k \geq 2$ components.  By (2)(b), there is a branch $\beta$ of $\G$ between $a_1$ and $a_2$.  Assume $k = 2$.  We will obtain a contradiction by showing that every interior vertex of $\beta$ is in $A$.  Let $b$ be an interior vertex of $\beta$.   First notice that the pairs $\{a_1, b \}$ and $\{ a_2, b\}$ both separate $|\G|$, and that if $b'$ is any other interior vertex of $\beta$ then $\{b,b'\}$ also separates $|\G|$.  Now since $k = 2$, for every $a \in A$ which is not in $\beta$, the pair $\{a, b\}$ separates $|\G|$ if and only if both $\{a, a_1\}$ and $\{a,a_2\}$ separates $|\G|$.  There is at least one such $a$ since $L \geq 3$.  It follows that for all $a \in A$ and for all vertices $b$ in the interior of $\beta$, the pair $\{a,b\}$ separates $|\G|$.  Hence by the maximality condition ($\alpha_2$), every interior vertex of $\beta$ is in $A$.  This contradicts $a_1$ and $a_2$ being adjacent in the cyclic ordering on $A$, so $k_j \geq 3$ as required.

To prove (2)(e), we have by Corollary~\ref{cor:jsj} that $\G \setminus \{a,b\}$ has $k \geq 3$ components.  Let $\sigma$ be an induced cycle in $\G$ containing all vertices of $A$, as guaranteed by Lemma 3.12 of~\cite{dani-thomas-jsj}, and suppose $a$ and $b$ are not consecutive in the cyclic ordering on $A$.  Then there are vertices $c, d \in A$ so that one of the arcs of $\sigma$ from $a$ to $b$ contains $c$ and the other arc contains $d$.  Since $k \geq 3$, there is also a reduced path from $a$ to $b$ which misses both $c$ and $d$.  But this contradicts Lemma 3.14(2) of~\cite{dani-thomas-jsj}.  Hence $a$ and $b$ are consecutive in the cyclic ordering on $A$, as required.  
\end{proof}

\begin{lemma}\label{lem:not_cycle}  If a Type~2 vertex in the JSJ decomposition of $W_\G$ has stabilizer $\langle A \rangle$, then the set $A$ is not equal to the vertex set of an induced cycle in $\G$.
\end{lemma}

\begin{proof}  If $A$ is equal to the vertex set of an induced cycle in $\G$, then the group $\langle A \rangle$ is cocompact Fuchsian, which contradicts the characterization of stabilizers of Type 2 vertices in~\cite{bowditch}. 
\end{proof}

\begin{remark}\label{rmk:TreeJSJ}  Suppose $\Lambda$ is the JSJ graph of a right-angled Coxeter group $W_\G$ as in the statement of Corollary~\ref{cor:jsj}.  Then $\Lambda$ is a finite tree (containing at least one edge).  By Corollary~\ref{cor:jsj}(1), all Type 1 vertices of $\Lambda$ have valence $\geq 3$.  Hence all valence one vertices of $\Lambda$ are of Type 2, so $\Lambda$ has even diameter.  Using Corollary~\ref{cor:jsj} and the above technical lemmas, it is not hard to check that $\Lambda$ has diameter 2 if and only if $\G$ is a generalized $\Theta$-graph, and that~$\Lambda$ has diameter $4$ if and only if $\G$ is a cycle of generalized $\Theta$-graphs (with $\G$ being $3$-convex and satisfying Assumptions~\ref{assumptions} and~\ref{K4assumption} in both cases).
\end{remark}

\subsection{Surface amalgams and topological rigidity}\label{sec:lafont}

We now recall some definitions and a topological rigidity result from Lafont~\cite{lafont}.

\begin{definition}[Surface amalgams and geometric amalgams of free groups] 
\label{def:surface-amalgam}
Consider a graph of spaces over an oriented graph $\Lambda$ with the following properties. 
\begin{enumerate}
\item The underlying graph $\Lambda$ is bipartite with vertex set $V(\Lambda) = V_1 \sqcup V_2$ and edge set $E(\Lambda)$ such that each edge $e \in E(\Lambda)$ has $i(e) \in V_1$ and  $t(e) \in V_2$. 
\item 
The vertex space $C_x$ associated to a vertex $x \in V_1$ is a copy of the circle $S^1$.  
The vertex space $S_y$ associated to a vertex $y \in V_2$ is a connected surface with negative Euler characteristic and nontrivial boundary.  

\item Given an edge $e \in E(\Lambda)$, the edge space $B_e$ is a copy of $S^1$.  The map 
$\phi_{e,i(e)}:B_e \to C_{i(e)}$ is a homeomorphism, and the map $\phi_{e,t(e)}:B_e \to S_{t(e)}$
is a homeomorphism onto a boundary component of $S_{t(e)}$.  

\item Each vertex $x \in V_1$ has valence at least 3.  
Given any vertex $y \in V_2$, for each boundary component $B$ of $S_y$, there exists an edge 
$e$ with $t(e) = y$, such that the associated edge map identifies $B_e$ with $B$.  The valence of $y$ is the number of boundary components of $S_y$.  
\end{enumerate}
A \emph{surface amalgam} $\cX = \cX(\Lambda)$ is 
the space
\newcommand{\bigslant}[2]{{\raisebox{.7em}{$#1$}\left/\raisebox{-.5em}{$#2$}\right.}}
\begin{equation*}
 \bigslant{\displaystyle\cX= \left(\bigsqcup_{x \in V_1} C_x \sqcup \bigsqcup_{y\in V_2} S_y \sqcup  \bigsqcup_{e\in E(\Lambda)} B_e\right)}{\sim}
\end{equation*}
where for each $e \in E(\Lambda)$ and each $b\in B_e$, we have $b \sim \phi_{e,i(e)}(b)$ and 
$b \sim \phi_{e,t(e)}(b)$. 
If $\cX$ is a surface amalgam, the \emph{surfaces in $\cX$} are the surfaces $S_y$ for $y \in V_2(\Lambda)$. 
The fundamental group of a surface amalgam is a \emph{geometric amalgam of free groups}. \end{definition}

\noindent Note that in~\cite{lafont}, surface amalgams are called \emph{simple, thick, $2$-dimensional hyperbolic $P$-manifolds}.

For an oriented graph $\Lambda$ as in Definition~\ref{def:surface-amalgam}, we may by abuse of notation write $\Lambda$ for the unoriented graph with the same vertex and edge sets and $\epsilon:E(\Lambda) \to V(\Lambda)/C_2$ given by $\epsilon(e) = [i(e),t(e)]$.  We note that:

\begin{remark}\label{rem:jsj}  If $\cX = \cX(\Lambda)$ is a surface amalgam, then the JSJ graph of the geometric amalgam of free groups $\pi_1(\cX)$ is the (unoriented) graph $\Lambda$. For details, see \cite[Section 4.1]{malone}.
\end{remark}

We will use the following topological rigidity result of Lafont. 

\begin{thm}\cite[Theorem 1.2]{lafont}\label{thm:Lafont}  Let $\cX$ and $\cX'$ be surface amalgams.  Then any isomorphism $\phi:\pi_1(\cX) \to \pi_1(\cX')$ is induced by a homeomorphism $f:\cX \to \cX'$.
\end{thm}

\subsection{Coverings of surfaces}\label{sec:coverings}

We now recall some results on coverings of surfaces.

We write $S_{g,b}$ for the connected, oriented surface of genus $g$ with $b$ boundary components.  This surface has Euler characteristic $\chi(S_{g,b}) = 2 - 2g - b$.  The next lemma allows us to obtain positive genus covers of any $S_{g,b}$ with negative Euler characteristic.

\begin{lemma}\label{lem:PosGenus}  Suppose $\chi(S_{g,b}) < 0$.  Then $S_{g,b}$ has a connected $3$-fold covering $S_{g',b}$, where $g' = 3g + b - 2$ and so $g' > 0$.
\end{lemma}

\begin{proof}  By Proposition 5.2 of Edmonds, Kulkarni, and Stong~\cite{edmonds-kulkarni-stong}, it is enough to check that $\chi(S_{g',b}) = 3\chi(S_{g,b})$.  This is an easy calculation.
Alternatively, as the referee suggested to us, one may explicitly construct a degree 3 covering map from $S_{1,3}$ to $S_{0,3}$, and then use a pants decomposition to deduce the result for other surfaces.  To construct such a map, note that a regular hexagon $H$ with opposite sides attached forms a torus, and the vertices of $H$ project to two points on the torus.  
Obtain $S_{1,3}$ by removing open balls around these two points as well as around the center of $H$.  Then the order~3 rotation of~$H$ induces an isometry of $S_{1,3}$ and the corresponding quotient is $S_{0,3}$. 
\end{proof}
We have the following easy corollary.
\begin{cor}\label{cor:PosGenus} 
If $\cX=\cX(\Lambda)$ is a surface amalgam, then $\cX$ has a degree 3 cover $\cX'$ which is a surface amalgam whose underlying graph is also $\Lambda$, so that each surface in $\cX'$ has positive genus.  \hfill{\qed}
\end{cor}

We will make repeated use of the following lemma concerning coverings of positive genus surfaces with boundary, from Neumann~\cite{neumann}.  As discussed in~\cite{neumann}, the result appears to be well-known.

\begin{lemma}\cite[Lemma 3.2]{neumann}\label{lem:neumann}  Let $S = S_{g,b}$ where $g > 0$ and $b > 0$.  Let $D$ be a positive integer.  Suppose that for each boundary component of $S$, a collection of degrees summing to $D$ is specified.  Then $S$ has a connected $D$-fold covering $S'$ with $b' \geq b$ boundary components and these specified degrees on the collection of boundary components of $S'$ lying over each boundary component of $S$ if and only if $b'$ has the same parity as $D \cdot\chi(S)$.
\end{lemma}

\subsection{Euler characteristic for orbicomplexes}\label{sec:euler}

We now recall the definition of orbicomplex Euler characteristic, in the special case that we will need.   

All of the orbicomplexes that we construct will be $2$-dimensional and have (possibly disconnected) underlying spaces obtained by gluing together some collection of right-angled hyperbolic polygons.  When the underlying space is a single polygon, the orbifold will be a reflection polygon.  We will then identify certain of these reflection polygons along non-reflection edges to obtain other orbicomplexes.  The local groups are as follows.  All edge groups will be either trivial or $C_2$, corresponding to non-reflection and reflection edges, respectively.  All vertex groups will be the direct product of the adjacent edge groups, and will be either $C_2$ or $C_2 \times C_2$.   

If $\cO$ is such an orbicomplex, write $F$ for the number of faces (i.e. polygons) in $\cO$, $E_1$ (respectively,~$E_2$) for the number of edges in $\cO$ with trivial (respectively, $C_2$) local groups, and $V_2$ (respectively, $V_4$) for the number of vertices in $\cO$ with $C_2$ (respectively, $C_2 \times C_2$) local groups.  Then
\[
\chi(\cO) = F - \left( E_1 + \frac{E_2}{2} \right) + \left( \frac{V_2}{2} + \frac{V_4}{4}\right).
\]

The fundamental groups of the connected orbicomplexes that we construct will be right-angled Coxeter groups with triangle-free defining graphs, such that if $W_\G = \pi_1(\cO)$ then the Euler characteristic of~$W_\G$ from Definition~\ref{def:euler} is equal to $\chi(\cO)$.

\section{Orbicomplex construction}\label{sec:orbicomplex}

From now on, $\G$ is a finite, simplicial, $3$-convex graph satisfying Assumptions~\ref{assumptions} and~\ref{K4assumption}.   In this section we construct a piecewise-hyperbolic orbicomplex $\cO_\G$ with fundamental group $W_\G$, such that covers of $\cO_\G$ corresponding to torsion-free, finite-index subgroups of $W_\G$ are surface amalgams.  The bottom right of Figure \ref{GenTheta} gives an example  of the orbicomplex $\cO_\G$ when $\Gamma$ is a generalized $\Theta$-graph, and Figure \ref{CycleGenTheta} contains an example of $\cO_\G$ when $\Gamma$ is a cycle of generalized $\Theta$-graphs.  We begin by constructing hyperbolic orbifolds in Sections~\ref{sec:Pbeta} and~\ref{sec:A} which have fundamental groups the stabilizers of Type 2 vertices in the JSJ decomposition of $W_\G$ (see Corollary~\ref{cor:jsj}).  The underlying spaces of these orbifolds are right-angled hyperbolic polygons.  We then in Section~\ref{sec:OGamma} glue these orbifolds together along their non-reflection edges to obtain $\cO_\G$.  Some features of the orbifolds constructed in this section are summarized in the first two columns of Table~\ref{table:summary} at the end of Section~\ref{sec:tfcover}.  

\subsection{Branch orbifolds}\label{sec:Pbeta}

For each branch $\beta$ in $\G$, we construct an orbifold $\cP_\beta$ with fundamental group the special subgroup $W_\beta$.  We call $\cP_\beta$ a \emph{branch orbifold}.  We will also assign types to some of the edges and vertices of $\cP_\beta$, which will later be used to glue $\cP_\beta$ to other orbifolds.

Let $\beta$ be a branch with $n = n_\beta$ vertices (including its endpoints).  Since $\G$ is $3$-convex, we have $n \geq 4$.  Let $P = P_\beta$ be a right-angled hyperbolic $p$-gon where $p = n + 1 \geq 5$.  We construct $P$ to have one edge of length $1$.  

Now we construct the orbifold $\cP_\beta$ over $P$.  The distinguished edge of $P$ with length 1 is a non-reflection edge of $\cP_\beta$.  The other $ n = p-1$ edges of $P$ are reflection edges of $\cP_\beta$, with local groups $\langle b_1 \rangle, \dots, \langle b_n \rangle$ in that order, where $b_1, \dots,b_n$ are the vertices in $\beta$ going in order along the branch.  For the vertex groups of $\cP_\beta$, the endpoints of the unique non-reflection edge of $\cP_\beta$ have groups $\langle b_1 \rangle$ and $\langle b_{n}\rangle$, so that for $i = 1$ and $i = n$ the vertex group $\langle b_i \rangle$ is adjacent to the edge group $\langle b_i \rangle$.  The other $n - 1 = p-2$ vertex groups of $\cP_\beta$ are $\langle b_i, b_{i+1} \rangle \cong C_2 \times C_2$ for $1 \leq i < n$, with $\langle b_i, b_{i+1} \rangle$ the local group at the vertex of $\cP_\beta$ whose adjacent edges have local groups $\langle b_i \rangle$ and $\langle b_{i+1} \rangle$.  

An easy calculation shows that the Euler characteristic of $\cP_\beta$ is $\chi(\cP_\beta) = \frac{3-n}{4} = \frac{4-p}{4} = \chi(W_\beta)$.  In fact, we have: 

\begin{lemma}\label{lem:fundgpbeta}  The fundamental group of $\cP_\beta$ is $W_\beta$.
\end{lemma}

\begin{proof}  We regard $\cP_\beta$ as a simple polygon of groups over the underlying polygon $P$, with trivial face group, trivial group on the non-reflection edge, and the other edge and vertex groups as described above (see~\cite[Example 12.17(6)]{bridson-haefliger} for the general definition of a simple polygon of groups).  Notice that each vertex group is generated by its adjacent edge groups.  Since the face group is trivial, it follows that the fundamental group of $\cP_\beta$ is generated by its edge groups, subject to the relations imposed within its vertex groups (compare~\cite[Definition 12.12]{bridson-haefliger}).  Hence by construction, $\pi_1(\cP_\beta)$ is generated by the vertices of the branch~$\beta$, which are $b_1,\dots,b_n$ going in order along the branch, with relations $b_i^2 = 1$ for $1 \leq i \leq n$, and $[b_i,b_{i+1}] = 1$ for $1 \leq i < n$.  That is, $\pi_1(\cP_\beta)$ is the special subgroup $W_\beta$ generated by the vertex set of $\beta$.

Alternatively, as suggested to us by the referee, consider the hyperbolic orbifold whose underlying space is $S_{0,1}$, the sphere with one boundary component, with $n - 1 = p - 2 \geq 3$ cone points of order 2.  Its fundamental group is the free product of $n-1$ copies of $C_2$.  Then $\cP_\beta$ is obtained by quotienting this orbifold by a reflection in a properly embedded segment containing all the cone points.  An easy computation shows that the orbifold fundamental group of $\cP_\beta$ is $W_\beta$, as desired.
\end{proof}

We assign the non-reflection edge of $\cP_\beta$ to have type $\{b_1, b_{n}\}$.  Note that since $b_1$ and $b_n$ are the end vertices of a branch in $\G$, the pair $\{b_1, b_n\}$ is an essential cut pair in $\G$.  For $i = 1$ and $i = n$, we assign type $\{b_i\}$ to the vertex of $\cP_\beta$ which has group $\la b_i \ra$.

\subsection{Essential vertex orbifolds and non-branch orbifolds}\label{sec:A}

Now let $A$ be a subset of vertices of $\G$ so that $W_A = \langle A \rangle$ is the stabilizer of a Type 2 vertex in the JSJ decomposition, and let $L$ be the number of essential vertices of $A$.  Recall from Lemma~\ref{lem:L2} that $L \geq 2$.  In this section, we assume that $L \geq 3$ and construct two hyperbolic orbifolds, $\cQ_A$ and $\cA$.  

By Lemma~\ref{lem:L2}, since $L \geq 3$ the set $A$ is not equal to the set of vertices of a branch of $\Gamma$.  However the set $A$ may still contain the vertex sets of branches.  The orbifold $\cA$ will be constructed in two stages: we first construct the orbifold $\cQ_A$ over a $2L$-gon, and then obtain $\cA$ by gluing on the branch orbifold $\cP_\beta$ for each branch $\beta$ whose vertex set is contained in $A$.  The fundamental group of $\cQ_A$ will be the special subgroup generated by the essential vertices in $A$, and $\cA$ will have fundamental group $W_A$.  
We refer to $\cQ_A$ as an \emph{essential vertex orbifold} and to $\cA$ as a \emph{non-branch orbifold}.

We now construct $\cQ_A$.  Let $Q = Q_A$ be a right-angled hyperbolic $2L$-gon ($L \geq 3$).  Using the following lemma, we can specify that alternate edges of $Q$ have length $1$.  

\begin{lemma}
Given $L\ge 3$, there exists a right-angled hyperbolic $2L$-gon in which alternate edges have length 
$1$.   
\end{lemma}

\begin{proof}  Given any three positive numbers, there exists a right-angled hyperbolic hexagon with alternate edges having lengths equal to these three numbers (cf.~Proposition B.4.13 of~\cite{benedetti-petronio}.)  The result then follows by induction on $L$, since if a right-angled hyperbolic hexagon with alternate edges of length 1 is glued along an edge of length 1 to a right-angled hyperbolic $2L$-gon with alternate edges of length 1, the result is a right-angled hyperbolic $2(L+1)$-gon with alternate edges of length 1.  
\end{proof}

The essential vertex orbifold $\cQ_A$ is constructed over $Q$.  The alternate edges of $Q$ of length~$1$ will be non-reflection edges of $\cQ_A$.  The remaining edges of $\cQ_A$ will be reflection edges, as follows.  By Lemma~\ref{lem:k23}(1), the set $A$ has a well-defined cyclic ordering $a_1,\dots,a_n$.  Let $a_{i_1}, \dots, a_{i_L}$ be the essential vertices of $A$ in this induced cyclic order.  The reflection edges of $\cQ_A$ will have groups $\langle a_{i_j} \rangle  \cong C_2$ for $1 \leq j \leq L$, going in order around $Q$.  Now each vertex of $\cQ_A$ is adjacent to one non-reflection edge and one reflection edge with group $\la a_{i_j} \ra$, and this vertex will also have group~$\la a_{i_j} \ra$.  

An easy calculation shows that $\chi(\cQ_A) = \frac{2-L}{2}$.  This is equal to Euler characteristic of the special subgroup generated by the $L$ essential vertices in $A$, since as $\G$ is $3$-convex, there are no edges between any two essential  vertices.

We next assign types to certain edges and all vertices of $\cQ_A$.  The non-reflection edges going around $\cQ_A$ are  assigned type $\{a_{i_j},a_{i_{j+1}}\}$ (where $j \in \Z/L\Z$), and we assign type $\{a_{i_j}\}$ to the vertices of $\cQ_A$ with group $\la a_{i_j} \ra$, so that the endpoints of the non-reflection edge of $\cQ_A$ with type $\{a_{i_j}, a_{i_{j+1}}\}$ have types $\{a_{i_j}\}$ and $\{a_{i_{j+1}}\}$.  Notice that each pair $\{a_{i_j},a_{i_{j+1}}\}$ is an essential cut pair, by Lemma~\ref{lem:k23}(2)(a).

We now construct the non-branch orbifold $\cA$.  If $A$ consists entirely of essential vertices, then we put $\cA = \cQ_A$.  Otherwise, by Lemma~\ref{lem:k23}(2)(c), there is at least one pair $\{a_{i_j},a_{i_{j+1}}\}$ of essential vertices in $A$ so that  $\G \setminus \{a_{i_j},a_{i_{j+1}}\}$ has $k_j = 2$ components, and for all such~$j$, there is a unique branch $\beta_j$ of $\G$ between $a_{i_j}$ and $a_{i_{j+1}}$.  Denote by $P_j$ the right-angled polygon underlying the branch orbifold $\cP_{\beta_j}$ constructed in Section~\ref{sec:Pbeta} above.  Recall that all non-reflection edges in $\cQ_A$ have length~1, and that the unique non-reflection edge in $\cP_{\beta_j}$ has length~1 as well.  Also, the non-reflection edge of $\cP_{\beta_j}$ has type $\{a_{i_j},a_{i_{j+1}}\}$.  

To obtain $\cA$, for each $j$ such that $A$ contains the vertex set of $\beta_j$, we glue the non-reflection edge of $\cQ_A$ of type $\{a_{i_j},a_{i_{j+1}}\}$ to the unique non-reflection edge of $\cP_{\beta_j}$, so that the types of the end-vertices match up.  In both $\cQ_A$ and $\cP_{\beta_j}$, the end-vertices of the edge which has just been glued have groups $\langle a_{i_j} \rangle$ and $\langle a_{i_{j+1}}\rangle$.  Also, the edge groups adjacent to the vertex group $\langle a_{i_j} \rangle$ (respectively, $\langle a_{i_{j+1}}\rangle$) are both $\langle a_{i_j} \rangle$ (respectively, $\langle a_{i_{j+1}}\rangle$).  So we may erase all of the non-reflection edges along which we just glued branch orbifolds to $\cQ_A$, and combine the edges of $\cQ_A$ and $\cP_{\beta_j}$ with group $\langle a_{i_j} \rangle$ (respectively, $\langle a_{i_{j+1}} \rangle$) into a single edge with group $\langle a_{i_j} \rangle$ (respectively, $\langle a_{i_{j+1}} \rangle$).  We now define $\cA$ to be the resulting orbifold over the right-angled hyperbolic polygon obtained by gluing together the polygon $Q_A$ which underlies $\cQ_A$ and the polygons $P_j$ which underly the $\cP_{\beta_j}$.  

\begin{lemma}\label{lem:fundgpA} The fundamental group of $\cA$ is $W_A$.
\end{lemma}

\begin{proof}  By construction, the reflection edges of $\cA$ have groups $\langle a_1\rangle$, \dots, $\langle a_n \rangle$, where $a_1,\dots,a_n$ is the cyclic ordering on $A$ given by Lemma~\ref{lem:k23}(1), and the reflection edges with groups $\langle a_i \rangle$ and $\langle a_{i+1}\rangle$ are adjacent in $\cA$ if and only if $a_i$ and $a_{i+1}$ are adjacent in $\G$ (for $i \in \Z / n\Z$).  The proof then uses similar arguments to Lemma~\ref{lem:fundgpbeta}.  
\end{proof}

For Euler characteristics, a somewhat involved calculation shows that $\chi(\cA) = W_A$.  We remark that there is no simple formula for this Euler characteristic in terms of $\ell$, the valence of the Type~2 vertex stabilized by $W_A$ in the JSJ decomposition, since $\chi(W_A)$ depends on the number of edges between vertices in $A$, and this varies independently of $\ell$.

Observe that since $A$ is not the vertex set of an induced cycle in $\G$ (by Lemma~\ref{lem:not_cycle}),  it follows from our construction that $\cA$ has at least one non-reflection edge.  The non-reflection edges of $\cA$ retain their types $\{a_{i_j},a_{i_{j+1}}\}$ from $\cQ_A$, as do the endpoints of such edges. 

\subsection{Construction of orbicomplex}\label{sec:OGamma}

We now construct the orbicomplex $\cO_\G$ by gluing together certain branch orbifolds $\cP_\beta$ from Section~\ref{sec:Pbeta} and all of the non-branch orbifolds $\cA$ from Section~\ref{sec:A}.  

Consider a Type 2 vertex in the JSJ decomposition with stabilizer $\langle A \rangle$, where $A$ is a set of vertices of $\G$.  Let $L$ be the number of essential vertices of $A$.   Then by Lemma~\ref{lem:L2}, we have that $L \geq 2$, and  $L = 2$ exactly when $A$ is the set of vertices of a branch $\beta$ of $\G$.  So if $L = 2$ then $W_A = W_\beta$ is the fundamental group of the branch orbifold $\cP_\beta$, and if $L \geq 3$ then $W_A$ is the fundamental group of the non-branch orbifold $\cA$. 

Let $\cC$ be the collection of all branch orbifolds $\cP_\beta$ such that $W_\beta = \pi_1(\cP_\beta)$ is a Type 2 vertex stabilizer, together with all non-branch orbifolds $\cA$.  By the discussion in the previous paragraph, we have:

\begin{cor}\label{cor:C}  The set of orbifolds $\cC$ is in bijection with the set of Type 2 vertices  in the JSJ decomposition of $W_\G$.  Moreover, for each orbifold $\cO$ in the collection $\cC$, we have that $\pi_1(\cO)$ is equal to the stabilizer of the corresponding Type 2 vertex.  
\end{cor}

We now consider the relationship between Type 1 vertices in the JSJ decomposition and non-reflection edges of orbifolds in the collection $\cC$.  Recall that each $\cP_\beta$ has a unique non-reflection edge, each $\cA$ has at least one non-reflection edge, and each non-reflection edge in either a $\cP_\beta$ or an~$\cA$ has type $\{a,b\}$ where $\{a,b\}$ is an essential cut pair of $\G$. 

\begin{lemma}\label{lem:ell}  Each non-reflection edge in the collection of orbifolds $\cC$ has type $\{a,b\}$ where $\langle a, b \rangle$ is the stabilizer of a Type 1 vertex in the JSJ decomposition of $W_\G$.  
\end{lemma}

\begin{proof}  First suppose that $\cP_\beta$ is in $\cC$, let the endpoints of the branch $\beta$ be $b$ and $b'$, and let $A$ be the vertex set of the branch $\beta$.  Then since $W_\beta = W_A$ is the stabilizer of a Type 2 vertex, Lemma~\ref{lem:k3} implies that $\G \setminus \{b,b'\}$ has $k \geq 3$ components.  The result then follows from Corollary~\ref{cor:jsj}(1).

Now let $\cA$ be a non-branch orbifold and let $\{a_{i_j},a_{i_{j+1}}\}$ be the type of a non-reflection edge of $\cA$.  Then by construction of $\cA$ and Lemma~\ref{lem:k23}(d), we have that $\G \setminus \{a_{i_j},a_{i_{j+1}}\}$ has $k_j \geq 3$ components (otherwise, we would have glued a branch orbifold on at this edge of $\cQ_A$).   The result then also follows from Corollary~\ref{cor:jsj}(1).  
\end{proof}

We note that, by similar arguments to those in Lemma~\ref{lem:ell}, the number of non-reflection edges of $\cA$ is equal to $\ell \geq 1$, the valence of the Type 2 vertex in the JSJ decomposition stabilized by $W_A$.

\begin{lemma}  Let $a$ and $b$ be vertices of $\G$ so that $\langle a, b \rangle$ is the stabilizer of a Type 1 vertex $v$ of valence $k \geq 3$ in the JSJ decomposition of $W_\G$.  Then $\{a,b\}$ is the type of a non-reflection edge in exactly $k$ orbifolds in the collection $\cC$.
\end{lemma}

\begin{proof}  Since $v$ has valence $k$, by the first statement in Corollary~\ref{cor:C} there are exactly $k$ orbifolds in the collection $\cC$ whose fundamental groups are stabilizers of Type 2 vertices adjacent to $v$.  By Corollary~\ref{cor:jsj}(3) and the second statement in Corollary~\ref{cor:C}, these $k$ orbifolds are exactly the elements of $\cC$ whose fundamental groups contain the generators $a$ and $b$.  If a branch orbifold $\cP_\beta$ is one of these $k$ orbifolds, then $a$ and $b$ are the endpoints of the branch $\beta$, so by construction the unique non-reflection edge of $\cP_\beta$ is of type $\{a,b\}$.   If a non-branch orbifold $\cA$ is one of these $k$ orbifolds, let $A$ be the set of vertices of $\G$ so that $W_A = \pi_1(\cA)$.  By the construction of $\cA$, it suffices to show that $a$ and $b$ are consecutive in the cyclic ordering on $A$ given by Lemma~\ref{lem:k23}(1), and this is Lemma~\ref{lem:k23}(2)(e).  
\end{proof}

Recall that each non-reflection edge in the orbifolds we have constructed has length $1$.  Now for each $\{a,b\}$ which is the type of some non-reflection edges in the collection $\cC$,  we glue together all non-reflection edges of type $\{a,b\}$ in this collection, so that the types of their end-vertices are preserved.  The resulting orbicomplex is $\cO_\G$.  

Note that in the resulting space, each non-reflection edge still has a well-defined type $\{a,b\}$, and is the unique non-reflection edge of this type.  Also, if a non-reflection edge $e$ of $\cO_\G$ has type $\{a,b\}$, and~$v$ is its vertex of type $\{a\}$ (respectively, $\{b\}$), then the vertex group of $\cO_\G$ at~$v$ is $\langle a \rangle$ (respectively, $\{b\}$), and each reflection edge of $\cO_\G$ which is adjacent to $v$ has group $\langle a \rangle$ (respectively,~$\{b\}$).  

\begin{lemma}  
The orbicomplex $\cO_\G$ has fundamental group $W_\G$.  
\end{lemma}

\begin{proof}  The underlying space of $\cO_\G$ is obtained by gluing together the polygons underlying the orbifolds in the collection $\cC$ along non-reflection edges.  After gluing, there is a polygon in the resulting space for each Type 2 vertex in the JSJ decomposition, and a single non-reflection edge contained in at least $3$ distinct polygons for each Type 1 vertex in the JSJ decomposition.  Since the JSJ decomposition is over a connected graph, it follows that the underlying space after gluing is connected.  Moreover, since the JSJ decomposition is over a tree, it follows that the underlying space after gluing is contractible.  Hence the fundamental group of $\cO_\G$ is generated by the fundamental groups of the orbifolds in $\cC$, subject to the identifications of generators induced by gluing non-reflection edges.  This gives fundamental group $W_\G$.
\end{proof}

We leave it to the reader to verify that $\chi(\cO_\G) = \chi(W_\G)$.

\section{Half-coverings and torsion-free covers}\label{sec:tfcover}

Let $\cO_\G$ be the orbicomplex with fundamental group $W_\G$ constructed in 
Section~\ref{sec:orbicomplex}.  In this section we 
construct a covering space $\cX$ of $\cO_\G$ so that $\pi_1(\cX)$ is an index 16 torsion-free subgroup of $\pi_1(\cO_\G) = W_\G$. Examples appear in 
Figures~\ref{GenTheta} and~\ref{CycleGenTheta}. The space $\cX$ will be a surface amalgam with each connected surface in $\cX$ having positive genus (so that we can obtain further covers by applying Lemma~\ref{lem:neumann}).  
We describe the construction using the terminology of half-coverings, which we define in 
Section~\ref{sec:half-coverings}.
The surfaces $S_\beta$, which cover the branch orbifolds $\cP_\beta$, and 
$S_\cA$, which cover the non-branch orbifolds $\cA$, are constructed in Sections~\ref{sec:S_beta}
and~\ref{sec:S_A} respectively.  
In Section~\ref{sec:X}, we explain how to glue the $S_\beta$ and $S_\cA$ together to obtain $\cX$.  Section~\ref{sec:SummaryRemarks} contains a table summarizing the surfaces we construct, and an explanation of why we use degree 16 covers.

\subsection{Half-coverings}\label{sec:half-coverings}

In this section we define half-coverings of general bipartite graphs, and a particular half-covering $\cH(T)$ where $T$ is a bipartite tree.

Given a graph $\Lambda$ which contains no loops, for each vertex $x \in V(\Lambda)$, let $\Lambda(x)$ the set of edges of $\Lambda$ which are incident to $x$. 
Now 
let $\Lambda$ and $\Lambda'$ be (unoriented) graphs, with associated maps $\epsilon:E(\Lambda) \to V(\Lambda)^2/C_2$ and $\epsilon':E(\Lambda') \to V(\Lambda')^2/C_2$.  (See 
Section~\ref{sec:graphs} for graph-theoretic definitions.)  Recall that a \emph{graph morphism} $\theta:\Lambda \to \Lambda'$ is a map taking $V(\Lambda)$ to $V(\Lambda')$ and $E(\Lambda)$ to $E(\Lambda')$, so that for all $e \in E(\Lambda)$, if $\epsilon(e) = [x,y]$ then $\epsilon'(\theta(e)) = [\theta(x),\theta(y)]$.

\begin{definition}\label{def:half-cover}  Let $\Lambda$ and $\Lambda'$ be bipartite graphs with vertex sets $V(\Lambda) = V_1 \sqcup V_2$ and $V(\Lambda') = V_1' \sqcup V_2'$, respectively.  A graph morphism $\theta:\Lambda \to \Lambda'$ is a \emph{half-covering} if:
\begin{enumerate}
\item For $i = 1,2$, the map $\theta$ takes $V_i$ to $V_i'$.
\item For all $x \in V_1$, the restriction of $\theta$ to $\Lambda(x)$ is a bijection onto $\Lambda'(\theta(x))$.
\item For all $y \in V_2$, and for every edge $e' \in \Lambda'(\theta(y))$, there is an $e \in \Lambda(y)$ so that $\theta(e) = e'$.
\end{enumerate}
If there is a half-covering $\theta:\Lambda \to \Lambda'$ then we say that $\Lambda$ \emph{half-covers} $\Lambda'$.
\end{definition}

In short, a half-covering is a morphism of bipartite graphs which preserves the bipartition, is locally bijective at vertices of Type 1, and is locally surjective at vertices of Type 2.  (In Section~\ref{sec:discussion}, we compare half-coverings with the \emph{weak coverings} in~\cite{behrstock-neumann}.)

In several of our constructions we will use half-coverings of the following particular form.  Examples appear in Figures~\ref{GenTheta} and~\ref{CycleGenTheta}.

\begin{definition}[The graph $\half(T)$ half-covering a tree $T$]\label{def:halfT}
Let $T$ be a bipartite tree with vertex set $V(T) = V_1 \sqcup V_2$.  Let $\half(T)$ be the bipartite graph defined as follows.  
\begin{enumerate}
\item The vertex set $V(\half(T))$ equals $V_1' \sqcup V_2'$, where $V_1'$ consists of two disjoint copies 
of $V_1$, and~$V_2'$ is a copy of $V_2$.  
\item Each edge of $\half(T)$ connects a vertex in $V_1'$ to one in $V_2'$.  Suppose $u \in V_1$ corresponds to~$u'$ and $u''$ in $V_1'$ and $v \in V_2$ corresponds to $v'$ in $V_2'$.
Then $u'$ and $u''$ are adjacent to $v'$ in 
$\half(T)$ if and only if $u$ is adjacent to $v$ in $T$. 
\end{enumerate}
\end{definition}
 
\noindent In other words, $\half(T)$ is obtained by taking two copies of the tree $T$ and identifying them along their vertices of Type 2.  By construction, the morphism $\half(T) \to T$ induced by sending each vertex in $V(\half(T))$ to the 
corresponding vertex in $V_1$ or $V_2$ (according to the identification of $V_1'$ with two copies of $V_1$, and of $V_2'$ with $V_2$) is a half-covering.  

\subsection{Covering the branch orbifolds}\label{sec:S_beta}

Let $\beta$ be a branch in $\G$ with $n_\beta$ vertices in total.  The orbifold $\cP_\beta$ 
constructed in Section~\ref{sec:Pbeta} 
has underlying space $P$ a right-angled $p$-gon with $p = n_\beta + 1 \geq 5$.   
We now construct a connected surface $S_\beta$ with genus $2(p-4) \geq 2$ and $2$ boundary components so that $S_\beta$ is a 16-fold cover of $\cP_\beta$. 

The surface $S_\beta$ we construct will be tessellated by 16 right-angled $p$-gons, so that:
\begin{itemize}
\item each $p$-gon has exactly one edge in a boundary component of $S_\beta$;
\item the two boundary components of $S_\beta$ each contain $8$ edges; and 
\item types can be assigned to the edges of this tessellation and to the edges of the $p$-gon $P$, such that there is a type-preserving map from $S_\beta$ to $P$ which takes each edge in a boundary component of $S_\beta$ to the (unique) non-reflection edge of $\cP_\beta$. 
\end{itemize}
It follows that the type-preserving map $S_\beta \to P$ induces a degree 16 covering map from the surface~$S_\beta$ to the orbifold $\cP_\beta$.  (More precisely, we can consider the trivial complex of groups $\cG_1(S_\beta)$ over this tessellation of $S_\beta$, that is, the complex of groups in which each local group is trivial.  The type-preserving map $S_\beta \to P$ then induces a covering of complexes of groups from $\cG_1(S_\beta)$ to the simple complex of groups $\cP_\beta$; see~\cite{bridson-haefliger} for the general definition of a covering of complexes of groups.  This induced covering is $16$-sheeted since each face group in $\cG_1(S_\beta)$ and in $\cP_\beta$ is trivial, and $S_\beta$ contains 16 $p$-gons while $P$ is one $p$-gon.)

For the construction of 
$S_\beta$, we consider two cases, $p \geq 5$ odd and $p \geq 6$ even. 

  \begin{figure}
\centering
\includegraphics[scale=.4]{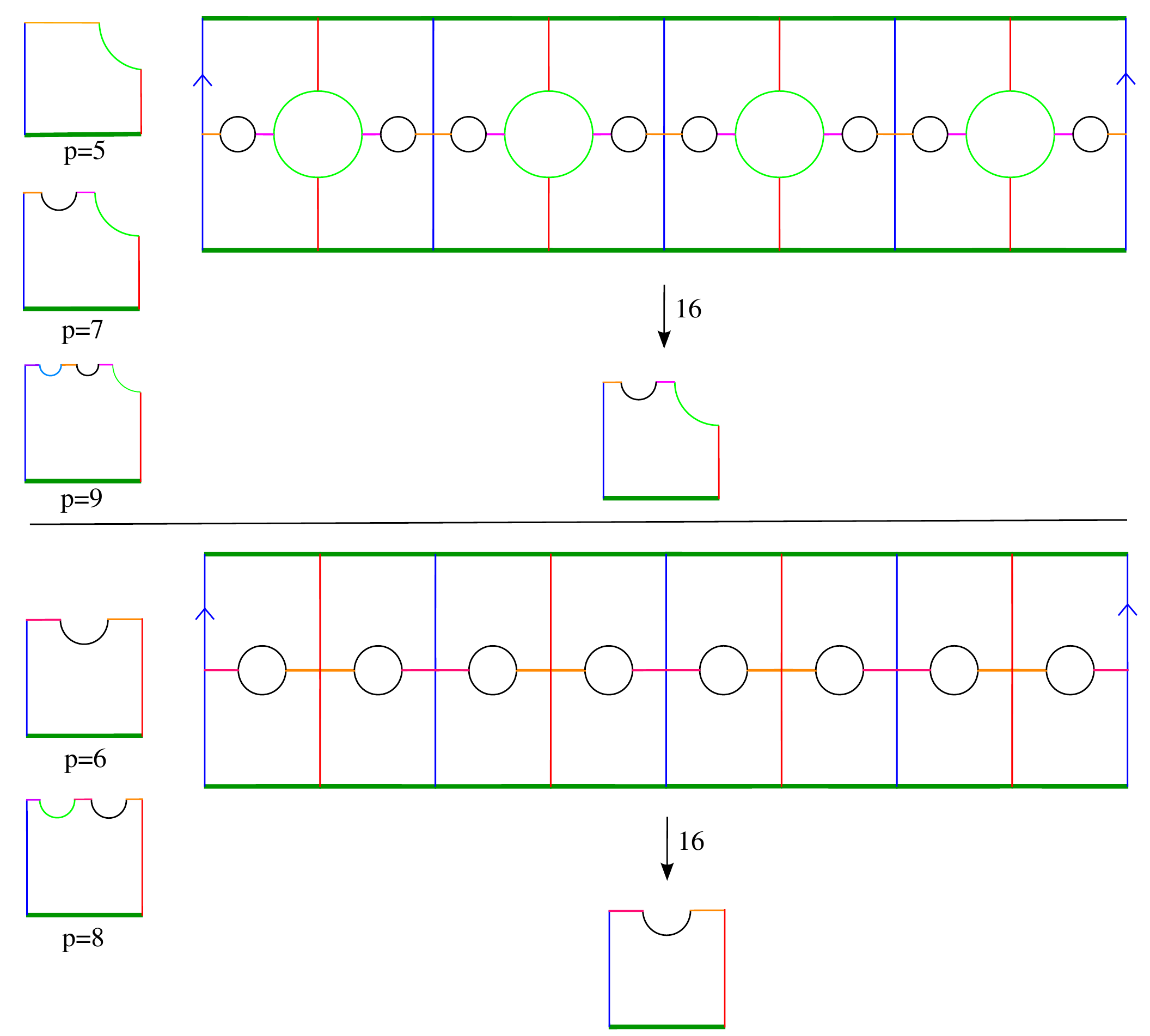}
\caption{{\small On the left are the ``jigsaw puzzle pieces'' in the odd and even cases. On the right are the degree $16$ covers constructed by gluing the pieces together. The left and right blue boundaries are glued and the center circles are glued in pairs to construct a surface with two boundary components, drawn in green. }}
\label{jigsaw_covers}
\end{figure}

\subsubsection*{Case 1: $p \geq 5$ is odd.}
We obtain a tessellated surface $S_\beta$ by gluing together, in a type-preserving manner, 16 copies of the right-angled $p$-gon ``jigsaw puzzle piece"  shown on the top left in Figure~\ref{jigsaw_covers}.  Each piece is straight on the bottom and sides, while along the top we have $\frac{p-5}{2}$ full scoops and $\frac{p-5}{2} + 1$ horizontal edges, and there is one half-scoop in the top right corner.  (So there are $3 + 2\frac{p-5}{2} +1+ 1 = p$ sides in total.)  
Glue 16 of these pieces together into an $8$-by-$2$ block as in Figure~\ref{jigsaw_covers}.  Now glue the left and right boundaries of this block as indicated by the arrows.  The resulting surface will be a sphere with $2$ boundary components corresponding to non-reflection edges (these are the outer horizontal edges drawn in thick lines; they form 2 cycles of length 8 after the gluing), $4$ boundary components corresponding to the half-scoops, and $\frac{p-5}{2} \times 8$ boundary components corresponding to the full scoops.  These boundary components contain $8$, $4$, and $2$ edges respectively.  The final step to obtain $S_\beta$ is to glue together 
the 
boundary components coming from half-scoops and full-scoops in type-preserving pairs.  

\subsubsection*{Case 2: $p \geq 6$ is even.}
In this case we use similar jigsaw puzzle pieces to Case 1, except that 
now 
we have $\frac{p-4}{2}$ full scoops along the top, as shown on the bottom left in 
Figure~\ref{jigsaw_covers}.  Glue~16 of these pieces together into an $8$-by-$2$ block as shown in Figure~\ref{jigsaw_covers}, and identify the left and right boundaries
as indicated by the arrows.  The resulting surface will be a sphere with $2$ boundary components corresponding to non-reflection edges 
and $\frac{p-4}{2} \times 8$ boundary components corresponding to the scoops along the top edge of the piece.  The final step to obtain $S_\beta$ is to glue together these latter boundary components in type-preserving pairs.  

\subsection{Covering the essential vertex and non-branch orbifolds}\label{sec:S_A}

Let $A$ be a subset of vertices of $\G$ so that $\langle A \rangle$ is the stabilizer of a Type 2 vertex in the JSJ decomposition of $W_\G$, and let $L \geq 2$ be the number of essential vertices of $A$.  In this section, we assume that $L \geq 3$ and construct a connected surface $S_\cA$ which is a 16-fold cover of the orbifold $\cA$.  (See Section~\ref{sec:A} for the construction of $\cA$.)  The surface $S_\cA$ will have genus at least $2$ and $2\ell$ boundary components, where $\ell \geq 1$ is the valence in the JSJ decomposition of the vertex with stabilizer $\langle A \rangle$. An illustration of the construction appears in Figure \ref{jigsaw_cut_boundary_cover} (which should be viewed in color).

As in Section~\ref{sec:A}, we will first consider the essential vertex orbifold $\cQ_A$.  We will cover $\cQ_A$ by a connected surface $S_A$ with genus $3L - 7 \geq 2$ and $2L$ boundary components.  Recall that the underlying space of $\cQ_A$ is a right-angled $2L$-gon $Q$.    
The surface $S_A$ we construct will be tessellated by 16 right-angled $2L$-gons, so that:
\begin{itemize}
\item each $2L$-gon has its alternate edges in boundary components of $S_A$;
\item each boundary component of $S_A$ contains $8$ edges; and 
\item types can be assigned to the edges of this tessellation of $S_A$ and to the edges of the $2L$-gon $Q$, such that there is a type-preserving map from $S_A$ to $Q$ which takes each edge in a boundary component of $S_A$ to a non-reflection edge of $\cQ_A$.  In particular, each boundary component of $S_A$ has the same type as a non-reflection edge of $\cQ_A$, and $S_A$ has two boundary components of each type.
\end{itemize}
As for branch orbifolds in Section~\ref{sec:S_beta} above, it follows that the type-preserving map $S_A \to Q$ induces a degree 16 covering map from the surface $S_A$ to the orbifold $\cQ_A$.

  \begin{figure}
\centering
\includegraphics[scale=.5]{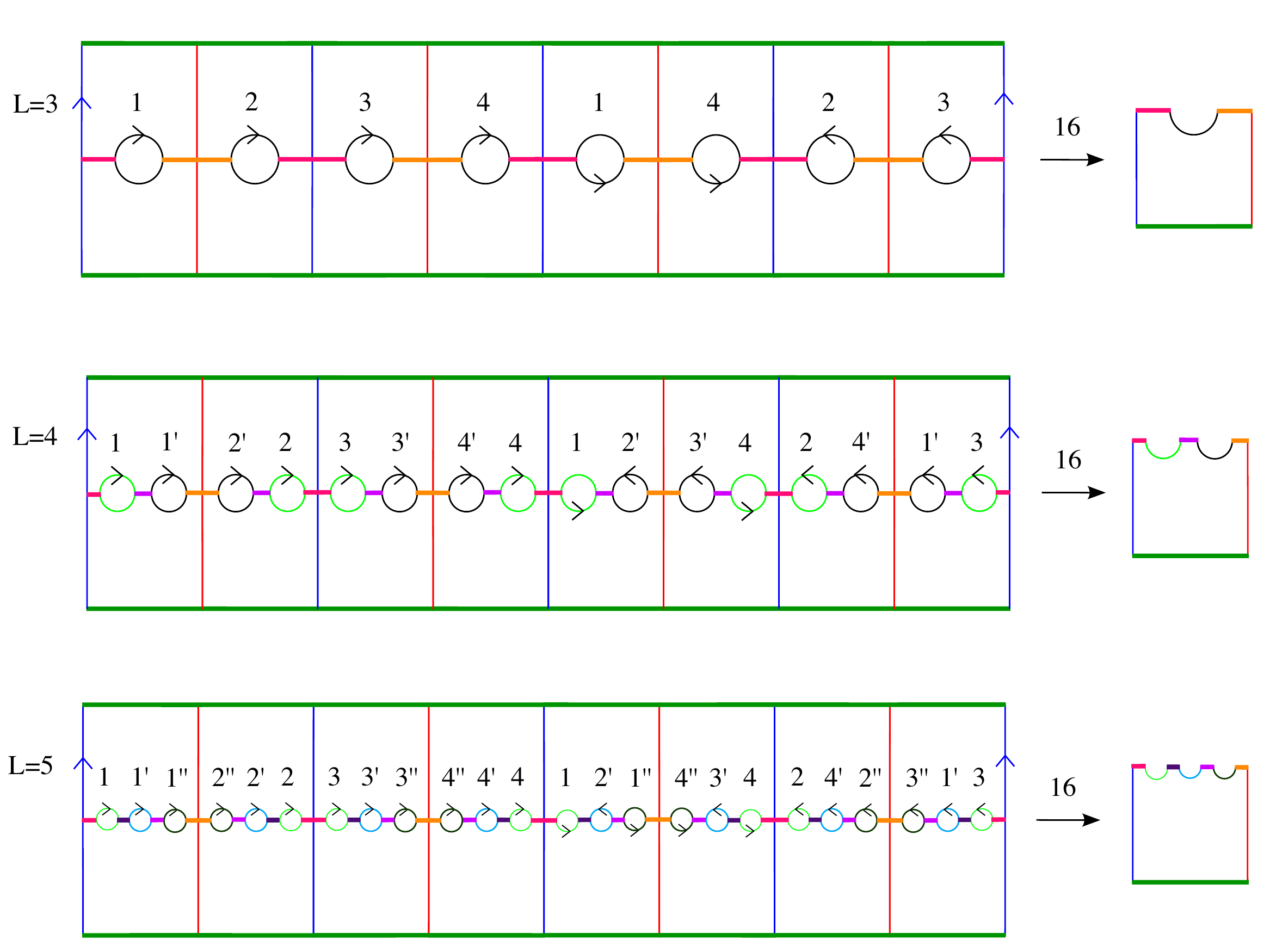}
\caption{{\small On the right are orbifolds with $2L$ sides which alternate between reflection edges 
(non-horizontal, and drawn in thin lines) and non-reflection edges (horizontal, and drawn in thick lines).  On the left are the initial 8 by 2 blocks of jigsaw puzzle pieces for their degree 16 covers. 
The blue vertical sides of these blocks are glued as indicated by the arrows, and the center circles are glued as indicated by the labels and arrows, with the top half of a circle in the left half glued to the top (respectively, bottom) half of a circle in the right half as shown by the arrow on the right-hand circle being on the top (respectively, bottom). After gluing, the center horizontal arcs are cut to obtain a surface with $2L$ boundary components.  }}
\label{jigsaw_cut_boundary_cover}
\end{figure}

Start with the $8$ by $2$ block of jigsaw puzzle pieces from the case $p = 2L \geq 6$ even above.  The non-reflection edges will be the horizontal edges of the puzzle piece.
We will first glue this into a surface with just 2 boundary components of the same type, this type being one of the non-reflection edges of $\cQ_A$.   We will then cut along $L - 1$ curves which have types the other non-reflection edges of $\cQ_A$, to obtain a connected surface with $2 + 2(L-1) = 2L$ boundary components, and with $2$ boundary components for every type of non-reflection edge of $\cQ_A$.

The first stage is to glue together the left and right sides of the block, as indicated by arrows.  We then glue pairs of boundary components which consist of 2 edges; these are the ``center circles" in Figure~\ref{jigsaw_cut_boundary_cover}.  This is done is such a way that the inner horizontal edges of the 8 by 2 block, those coming from the top edges of the jigsaw piece, form cycles of a single type of length 8.  For each non-reflection type, there is one such cycle of $8$ edges.  

In more detail, when $L = 3$, the gluing pattern for the center circles is as indicated by the arrows and labels in the top row of Figure~\ref{jigsaw_cut_boundary_cover}.  The orange horizontal edges, which occur in pairs in the initial block, then form a cycle of 8 edges passing through the center circles with labels $1,2,3,4$ in that order, and similarly the pink horizontal edges form a cycle of 8 edges pass through the center circles with labels $1,3,2,4$ in that order.  The construction is similar for the pink edges for all $L \geq 4$, as indicated in Figure~\ref{jigsaw_cut_boundary_cover}, with the pink edges passing through light green center circles labeled $1,3,2,4$ for all $L \geq 4$.  

For the other horizontal edges, a suitable labeling is constructed by induction on $L \geq 4$.  When $L = 4$ we use the labeling from the middle row of Figure~\ref{jigsaw_cut_boundary_cover}, so that the orange edges form a cycle of 8 edges passing through the center circles with labels $1',2',3',4'$ in that order, and the purple edges form a cycle of 8 edges passing through alternately black and light green center circles with labels $1, 1', 3, 3', 4, 4', 2, 2'$ in that order.  Write $n^{(k)}$ for the integer $n$ followed by $k$ primes.  The labeling of center circles in the left half of the initial block is then, from left to right, $1, 1', \dots, 1^{(L-3)}, 2^{(L-3)},\dots, 2',2, 3, 3',\dots, 3^{(L-3)}, 4^{(L-3)},\dots,4',4$.  On the right half of the initial block, going from left to right, the 5th (respectively, 6th, 7th, and 8th) puzzle pieces have center circles labeled $1,2',1'',\dots$ (respectively, $4^{(L-3)},3^{(L-4)},4^{(L-5)}, \dots$;  $2,4',2'',\dots$; and $3^{(L-3)},1^{(L-4)},3^{(L-5)}, \dots$) when $L$ is odd, and center circles labeled $1,2',1'',\dots$ (respectively, $3^{(L-3)},4^{(L-4)},3^{(L-5)}, \dots$;  $2,4',2'',\dots$; and $1^{(L-3)},3^{(L-4)},1^{(L-5)}, \dots$) when $L$ is even.  By induction and the alternating pattern of labels on the right half, the only color to be checked is orange.  The orange horizontal edges come in pairs in the initial block, and it is not hard to verify that for all $L \geq 4$, they form a cycle of length 8 passing through black center circles with labels $1^{(L-3)}, 2^{(L-3)}, 3^{(L-3)}, 4^{(L-3)}$ in that order.

Now cut along each of these cycles of length 8.  There are then $2$ boundary components of each non-reflection type, each consisting of $8$ edges.  Using the gluing described above, the resulting surface is connected: for this, we just need to glue at least one top half of a scoop to the bottom half of a scoop, and this follows from our construction.

We now describe the surface $S_\cA$ which 16-fold covers the orbifold $\cA$.  If $A$ consists entirely of essential vertices, then we put $S_\cA = S_A$.  Otherwise, for each branch $\beta$ so that the vertex set of $\beta$ is in the set $A$, we glue $S_\beta$ to $S_A$ as follows.  Recall that all edges in the tessellations of $S_\beta$ and $S_A$ have types, that these types are preserved by the covering maps $S_\beta \to \cP_\beta$ and $S_A \to \cQ_A$, and that each boundary component of $S_\beta$ and $S_A$ has $8$ edges.  Recall also that $\cA$ is obtained by gluing $\cP_\beta$ to~$\cQ_A$ along a non-reflection edge in a type-preserving manner, for each branch $\beta$ with vertex set in $A$.  For all such $\beta$, we now glue one boundary component of $S_\beta$ to one boundary component of $S_A$ with its same type, and the other boundary component of $S_\beta$ to the other boundary component of $S_A$ with its same type, so that these gluings match up edges and vertices of the existing tessellations, and so that at each vertex of the resulting tessellation, the incident edges have exactly two types.

Finally, we erase the edges in the resulting tessellation which were in boundary components of $S_A$, and denote the resulting tessellated surface by $S_\cA$.
Note that $S_\cA$ is still tessellated by right-angled polygons whose edges have well-defined type.  By construction, the 16-fold coverings $S_\beta \to \cP_\beta$ and $S_A \to \cQ_A$ induce a 16-fold covering $S_\cA \to \cA$.

We observe that since $S_A$ has genus $3L - 7 \geq 2$, the surface $S_\cA$ has genus at least $2$.   Also, the surface $S_\cA$ has two boundary components for each non-reflection edge of $\cA$.  Now by construction of the orbicomplex $\cO_\G$ in Section~\ref{sec:OGamma}, the number of non-reflection edges of $\cA$ equals the valence $\ell \geq 1$ of the vertex stabilized by $\langle A \rangle$ in the JSJ decomposition of $W_\G$.  Hence $S_\cA$ has $2\ell$ boundary components.

\subsection{The surface amalgam $\cX$}\label{sec:X}

We now construct a surface amalgam $\cX$ which 16-fold covers the orbicomplex $\cO_\G$, by gluing together certain surfaces $S_\beta$ constructed in Section~\ref{sec:S_beta} and all of the surfaces $S_\cA$ constructed in Section~\ref{sec:S_A}. Examples appear in Figures~\ref{GenTheta} 
and~\ref{CycleGenTheta}.

  \begin{figure}
\centering
\includegraphics[scale=.5]{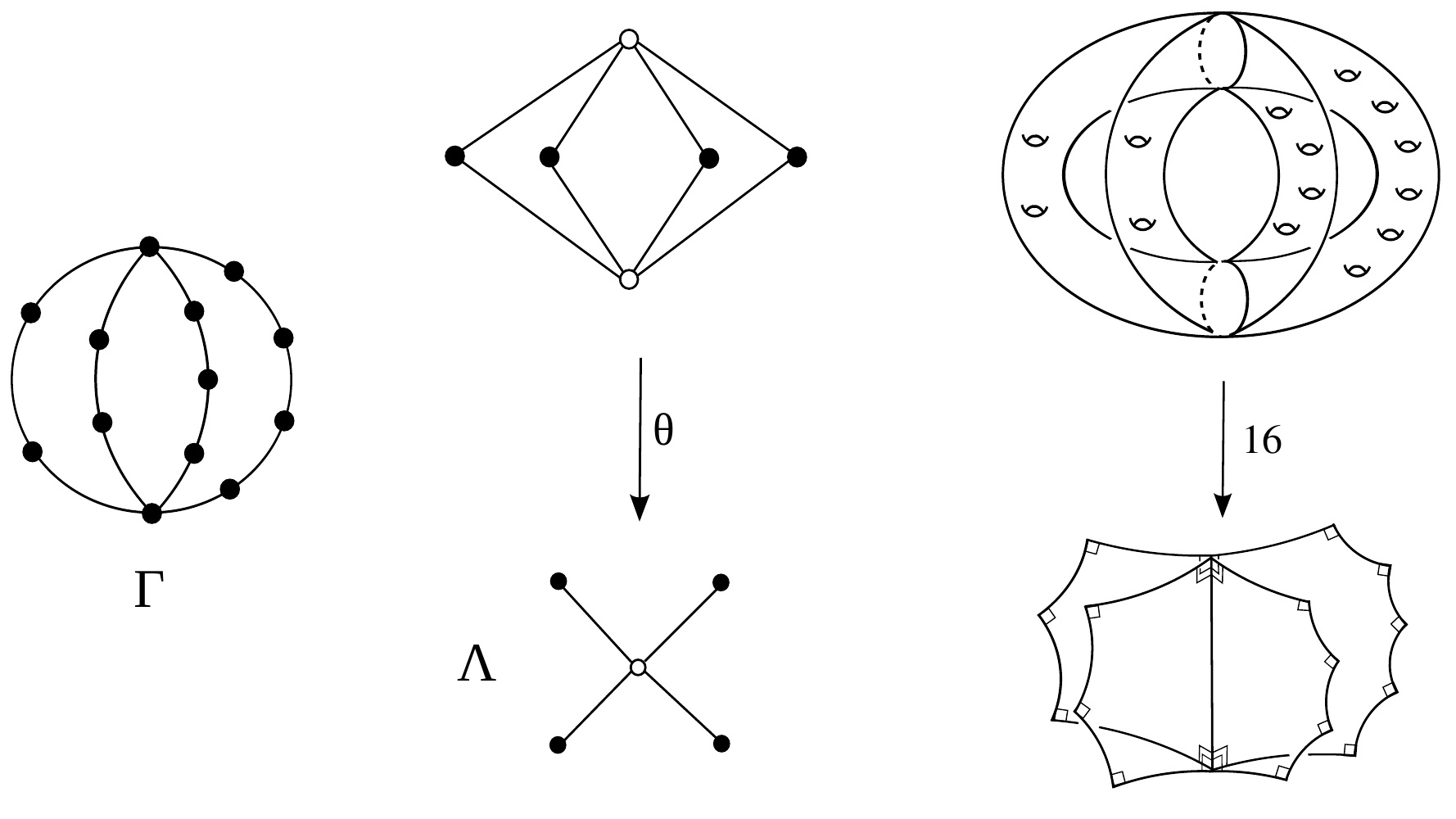}
\caption{{\small The graph $\Gamma$ is a generalized $\Theta$-graph, and the graph $\Lambda$ is the JSJ graph of $W_{\G}$. On the right is the degree $16$ cover $\cX_{\G} \rightarrow \cO_{\G}$, where $\cO_{\G}$ is the orbicomplex described in Section \ref{sec:OGamma}. In the center is the half-covering $\theta$ from the JSJ graph of $\cX_{\G}$ to the JSJ graph of $\cO_{\G}$.  }}
\label{GenTheta}
\end{figure}

Let $\Lambda$ be the JSJ graph for $W_\G$ and recall that $\Lambda$ is a bipartite tree.  Let $\half (\Lambda)$ be the graph from Definition~\ref{def:halfT} which half-covers $\Lambda$.  For each $v \in V_2(\Lambda)$, we write $v'$ for the corresponding vertex in $V_2(\half(\Lambda))$.  Now for all $v \in V_2(\Lambda)$, let $S_v$ be the surface $S_\beta$, if $v$ has stabilizer $W_\beta$ for $\beta$ a branch of $\G$, and let $S_v$ be the surface $S_\cA$,  if $v$ has stabilizer $\langle A \rangle$ with $A$ not equal to the vertex set of a branch of $\G$.  Then the collection of surfaces $\{S_v \mid v \in V_2(\Lambda) \} = \{S_v \mid v' \in V_2(\half(\Lambda)) \}$ is in bijection with the Type 2 vertices of $\half(\Lambda)$.  

We now obtain $\cX$ by gluing together the surfaces $\{ S_v \}$ along boundary components, according to the adjacencies in the graph $\half(\Lambda)$.  (In fact, the JSJ graph of $\cX$ will be $\half(\Lambda)$.)  By construction, each $S_v$ has $2\ell$ boundary components, where $\ell \geq 1$ is the valence of $v$ in $\Lambda$.  Since $v'$ has valence $2\ell$ in $\half(\Lambda)$, each $S_v$ has number of boundary components equal to the valence of $v'$.  We now relabel the boundary components of $S_v$ using the vertices of $\half(\Lambda)$ which are adjacent to $v'$, as follows.  Recall that the boundary components of $S_v$ come in pairs which cover the same non-reflection edge (in either $\cP_\beta$ or $\cA$), and so have type $\{a,b\}$ where $\langle a ,b\rangle$ is the stabilizer of a Type 1 vertex in $\Lambda$.  For each $u\in V_1(\Lambda)$ with stabilizer $\langle a, b \rangle$, with corresponding Type 1 vertices $u'$ and $u''$ in $\half(\Lambda)$, and for each $v$ adjacent to $u$ in $\Lambda$, we label one boundary component of $S_v$ of type $\{a,b\}$ by $u'$ and the other by $u''$.  Now every boundary component in the collection $\{ S_v \mid v' \in V_2(\half(\Lambda)) \}$ has been assigned a type in $V_1(\half(\Lambda))$, and for each $v \in V_1(\Lambda)$, there is a bijection between the types of boundary components of $S_v$ and the vertices adjacent to $v'$ in $\half(\Lambda)$.  We then glue together all boundary components in the collection $\{S_v\}$ which have the same type, so that these gluings match up edges and vertices of the existing tessellations.  The resulting surface amalgam is $\cX$.  By construction, the 16-fold covers $S_\beta \to \cP_\beta$ and $S_\cA \to \cA$ induce a 16-fold cover $\cX \to \cO_\G$.

  \begin{figure}
\centering
\includegraphics[scale=.5]{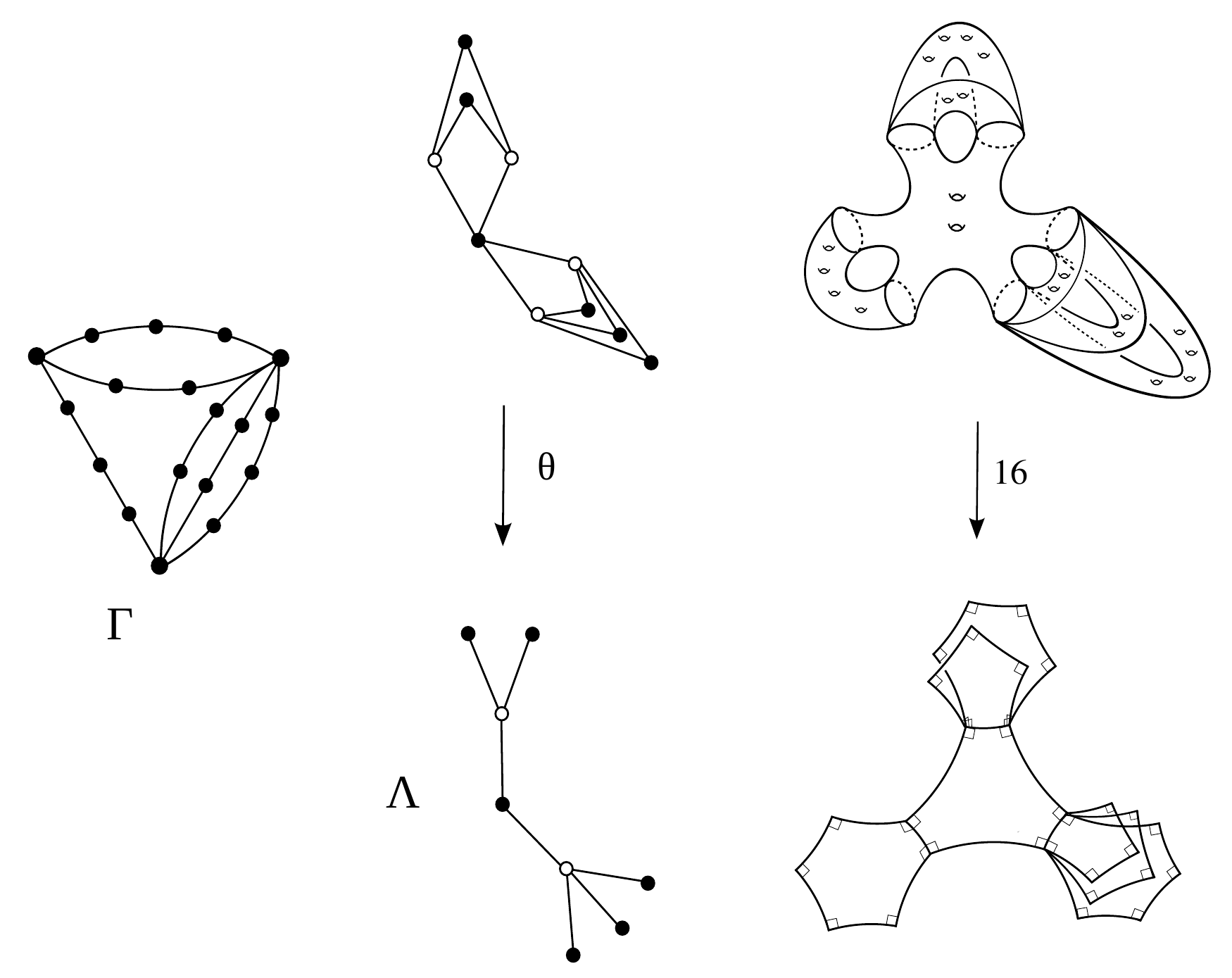}
\caption{{\small The graph $\Gamma$ is a cycle of generalized $\Theta$-graphs, and the graph $\Lambda$ is the JSJ graph of $W_{\G}$. On the right is the degree $16$ cover $\cX_{\G} \rightarrow \cO_{\G}$, where $\cO_{\G}$ is the orbicomplex described in Section \ref{sec:OGamma}. In the center is the half-covering $\theta$ from the JSJ graph of $\cX_{\G}$ to the JSJ graph of $\cO_{\G}$.  }}
\label{CycleGenTheta}
\end{figure}

\subsection{Summary and remarks}\label{sec:SummaryRemarks}

We now provide in Table~\ref{table:summary} a summary of the surfaces constructed in this section, and explain why we constructed covers of degree 16 (rather than some smaller degree).

\begin{table}[htp]
\begin{center}
\begin{tabular}{| l | l | l | l | l |}
\hline
Orbifold & Orbifold & Connected & Genus of &  No. of boundary \\
& Euler & covering & covering & components of \\
& characteristic & surface & surface & covering surface\\
\hline \hline
Branch orbifold $\cP_\beta$ & $\chi(W_\beta) = \frac{4-p}{4}$ & $S_\beta$ & $2(p-4) \geq 2$ & $2$ \\
over $p$-gon, $p \geq 5$, & & & & \\
with $p-1$ reflection edges, &&&& \\
$1$ non-reflection edge &&&& \\
\hline
Essential vertex orbifold $\cQ_A$ & $\frac{2 - L}{2}$ & $S_A$ & $3L - 7 \geq 2$ & $2L$ \\
over $2L$-gon, $L \geq 3$, &&&& \\
with $L$ reflection edges, & && & \\
$L$ non-reflection edges &&&& \\
\hline
Non-branch orbifold $\cA$ & $\chi(W_A)$ & $S_\cA$ & $\geq 2$ & $2\ell$ \\
with $\ell \geq 1$ non-reflection edges & & & & \\
\hline
\end{tabular}
\end{center}
\caption{Summary of orbifolds constructed in Section~\ref{sec:orbicomplex} and their 16-fold connected covering surfaces constructed in Section~\ref{sec:tfcover}.}
\label{table:summary}
\end{table}%

To cover a branch orbifold $\cP_\beta$ with underlying space a right-angled $p$-gon, we wanted a construction for all $p \geq 5$ of a surface with positive genus and a small positive number of boundary components.  The Euler characteristic of $\cP_\beta$ is $\frac{4-p}{4}$ so a degree $D$ cover has Euler characteristic $\frac{D(4-p)}{4}$.  For a surface, the Euler characteristic must be an integer, hence $D = 4k$ for some $k \geq 1$.  Now if the covering surface $S_\beta$ has genus $g$ and $b$ boundary component, then $\chi(S_\beta) = 2 - 2g - b$.  Equating Euler characteristics when $b = 1$, we get $k(4-p) = 1 - 2g$.  This has no solution when $p$ is even.  Thus we consider $b = 2$ boundary components.  Equating Euler characteristics when $b = 2$, we get $k(4-p) = - 2g$.  This has a positive genus solution for all $p \geq 5$, provided $k$ is even.  So we work with covers of degree $D = 4k$ where $k \geq 2$ is even.

Now, an essential vertex orbifold $\cQ_A$ with $L$ reflection edges and $L$ non-reflection edges has Euler characteristic $\frac{2 - L}{2}$, where $L \geq 3$.  We wanted to cover each such $\cQ_A$ by a surface $S_A$ of positive genus~$g$ with $2L$ boundary components (since each surface $S_\beta$ covering a branch orbifold~$\cP_\beta$ will have $2$ boundary components, and certain $S_\beta$ will be glued onto the surface covering $\cQ_A$ in order to cover the non-branch orbifold $\cA$).  If the degree of the cover is $D = 8$, by equating Euler characteristics we get $4(2 - L) = 2 - 2g - 2L$ hence $g = L - 3$.  For $L = 3$ this gives $g = 0$, but we would like $S_A$ to have positive genus.  The next possible degree is  $D = 16$, and the Euler characteristic equation $8(2 - L) = 2 - 2g - 2L$ has positive genus solution $g = 3L - 7$ for all $L \geq 3$.  Hence we work with degree 16 covers.

\section{Generalized $\Theta$-graphs}\label{sec:GenTheta}

In this section we prove Theorem~\ref{thm:GenTheta}, which gives the commensurability classification of right-angled Coxeter groups with defining graph a $3$-convex generalized $\Theta$-graph.  We remark that by Corollary~\ref{cor:jsj}, the JSJ graph of such a group is a $k$-valent star, so in particular, the JSJ graph is a tree of diameter 2.   

 Let  $W_{\Theta}$ and $W_{\Theta'}$ be right-angled Coxeter groups whose defining graphs are
 $3$-convex generalized $\Theta$-graphs $\Theta = \Theta(n_1,\dots,n_k)$ and $\Theta' = \Theta(n_1',
 \dots,n'_{k'})$ respectively, with $k, k' \geq 3$.  If $\Theta$ has branches~$\beta_i$ for $1 \le i \le k$, then 
$W_{\Theta}$ is the fundamental group of the orbicomplex $\cO=\cO_\Theta$ 
  constructed in Section~\ref{sec:orbicomplex},
  obtained by gluing together 
  the branch orbifolds $\mathcal{P}_i = \cP_{\beta_i}$, for $1 \leq i \leq k$, along their non-reflection 
  edge.  
Since $W_i = W_{\beta_i}$ is the fundamental group of~$\cP_{i}$, the 
 Euler characteristic vector of $W_{\Theta}$  
 from Definition~\ref{def:theta-euler} is $v = (\chi(W_1), \ldots, \chi(W_k))$. 
 Similarly, $W_{\Theta'}$ is the fundamental group 
of $\cO'=\mathcal{O}_{\Theta'}$ obtained by gluing together the branch orbifolds $\cP_i' = 
  \cP_{\beta_i'}$, for $1 \leq i \leq k'$, and~$W_{\Theta'}$ has Euler characteristic vector
 $v' = (\chi(W_1'), \ldots, \chi(W_{k'}'))$ where $W'_i = W_{\beta_i'}$.
 
\subsection{Sufficient conditions for commensurability}
  
Suppose that the vectors $v$ and $v'$ are commensurable. Then by definition, $k = k'$ and there exist integers $K, K' \geq 1$ so that $Kv = K'v'$. Let~$a$ and $b$ be the two essential vertices of $\Theta$.  Let $r_a$ 
be the subcomplex of $\cO$ consisting of all reflection edges with local group $\langle a \rangle$ (so $r_a$ is a star of valence $k$), and let $r_b \subset \cO$ be the corresponding subcomplex for $b$.  
An immediate generalization of \cite[Section 3.1]{crisp-paoluzzi} is that, for any positive integer~$R$, there is a degree $R$ orbicomplex covering $R\mathcal{O} \rightarrow \mathcal{O}$ given by unfolding $R$ times along copies of $r_a$ and $r_b$ so that the central branching edge of $R \mathcal{O}$ forms a geodesic path. An easy counting argument then proves that the orbicomplexes $K\mathcal{O}$ and $K'\mathcal{O'}$ are homeomorphic, hence have isomorphic fundamental groups which are finite-index subgroups of $W_\Theta$ and $W_{\Theta'}$, respectively. Therefore $W_{\Theta}$ and $W_{\Theta'}$ are commensurable.

\subsection{Necessary conditions for commensurability}

The proof of the necessary conditions in Theorem~\ref{thm:GenTheta} follows from a slight generalization of \cite[Proposition 3.3.2]{stark}. In that setting, it was assumed that the analogs of 
$\cO$ and $\cO'$ had the same Euler characteristic, while here, $\chi(W_{\Theta}) = \chi(\cO)$ and $\chi(W_{\Theta'}) = \chi(\cO')$ could be unequal. The proof  from~\cite{stark} can thus be altered as follows. 
  
Suppose that the groups $W_{\Theta}$ and $W_{\Theta'}$ are commensurable.  Then they are 
quasi-isometric, and by \cite[Theorem~3.36] {dani-thomas-jsj}, we have $k=k'$. Since $W_{\Theta}$ and $W_{\Theta'}$ are virtually torsion-free, 
they have isomorphic 
torsion-free, finite-index subgroups $H \leq W_{\Theta}$ and $H' \leq W_{\Theta'}$. 
Then the corresponding covers  $\cX$
and $\cX'$ of $\cO$ and $\cO'$ are surface amalgams.  Let $D$ and $D'$ be the degrees with which 
$\cX$ and $\cX'$ respectively cover $\cO$ and $\cO'$.  
By 
 Theorem~\ref{thm:Lafont},
there is a homeomorphism $f:\cX \rightarrow \cX'$ that induces the isomorphism between $H$ 
and $H'$. Suppose 
\begin{align*}
 \chi(W_1) = \chi(W_2) = \dots = \chi(W_s) &> \chi(W_{s+1}) \geq \dots \geq \chi(W_k), \text{ and }\\
    \chi(W_1') = \chi(W_2') = \dots = \chi(W_t') &> \chi(W_{t+1}') \geq \dots \geq \chi(W_k')
  \end{align*}
   for some $s,t \leq k$.  (Note that the ordering in the above display is reversed with respect to the ordering in the proof of Proposition 3.3.2 of~\cite{stark}.)
   Without loss of generality, $D\cdot \chi(W_1) \geq D'\cdot \chi(W_1')$, and if $D\cdot\chi(W_1)= D'\cdot\chi(W_1')$, then $s\geq t$. The remainder of the proof \cite[Proposition 3.3.2]{stark} may now be applied with the first and last line of the main equation in the proof changed to account for the fact that the degrees of the covering maps are different.

\section{Necessary conditions for cycles of generalized $\Theta$-graphs}\label{sec:CycleGenThetaNec}

In this section, we establish our necessary conditions for the commensurability of right-angled Coxeter groups defined by cycles of generalized $\Theta$-graphs, from 
Theorem~\ref{thm:CycleGenTheta}.   We show:

\begin{prop}\label{prop:ness-cond}
Let $W$ and $W'$ be as in Theorem~\ref{thm:CycleGenTheta}.  If $W$ and $W'$ are commensurable, then at least one of (1) and (2) from Theorem~\ref{thm:CycleGenTheta} holds.   
\end{prop}

In Section~\ref{sec:nec_notation} we fix  notation, and we explain the three cases we will consider in Section~\ref{sec:cases}.  Case~1 is proved in Section~\ref{sec:Case1}.  We establish some results for both Cases 2 and 3 in Section~\ref{sec:prelims23}, then complete the proof of Case 2 in Section~\ref{sec:Case2} and that of Case 3 in Section~\ref{sec:Case3}.

\subsection{Notation}\label{sec:nec_notation}

Suppose $W$ and $W'$ as in Theorem~\ref{thm:CycleGenTheta} are commensurable.  We continue all notation from Section~\ref{sec:defs-statements}.  In addition, let $\cO_\G$ and $\cO_{\G'}$ be the orbicomplexes with fundamental groups $W = W_{\G}$ and $W' = W_{\G'}$, respectively, constructed in Section~\ref{sec:orbicomplex}, and let $\cA$ and $\cA'$ be the (unique) non-branch orbifolds in $\cO_\G$ and $\cO_{\G'}$, with fundamental groups $W_A$ and $W_{A'}$, respectively.  So $\cO_\G$ (respectively, $\cO_{\G'}$) is obtained by gluing certain branch orbifolds to $\cA$ (respectively, $\cA'$) along non-reflection edges.   For each $i \in I$ and $1 \leq j \leq r_i$, let $\cP_{ij}$ be the branch orbifold in $\cO_\G$ with fundamental group $W_{ij}:=W_{\beta_{ij}}$, and let $\cO_i$ be the sub-orbicomplex of $\cO_\G$ obtained by gluing together the $r_i$ orbifolds $\cP_{ij}$ along their non-reflection edge.  Then $\pi_1(\cO_i) = W_{\Theta_i}$ for each $i \in I$. Similarly define $\cP'_{kl}$, $W'_{kl}$, and $\cO'_k$ with $\pi_1(\cO'_k) = W_{\Theta'_k}$ for each $k \in I'$ and $1 \leq l \leq r_k'$.  We write $E_i$ for the $i$th \emph{branching edge} of $\cO_\G$, that is, the non-reflection edge along which $\cO_i$ is glued to $\cA$, and similarly write $E_k'$ for the $k$th branching edge of $\cO_{\G'}$.

If $r_i = r = r'_k$ for each $i \in I$ and each $k \in I'$, then for all $1 \leq j \leq r$ we define the subcomplex $\cR_j$ of $\cO_\G$ to be the disjoint union of the ``$j$th ring" of branch orbifolds $\{ \cP_{ij} \mid i \in I\}$ in $\cO_\G$, and define $\chi_j$ to be its Euler characteristic.   
  (Note that, by construction, if $i_1 \neq i_2 \in I$ then the underlying spaces of $\cP_{i_1j}$ and $\cP_{i_2 j}$ are disjoint polygons.)  
Thus $\chi_j = \chi(\cR_j) = \sum_{i \in I} \chi(P_{ij})  = \sum_{i \in I} \chi_{ij} $.
For each $1 \leq j \leq r$, we define $\cR'_j \subset \cO_{\G'}$ analogously, and put 
$  \chi'_j= \chi(\cR'_j)$.  

Let $\rho: \cX \to \cO_\G$ and $\rho':\cX'\to \cO_{\G'}$ be the degree 16 torsion-free covers from Section~\ref{sec:tfcover}.  Then $\pi_1(\cX)$ and $\pi_1(\cX')$ are commensurable as well.  Let $\eta:\cY \to \cX$ and $\eta':\cY' \to \cX'$ be covers corresponding to isomorphic finite-index subgroups of $\pi_1(\cX)$ and $\pi_1(\cX')$, and set $\pi = \eta \circ \rho$ and  
 $\pi' = \eta' \circ \rho'$.  Let $D$ and $D'$ be the degrees of the covering maps $\pi:\cY \to \cO_\G$
 and $\pi':\cY' \to \cO_{\G}$, respectively.   Finally, let $f:\cY \to \cY'$ be the homeomorphism guaranteed by Theorem~\ref{thm:Lafont}.
 
\subsection{Cases}\label{sec:cases} 
 
To prove Proposition~\ref{prop:ness-cond}, we consider three cases, by comparing the subsets $f(\pi^{-1} (\cA))$ and $ \pi'^{-1} (\cA')$ of the surface amalgam $\cY'$.  Note first that since $\cA$ (respectively, $\cA'$) contains every branching edge of $\cO_\G$ (respectively, $\cO_{\G'}$), the sets $f(\pi^{-1} (\cA))$ and $ \pi'^{-1} (\cA')$ will always have non-empty intersection containing all branching curves in $\cY'$.  We write $f(\pi^{-1} (\cA)) \intcap \pi'^{-1} (\cA') = \emptyset$ if the \emph{interiors} of $f(\pi^{-1} (\cA))$ and $ \pi'^{-1} (\cA')$ are disjoint, that is, these subsets of $\cY'$ have no surfaces in common, and we write $f(\pi^{-1} (\cA)) \intcap \pi'^{-1} (\cA')\neq \emptyset$ if the \emph{interiors} of $f(\pi^{-1} (\cA))$ and $\pi'^{-1} (\cA')$ are non-disjoint, that is, these subsets of $\cY'$ have at least one surface in common.  The cases we consider, and their consequences, are as follows:

\begin{enumerate}[{{Case} 1.}]
\item If $f(\pi^{-1} (\cA)) = \pi'^{-1} (\cA')$, we show that condition (1) holds. 
\item If $f(\pi^{-1} (\cA)) \stackrel{\circ}{\cap} \pi'^{-1} (\cA') = \emptyset$, we show that condition (2) holds.
\item If $f(\pi^{-1} (\cA)) \neq \pi'^{-1} (\cA')$ and $f(\pi^{-1} (\cA)) \intcap \pi'^{-1} (\cA')\neq \emptyset$, we construct new homeomorphic finite-sheeted covers of 
$\cO_{\G}$ and $\cO_{\G'}$ which satisfy Case 1.  It follows that condition (1) holds.  
\end{enumerate}

\subsection{Case 1}\label{sec:Case1}
  
In this case we prove:
\begin{prop}
If $f(\pi^{-1} (\cA)) = \pi'^{-1} (\cA')$ then condition (1) in Theorem~\ref{thm:CycleGenTheta} holds.  
\end{prop}

\begin{proof}
Given $i \in I$, consider one component $\cS$ of the preimage of $\cO_i$ in $\cY$.  The assumption that $f(\pi^{-1} (\cA)) = \pi'^{-1} (\cA')$
implies that $\pi'(f(\cS))$ cannot intersect the interior of $\cA'$, and so $f(\cS)$ must 
cover some orbicomplex $\mathcal{O}'_k$ in $\cO_{\G'}$ with $r_k' =r_i$. 
If $r_i\geq 3$, then the vectors $v_{i}$ and $v_k'$ are  commensurable by Theorem~\ref{thm:GenTheta} for 
generalized $\Theta$-graphs.  

If $r_i =2$ and $\Theta_i$ has branches $\beta_{i1}$ and $\beta_{i2}$, then form a new $\Theta$-graph $\Theta$ with branches $\beta_{i1}$ and two copies of $\beta_{i2}$, and do the same for $\Theta'_k$ to get $\Theta'$.  Now $W_\Theta$ and $W_{\Theta'}$ are commensurable, since we can construct homeomorphic covers of $\cO_\Theta$ and $\cO_{\Theta'}$ from $\cS$ and $f(\cS)$ by adding extra copies of the surfaces mapping to $\cP_{i2}$ and $\cP'_{k2}$, respectively.  It follows from Theorem~\ref{thm:GenTheta} that the vectors 
$(\chi_{i1}, \chi_{i2}, \chi_{i2})$ and 
$(\chi_{k1}', \chi_{k2}', \chi_{k2}')$ are commensurable, and so $v_i = (\chi_{i1}, \chi_{i2})$ and $v_k'=(\chi_{k1}', \chi_{k2}')$ are as well.

Applying the same argument with $f^{-1}$, we conclude that the sets of commensurability classes of the vectors $\{v_i \mid r_i \geq 2\}$ and $\{v_k' \mid r_k' \geq 2\}$ coincide, proving condition (1)(a). 

We now prove (1)(b).  It follows from the proof of (1)(a) that 
\[  f\left(\pi^{-1}\left( \bigcup_{i\in I} \cO_i\right)\right) = (\pi')^{-1}\left( \bigcup_{k\in I'}\cO'_k\right). \]
Now $\pi_1(\cO_i) = W_{\Theta_i}$ for each $i \in I$ and $\pi_1(\cO'_k) = W_{\Theta'_k}$ for each $k \in I'$.  As the degrees of $\pi$ and $\pi'$ are $D$ and $D'$ respectively, we deduce that
\begin{equation}\label{eq:full-preimage}
D\; \left( \displaystyle \sum_{i \in I} \chi(W_{\Theta_i})\right)=
D'\; \left( \displaystyle  \sum_{k \in I'} \chi(W_{\Theta'_k})\right). 
\end{equation}
Now $f(\pi^{-1}(\mathcal{A})) = (\pi')^{-1}(\mathcal{A}')$ by assumption, and $\pi_1(\cA) = W_A$ and $\pi_1(\cA') = W_{A'
}$, so we also have $D\cdot  \chi(W_A)=D' \cdot \chi(W_{A'})$.  This together with Equation~\eqref{eq:full-preimage} implies (1)(b). 
\end{proof}

\subsection{Results for both Cases 2 and 3}\label{sec:prelims23} 

In this section we establish some results which are relevant to both of the remaining cases.  From now on, we suppose that $f(\pi^{-1} (\cA)) \neq \pi'^{-1} (\cA')$.  We start by showing in Lemma~\ref{lem:equal-branching} that all nontrivial generalized $\Theta$-graphs in both $\G$ and $\G'$ have the same number of branches.  Next, in Proposition~\ref{prop:ww'-commensurable} we prove that the vectors $w$ and~$w'$ are commensurable, hence condition (2)(b) in Theorem~\ref{thm:CycleGenTheta} holds.  We then color certain sub-orbicomplexes of $\cO_\G$ and $\cO_{\G'}$, using the commensurability of $w$ and $w'$, and make some observations about this coloring in Remark~\ref{rem:colors-preserved} and Lemma~\ref{lem:orbifold-color}.  In Remark~\ref{remark:preimageAA'} we discuss the structure of the subset $\pi^{-1} (\cA) \cup f^{-1} (\pi'^{-1} (\cA'))$ of $\cY$, and then in Lemma~\ref{lem:epsilonik} we consider pre-images of branching edges in $\cO_\G$ and $\cO_{
\G'}$.

\begin{lemma}\label{lem:equal-branching}
If $f ( \pi^{-1} (\cA)) \neq \pi'^{-1} (\cA')$ then there exists $r \ge 2$ such that $r_i=r'_k=r$ 
for all $i \in I$ and all  $k \in I'$.  
\end{lemma}

\begin{proof}
We prove the contrapositive.  If the conclusion fails, then one of the following holds: 
\begin{enumerate}[(i)]
\item 
$r_{i_1} \neq r_{i_2}$ for some $r_{i_1}, r_{i_2} \in I$;
\item 
$r'_{k_1} \neq r'_{k_2}$ for some $r'_{k_1}, r'_{k_2} \in I'$; or
\item 
 $r_i=r$ 
for all $i \in I$ and $r'_k=s$ 
for all $k \in I'$, but $r\neq s$.  
\end{enumerate}
Suppose (i) holds. If $S$ is any component of $\pi^{-1}(\cA)$, then $S$ is incident to branching curves in $\cY$ 
with different 
branching degrees $r_{i_1} + 1$ and $r_{i_2} + 1$.   Then $\pi'(f(S))$ is an orbifold of $\cO_{\G'}$ which is incident to branching edges of at least two different degrees.  It follows that $\pi'(f(S)) = \cA'$ and~$\cA'$ has branching edges of at least two degrees.  Applying the same argument with $f^{-1}$, we see that 
$\pi(f^{-1}(T) )= \cA$ for each component $T$ of ${\pi'}^{-1}(\cA')$.  Thus 
$f(\pi^{-1}(\cA)) =\pi'^{-1}(\cA')$. The proof is identical if (ii) occurs. Finally, (iii) cannot occur, since the degree of branching is preserved by homeomorphisms.  
\end{proof}

We next extend the techniques used to show that the Euler characteristic vectors of generalized $\Theta$-graphs are commensurable (in the proof of Theorem~\ref{thm:GenTheta}) to prove the following.

\begin{prop}\label{prop:ww'-commensurable}
The vectors $w$ and $w'$  are commensurable.
\end{prop}

\begin{proof}
Let $w = (w_{1}, \ldots, w_{r+1})$ and $w' = (w'_{1}, \ldots, w'_{r+1})$, and suppose
$$w_{1} = \dots = w_{s} > w_{s+1} \geq \dots \geq w_{r+1}, \quad \text{ and } \quad
 w'_{1} = \dots = w'_{t} > w'_{t+1} \geq \dots \geq w'_{r+1}.$$
Without loss of generality, we may assume that $D w_{1} \geq D' w'_{1}$, and if $D w_{1} = D'w'_{1}$ then $ t \leq s$.     
For each $1 \le u \le r+1$, let 
$\cT_u \subset \cO_\G$ be the sub-orbicomplex corresponding to the entry $w_{u}$ of $w$. That is, if $w_u = \chi(W_A)$ then $\cT_u = \cA$, and if $w_u = \chi_j$ then $\cT_u=\cR_j$.
Similarly, let $\cT'_{u} \subset \cO_{\G'}$ be the sub-orbicomplex corresponding to the entry $w'_{u}$ of $w'$.  Note that $f(\pi^{-1}(\cT_u)) \subset \cY'$ is a disjoint collection of connected surfaces with boundary, 
and the set of boundary curves of these surfaces is exactly the set of branching curves of $\cY'$.  
  
 We now partition $f(\pi^{-1}(\cT_1)) \subset \cY'$  as follows:
   \begin{itemize}
    \item $\mathcal{S}_{\cA'}$ is the union  of the connected surfaces in $f(\pi^{-1}(\cT_1))$ which cover~$\cA'$. 
   
    \item For each $k \in I'$, $\cS_k$ is the union of the connected surfaces in $f(\pi^{-1}(\cT_1))$ which cover a branch orbifold in $\cO_{k}'$.  
\end{itemize}

Suppose $\mathcal{S}_{\cA'}$ forms a cover of degree $D'' \leq D'$ of $\mathcal{A}'$.  This includes the possibility that $\cS_{\cA'}$ is empty, in which case we put $D'' = 0$.  Then $(\pi')^{-1}(E_k') \cap \mathcal{S}_{\cA'}$ covers the branching edge $E_k'$ by degree~$D''$ as well, for each $k \in I'$. 
 For each component surface $S$ of a collection $\cS_k$, let $d(S)$ be the degree of $\pi'$ restricted to $S$.   Since $(\pi')^{-1}(E_k')$ covers $E_k'$ by degree $D'$, it follows that 
  $\displaystyle \sum_{S \subset \cS_k}d(S) = D' - D''$ for all $k \in I'$.   So, 
\begin{equation*}
\begin{split}
Dw_{1} 
&= \chi(\pi^{-1}(\cT_1)) 
= \chi(f(\pi^{-1}(\cT_1))) 
= \chi(\mathcal{S}_{\cA'}) + \sum_{k \in I'}\chi(\mathcal{S}_k)  \\
&=D'' \; \chi({\cA'}) + \sum_{k \in I'} \left(\sum_{S \subset \cS_k}d(S) \;\chi(\mathcal{P}_{kl_S}')\right) \mbox{ where $S$ covers $\cP'_{kl_S}$}\\
&\leq   D''w'_{1} + \sum_{k \in I'} \chi(\mathcal{P}_{k1}') \left( \sum_{S \subset \cS_k}d(S) \right)
\leq  D''w'_{1} + (D' - D'')w'_{1}
= D'w'_{1}.
   \end{split}
\end{equation*}

By our assumption, $Dw_{1} \geq D'w'_{1}$.  Thus we conclude that $Dw_{1} = D'w'_{1}$.  Now, each branching curve in $\cY'$ is incident to exactly $s$ connected surfaces in $f(\pi^{-1}(\cT_1)) \cup 
\dots \cup f(\pi^{-1}(\cT_s))$.  It follows that $\pi'(f(\pi^{-1}(\cT_1)) \cup \ldots \cup f(\pi^{-1}(\cT_s)))$ must have in its image at least $s$ orbifolds in $\{\cT_1', \ldots, \cT_{r+1}'\}$, and therefore $t \geq s$. Thus we have $Dw_i = D'w_i'$ for $1 \leq i \leq s=t$. Moreover, $\bigcup_{i=1}^s\pi^{-1}(\cT_i)  =\bigcup_{i=1}^s\pi'^{-1}(\cT_i')$. Now the above argument can be repeated (at 
most finitely many times) with the remaining sub-orbicomplexes of $\cO_\G$ and $\cO_{\G'}$ which are of strictly smaller 
Euler characteristic, proving the claim. 
\end{proof}

The next step in both Cases 2 and 3 is to color certain sub-orbicomplexes of $\cO_\G$ and $\cO_{\G'}$, as follows.  
Let $w_0 \in \Z^{r+1}$ be the minimal integral element in the commensurability class of $w$ and  
$w'$, so that $w = Rw_0$ and $w' = R'w_0$.  Let $C= \{c_1, \dots, c_n\}$ be the set of distinct 
values occurring in $w_0$, and assume that $c_1> \dots > c_n$, so that 
$w_0 = (c_1, \dots, c_1, c_2, \dots, c_2, \dots, c_n, \dots, c_n)$.  We call the elements of the set $C$ the \emph{colors}.   
Let $\cT_u$ and $\cT'_u$ be defined as in the proof  of Proposition~\ref{prop:ww'-commensurable} for $1 \le u \le r$. 
 We now color $\cO_\G$ and $\cO_{\G'}$ so that $\cT_u$ has color $c \in C$ 
 if $w_{u} = Rc$ and $\cT_u'$ has color $c\in C$ if $w'_{u} = R'c$.  Note that for each color $c$, there is an $m = m_c$ so that each branching edge $E_i$ (respectively, $E'_k$) is incident to $m$ branch orbifolds $\cP_{ij}$ (respectively, $\cP'_{kl}$) of color $c$.

\begin{remark}\label{rem:colors-preserved}  
 (The map $f$ preserves colors.)  The following equation is an easy consequence of the proof of 
 Proposition~\ref{prop:ww'-commensurable}:
 $$
f\left(\pi^{-1}\left(\{\cT_u \mid \cT_u \subset \cO_\G \text{ has color } c\}\right)\right) = \pi'^{-1}\left(\{\cT_u' \mid \cT_u' \subset \cO_{\G'} \text{ has color } c\}\right).$$
  \end{remark}

It also follows that orbifolds of the same color which are attached to a single edge in $\cO_\G$ (respectively, $\cO_{\G'}$)  
 are identical: 
  
\begin{lemma}\label{lem:orbifold-color}
For $i \in I$, if
$\mathcal{P}_{i{j_1}}$ and $\mathcal{P}_{i{j_2}}$ have the same color $c$, then $\chi_{i{j_1}} = \chi_{i{j_2}}$, and $\mathcal{P}_{i{j_1}}$ and $\mathcal{P}_{i{j_2}}$ are therefore identical orbifolds.  
The analogous statement holds for $\mathcal{P}'_{k{l_1}}$ and $\mathcal{P}'_{k{l_2}}$ of  the same color, where $k \in I'$.
\end{lemma}

\begin{proof}
Let $i \in I$.  If the conclusion fails, we may assume that $\chi_{i{j_1}} > \chi_{i{j_2}}$.  
Then by our assumption that  $\chi_{uj} \ge \chi_{u j'}$ whenever  $j < j'$, we see that $j_1 < j_2$, and 
$\sum_{u \in I} \chi_{u{j_1}} > \sum_{u \in I} \chi_{u{j_2}} $.  This is a contradiction as both of these sums are equal to $Rc$.
Thus $\chi_{i{j_1}} = \chi_{i{j_2}}$, and $\mathcal{P}_{i{j_1}}$ and $\mathcal{P}_{i{j_2}}$ are identical orbifolds. The proof of the second sentence is identical.  
\end{proof}

For the proofs of Cases 2 and 3,
it will be important to understand the structure of the subset $\pi^{-1} (\cA) \cup f^{-1} (\pi'^{-1} (\cA'))$ of $\cY$, which is described in the following remark. 

\begin{remark}\label{remark:preimageAA'}
(Structure of $\pi^{-1} (\cA) \cup f^{-1} (\pi'^{-1} (\cA'))$).  
Each of $\pi^{-1} (\cA)$ and $f^{-1} (\pi'^{-1} (\cA'))$ is a disjoint union of (connected) surfaces in $\cY$.
Let $\cW$ be the collection of surfaces in $\cY$ which are in both these sets, that is, $\cW$ is the collection of surfaces in $\pi^{-1} (\cA) \cap f^{-1} (\pi'^{-1} (\cA'))$.  

 Now if $S$ is a surface in $\pi^{-1}(\cA) \setminus \cW$, then since $\pi'(f(S)) \neq \cA'$, for each boundary component of $S$ there is necessarily a surface in $f^{-1}(\pi'^{-1}(\cA'))$ incident to this boundary component.  If $S'$ is one of these, then $\pi(S') \neq \cA$, since $S$ and $S'$ share a boundary curve and $S$ is in $ \pi^{-1}(\cA)$.  Hence  
every other boundary component of $S'$ also borders a component of $\pi^{-1}(\cA)$.  Continuing in this way, we see that the connected component of $\pi^{-1} (\cA) \cup f^{-1} (\pi'^{-1} (\cA'))$ which contains $S$ has no boundary.  

Thus there is a decomposition 
$\pi^{-1} (\cA) \cup f^{-1} (\pi'^{-1} (\cA')) = \cW \sqcup \cZ$, where $\cZ$ is a disjoint union of closed surfaces.  Moreover, given a component  $Z$ of  $\cZ$, the branching curves on $Z$  
 partition it as 
$Z= Z_\cA \cup Z_{\cA'}$, 
where $Z_\cA$ and 
$Z_\cA'$ are the subsets of $Z$ consisting of surfaces in $\cY$ such that 
\[ 
\pi(Z_\cA) = \cA, \quad \pi(Z_{\cA'}) \subset \cup_{i \in I}\cO_i, \quad \pi'(f(Z_{\cA'}))=  \cA', \quad\mbox{and}\quad \pi'(f(Z_\cA)) \subset \cup_{k\in I'}\cO_k'.
\]
Observe that $Z_{\cA'}$ contains at least one surface mapping into $\cO_i$ for each $i\in I$, so 
$\cZ_{\cA'}$, and similarly $\cZ_{\cA}$, is necessarily disconnected.  
We denote the set of branching curves of a component $Z$ of $\cZ$
 by the (slightly counterintuitive) notation $\partial Z$.  Then by the description above,  
$Z_\cA$ and $Z_{\cA'}$ intersect in exactly in $\partial Z$.  So we have that 
$\partial Z = \partial Z_\cA = \partial Z_{\cA'}$.  It follows that  the degree of $\pi$ restricted to 
$Z_\cA$, the degree of $\pi$ restricted to $Z_{\cA'}$, and the degree of $\pi$ restricted to $\partial Z$ are all equal. 
\end{remark}

Observe that for any component $Z$ of $\cZ$, the collection $\partial Z$ intersects $\pi^{-1}(E_i)$ for each branching edge $E_i$ of $\cO_\G$, as well as $f^{-1}({\pi'}^{-1}(E_k'))$ for each branching edge 
$E_k'$ of $\cO_{\G'}$.
In what follows it will be necessary to consider the following subset 
of $\partial Z$.

\begin{definition}\label{def:epsilonik}
Let $Z$ be a component of $\cZ$.  
For $i \in I$ and $k \in I'$, let $\epsilon_{ik}= \epsilon_{ik}(Z)$ be the collection of branching curves in $\partial Z$ which map to $E_i$ under $\pi$ and to $E_k'$ under 
$\pi'\circ f$.  That is, $\epsilon_{ik} = \pi^{-1}(E_i) \cap f^{-1}(\pi'^{-1}(E_k')) \cap \partial Z$. 
\end{definition}

The degrees of $\pi$ and $\pi'$ restricted to $\epsilon_{ik}$ satisfy the following useful property.

\begin{lemma}\label{lem:epsilonik}
Let $Z$ be a component of $\cZ$,
and let 
 $\epsilon_{ik}$ be the curves from Definition~\ref{def:epsilonik}.
 Then for each $i \in I$ and $k \in I'$, the degree of the map $\pi: \epsilon_{ik} \to E_i$ is a number $\delta_k= \delta_k(Z)$ which depends on~$k$ but not on $i$, and the degree of the map $\pi': f(\epsilon_{ik}) \to E_k'$ is a number $\delta_i' = \delta_i'(Z)$ which depends on $i$ but not on $k$.  

Moreover, if $d$ and $d'$ are the degrees of $\pi$ and $\pi'$ restricted to $\partial Z$ and $f(\partial Z)$ respectively, then 
we have the following equations:
\begin{equation} \label{eq:degree sum}
\sum_{k \in I'} \delta_k = d 
\text{ and }
\sum_{i \in I} \delta_i' = d' .
\end{equation}
\end{lemma}

\begin{proof}
We write $Z= Z_{\cA}\cup Z_{\cA'}$ as in Remark~\ref{remark:preimageAA'}. By definition, no component of 
$Z_\cA$ maps to $\cA'$, so we partition $Z_\cA$ into  a disjoint collection of possibly disconnected surfaces $\cS_k$ so that $\pi'(f(\cS_k)) $ is contained in $\cO'_k$ for each $k \in I'$.   Observe that $\partial \cS_k$ contains all curves of $\partial Z = \partial Z_\cA$ which map to  $E_k'$ under 
$\pi' \circ f$.  Thus $\partial \cS_{k} = 
\cup_{i \in I} \epsilon_{ik}$. 

On the other hand, $\pi$ maps $\cS_k$ (and hence $\partial \cS_k$) to $\cA$ with some degree, say $\delta_k$, which depends on $k$.  
Then for any $i$, the edge $E_i$ has exactly $\delta_k$ lifts in $\cS_k$, or equivalently, the degree of $\pi$ restricted to $\epsilon_{ik}$ is $\delta_k$, which is independent of~$i$.  

The left equation in \eqref{eq:degree sum} follows from the fact that $\partial Z = 
\cup_{k \in I'}
\partial \cS_k$, so the degree $d$ of $\pi$ restricted to $\partial Z$ is the sum of the degrees of $\pi$ restricted to $\partial S_k$ over all $k \in I'$.
The corresponding statements about $\delta_i'$ are proved similarly. 
\end{proof}

\subsection{Case 2}\label{sec:Case2} 

We now complete the proof of our necessary conditions in the case that $$f(\pi^{-1} (\cA)) \intcap \pi'^{-1} (\cA') = \emptyset.$$  Lemma~\ref{lem:equal-branching} establishes the first sentence of (2) in Theorem~\ref{thm:CycleGenTheta} and Proposition~\ref{prop:ww'-commensurable} proves (2)(b).  Thus it remains to show:

\begin{prop}
If $f(\pi^{-1} (\cA)) \intcap \pi'^{-1} (\cA')= \emptyset$ then condition (2)(a) holds.  
\end{prop}

\begin{proof}
Recall from the discussion preceding Remark~\ref{rem:colors-preserved} that $w= Rw_0$, where 
 the distinct entries of $w_0$ correspond to distinct colors.  
 Let $\hat w_0$ denote the vector obtained from $w_0$ by deleting the $p$th entry,  
 where $p$ is the index of the entry $\chi(W_A)$ of $w$. 
  We show below that for each $i \in I$, the vector $v_i=(\chi_{i1}, \dots, \chi_{ir})$ 
 is commensurable to $\hat w_0$.  Hence the vectors $\{ v_i \mid i \in I \}$ belong to a single commensurability class.
 
Fix  $i \in I$.  We begin by showing that there exists a number $D_i'$
such that for $1 \le j \le r$,  if $\cP_{ij}$ has color $c$, then 
\begin{equation}\label{eq:Di'}
\chi( f(\pi^{-1}(\cP_{ij}))) = D_i' R' c.  
\end{equation}
Suppose the color of $\cA'$ is $c'$ (with $c'$ possibly equal to $c$).  
Partition $f(\pi^{-1}(\cP_{ij}))$  into (possibly disconnected) subsurfaces $\cT_{\cA'}$ and $\cT_k$ such that $\pi'(\cT_{\cA'}) = \cA'$ and $\pi'(\cT_k)$ is contained in $\cO'_k$, for each $k \in I'$.  
Note that $\pi^{-1}(\cP_{ij})$ contains the full preimage of the branching edge $E_i$.  
Let $E_{ik}$ be the collection of all branching curves in $\cY$ which map to $E_i$ under $\pi$ and to $E_k'$ under $\pi'\circ f$.  
Then
$f(E_{ik})$ consists of 
$ \partial \cT_k $ together with some curves from $\partial \cT_{\cA'}$.

Let $\cW$ and $\cZ$ be as in Remark~\ref{remark:preimageAA'}.
Then $\cW$ is empty, since $f(\pi^{-1} (\cA)) \intcap \pi'^{-1} (\cA')= \emptyset$.  Thus $E_{ik}$ is the union of the curves $\epsilon_{ik} = \epsilon_{ik}(Z)$ corresponding to all components $Z$ of $\cZ$.  Then it follows from Lemma~\ref{lem:epsilonik} 
that 
the degree of $\pi'$ restricted to $f(E_{ik})$ is  a number
$D_i' = \sum_{Z\subset \cZ} \delta_i'(Z)$,  
which is independent of $k$.    

Suppose  the degree of $\pi'$ restricted to $\cT_{\cA'}$ is $D_i'' <D_i'$.  Then 
Remark~\ref{rem:colors-preserved} implies that 
for each $k$, 
the map $\pi'$ sends $\cT_k$ into 
the subset of $\cup_{l=1}^r\cP_{kl}'$ consisting of orbifolds of color $c$ by total degree $D_i' - D_i''$.
 By 
Lemma~\ref{lem:orbifold-color}, for all $\cP_{kl}'$ of color $c$, the Euler characteristic  of $W'_{kl} = \pi_1(\cP_{kl}')$ is independent of $l$, and is equal to $\chi(W'_{k{l_0}})$ for some $l_0$.  
Now $\cR_{l_0}'$ consists of one orbifold of color $c$ for each $k\in I'$, and therefore $\chi_{l_0}' = R'c = \sum_{k \in I'}\chi(W'_{k{l_0}})$.  
 Thus we obtain Equation~\eqref{eq:Di'} as follows:
$$
\chi(f(\pi^{-1}(\cP_{ij}))) = D_i'' R' c +  \sum_{k \in I'} (D_i'-D_i'') \chi(W'_{k{l_0}}) = D_i'' R'c+ (D_i'-D_i'') R'c =D_i' R'c.
$$

To complete the proof, note that 
$\chi(f(\pi^{-1}(\cP_{ij}))) = D \chi(W_{ij}) = D\chi_{ij}$, 
since $\pi$ is a degree $D$ map.  Thus 
$$
Dv_i = D(\chi_{i1}, \dots, \chi_{ir}) = (D_i'R'c_1, \dots, D_i'R'c_1, \dots, \dots, D_i'R'c_n, \dots, D_i'R'c_n) = D_i'R' \hat w_0.
$$ 
This proves the first part of (2)(a), and the second part has a similar proof.
\end{proof}

\bigskip

\subsection{Case 3}\label{sec:Case3} 

We are now in the case that 
$f(\pi^{-1} (\cA)) \intcap \pi'^{-1} (\cA') \neq \emptyset$.  We show that condition~(1) of Theorem~\ref{thm:CycleGenTheta} holds in this case by constructing new homeomorphic finite-sheeted covers of $\cO_\G$ and $\cO_{\G'}$ such that the hypothesis of Case 1 is satisfied.  More precisely, we show:

\begin{prop}\label{prop:new-maps}
Suppose 
$f(\pi^{-1} (\cA)) \intcap \pi'^{-1} (\cA') \neq \emptyset$.
Then there exist finite-sheeted, connected torsion-free covers
$\mu: \cM  \to \cO_\G$ and $\mu':  \cM'  \to \cO_{\G'}$ and a homeomorphism $\tilde f: \cM  \to \cM' $  such that 
$\tilde f(\mu^{-1}(\cA)) =\mu'^{-1}(\cA')$.
\end{prop}
Then by applying Case 1 to the maps $\mu, \mu'$, and $\tilde f$, we get:
\begin{cor}\label{cor:2}
If $f(\pi^{-1} (\cA)) \intcap \pi'^{-1} (\cA') \neq \emptyset$, then condition (1) holds.   \hfill{\qed}
\end{cor}

To prove Proposition~\ref{prop:new-maps}, we assume that Case 1 does not already hold for $\cY$, $\cY'$, and  $f$.  
We observe:

\begin{obs}
Since $f(\pi^{-1} (\cA)) \intcap \pi'^{-1} (\cA') \neq \emptyset$, Remark~\ref{rem:colors-preserved} implies that $\cA$ and $\cA'$ have the same color, say $a$.  As Case 1 does not hold, each of $\cO_\G$ and $\cO_{\G'}$ also has orbifolds besides $\cA$ and~$\cA'$ of color $a$.  
\end{obs}

\subsubsection*{Idea of the construction of $\cM$ and $\cM'$}  
As described in Remark~\ref{remark:preimageAA'}, we have a decomposition 
$\pi^{-1} (\cA) \cap f^{-1} (\pi'^{-1} (\cA')) = \cW \sqcup \cZ$.
 The hypothesis of Case 1 holds on $\cW$ and  fails on $\cZ$.  Our assumption that 
 Case 1 does not hold for  $\cY$, $\cY'$, and $f$ implies that  $\cZ$ is nonempty.  
As in Remark~\ref{remark:preimageAA'}, each component 
$Z$  of $\cZ$ has a decomposition $Z = Z_\cA \cup Z_{\cA'}$ such that 
\[ 
\pi(Z_\cA) = \cA, \quad \pi(Z_{\cA'}) \subset \cup_{i \in I}\cO_i, \quad \pi'(f(Z_{\cA'}))=  \cA', \quad\mbox{and}\quad \pi'(f(Z_\cA)) \subset \cup_{k\in I'}\cO_k'.
\]
The key to constructing the new covers is the fact, which we establish below, that $\chi(\cZ_\cA) = \chi(\cZ_{\cA'})$.

 If $\cZ_\cA$ and $\cZ_\cA'$ were connected, then by redefining $f$ to interchange its images on these two surfaces, we would get a homeomorphism such that the hypothesis of Case~1 is satisfied.  
However, as observed in Remark~\ref{remark:preimageAA'}, the surfaces $\cZ_{\cA}$ 
and $\cZ_{\cA'}$ are never connected, so it is not clear that there is a homeomorphism from $\cZ_{\cA}$ to $f(\cZ_{\cA'})$.  Trying to modify $\pi'$ poses similar difficulties, as 
it is not clear that $\pi'$ can be consistently redefined so that the individual pieces of $f(\cZ_{\cA})$ each map to $\cA'$, and the individual pieces of 
$f(\cZ_{\cA'})$ each map to some $\cP'_{kl}$.  Thus a more subtle approach is required.  

Our strategy is to cut out the interiors of 
$\cZ_\cA$ and $\cZ_{\cA'}$ from $\cY$ (and correspondingly cut out $f(\cZ_\cA)$ and $f(\cZ_{\cA'})$ from $\cY'$) and replace them with different surfaces to get new (homeomorphic) surface amalgams $\cM$ and $\cM'$ covering $\cO_\G$ and $\cO_{\G'}$ respectively.  
We will replace $Z_\cA$ and $f(Z_\cA)$ with homeomorphic connected surfaces with boundary which cover both $\cA$ and $\cA'$. The Euler characteristic of these surfaces will be $\chi(\cZ_\cA) = \chi(\cZ_{\cA'})$. 
We will replace $Z_{\cA'}$  and $f(Z_{\cA'})$ by a collection of homeomorphic connected surfaces with boundary, so that each such homeomorphic pair covers some orbifold $\cP_{ij}$ of color $a$ in $\cO_\G$ and also some orbifold $\cP_{kl}'$ of color $a$ in $\cO_{\G'}$.  
The crucial ingredient for this step is Lemma~\ref{lem:epsilonik}, which allows us to derive the necessary Euler characteristic equations required to construct such pairs.  
After all these replacements, the new covering maps need to extend the already existing covering maps on $\partial Z = \partial Z_\cA = \partial Z_{\cA'}$ and $f(\partial Z)$.

We construct the replacement surfaces in Lemmas~\ref{lem:TA} and~\ref{lem:Tik} by passing to our torsion-free covers from Section~\ref{sec:tfcover} and then using 
Lemma~\ref{lem:neumann} for constructing further covers.   We use Lemma~\ref{lem:epsilonik} to prove the replacement surfaces are homeomorphic. 
 The parity condition in 
Lemma~\ref{lem:neumann} forces the covering maps on the replacement surfaces to have twice the degrees of the restrictions of  $\pi$ and $\pi'$  to $Z$ and $f(Z)$.  Thus to glue in these replacement surfaces, we need to pass to degree 2 covers of $\cY$ and $\cY'$.  

A further difficulty is posed by the fact that each connected replacement surface of the second type maps to a single $\cP_{ij}$ and  $\cP_{kl}'$, but it may be replacing a collection of surfaces in $Z_{\cA'}$ whose image 
in either $\cO_\G$ or $\cO_{\G'}$
contains all orbifolds  of color $a$ at either $E_i$ or $E_k'$. This would make it  impossible to extend at least one of $\pi$ or $\pi'$ to the new surface.  To fix this we pass to a further degree $m^2$ cover, where $m$ is the number of orbifolds of color $a$ in any $\cO_i$, $i \in I$.  

Thus we first construct spaces $\bar \cY$ and $\bar \cY'$ consisting of $2m^2$ disjoint copies of $\cY$ and $\cY'$.  We define new covering maps $\bar \pi$ and $\bar \pi'$ to $\cO_\G$ and $ \cO_{\G'}$ 
respectively, designed to make the eventual extensions to $\cM$ and $\cM'$ possible.  Then we cut out all $2m^2$ copies of $\cZ$ and replace them with $m^2$ copies of each of the replacement surfaces in such a way that $\bar \pi$ and $\bar \pi'$ extend.  The fact that Case 1 holds follows from the construction.  

We now give the details of the construction. 

\subsubsection*{Construction of the replacement surfaces}
Let $Z$ be a component of $\cZ$.  The following lemma defines the surfaces which will be used to replace copies of $Z_\cA$.  

\begin{lemma}\label{lem:TA}
Let  $d$ and $d'$ be the degrees of $\pi$ and $\pi'$ restricted to $\partial Z = \partial Z_\cA$  and $f(\partial Z)$   respectively.  Then there exist connected surfaces $T$ and $T'$ and orbifold covering maps $\alpha: T \to \cA$ and $\alpha': T' \to \cA'$ of degree $2d$ and $2d'$ respectively such that the following properties hold.  
\begin{enumerate}
\item 
There is a partition 
$\partial T = \partial T^{+} \sqcup \partial T^{-}$ and homeomorphisms
$\beta^\pm:   \partial T^{\pm}\to \partial Z$ such that 
$
\alpha  = \pi \circ \beta^\pm$.

\item 
There is a partition 
$\partial T' = \partial T'^{+} \sqcup \partial T'^{-}$ and homeomorphisms
$\beta'^\pm:   \partial T'^{\pm}\to f(\partial Z)$ such that 
$
\alpha'  = \pi' \circ \beta'^\pm$.

\item There is a homeomorphism $\psi : T \to T'$ whose restriction to 
$\partial T$ consists of the maps $({\beta'}^{\pm})^{-1} \circ f \circ \beta^{\pm}$. 

\end{enumerate}
\end{lemma}
\begin{proof}
Recall that $\pi = \rho \circ \eta$, where $\rho:\cX \to \cO_\G$ is the degree 16 torsion-free covering space from Section~\ref{sec:tfcover}. If $S_\cA = \rho^{-1}(\cA)$, then 
$\eta$ maps the possibly disconnected surface $Z_{\cA}$ to the connected surface $S_\cA $ by degree $d/ 16 \in\Z$.  We will use Lemma~\ref{lem:neumann} to construct a connected cover $T$ of $S_\cA$ of degree $d/16$.

Let $\partial T^{+}$ and $\partial T^{-}$ be sets of circles, each equal to a copy of $\partial Z$, and let $\beta^+$ and $\beta^-$ be homeomorphisms from 
$\partial T^{+}$ and $\partial T^{-}$ respectively to $\partial Z$.  Then 
$\eta \circ \beta^\pm$ a covering map from each circle in $\partial T^{\pm}$ to a circle in $\partial S_\cA$.  Moreover, the sum of the degrees of $\eta \circ \beta^+$ and $\eta \circ \beta^-$ over all circles that cover a given one in $\partial S_\cA$ is $2(d/16) \in 2\Z$ by construction (since $\eta$ is a covering map of degree $d/16$ on $\partial Z$).   

Now by
Lemma~\ref{lem:neumann}, there 
exists a connected surface $T$ and a  covering map $ \bar \alpha:T\to S_\cA$ of degree $2d/16$ which extends the maps 
$\eta \circ \beta^\pm$ on $\partial T^{+} \sqcup \partial T^{-}$, because both the prescribed number of boundary components, namely $2\card{(\partial Z)}$, and the number $2(d/16) \chi(S_\cA)$ are even.  
Hence $\alpha = \rho \circ \bar \alpha $ is a degree $2d$ covering map from $T$ to $\cA$.  This proves (1), and the map $\alpha': T' \to \cA'$ of degree $2d'$ satisfying (2) is constructed similarly.  

To prove (3) we first show that $T$ and $T'$ have the same Euler characteristic.  By construction $\chi(T) = 2d \chi(W_A)$ and $\chi(T') = 2d' \chi(W_{A'})$, and we now show that these are equal.  

As in the proof of Lemma~\ref{lem:epsilonik}, we partition $Z_\cA$ into unions of connected surfaces $\cS_{k}$ for $k \in I'$, where
$\pi'(f(\cS_k))$ is contained in $\cO'_k$.  Observe that the set of boundary components of $\cS_k$ is exactly 
$\cup_{i \in I}\epsilon_{ik}$, 
and by Lemma~\ref{lem:epsilonik}, the degree of $\pi'$ restricted to 
$f(\epsilon_{ik})$ is $\delta_i'$.  Now the degree of $\pi'$ restricted to $f(\cS_k)$ is exactly the sum of the degrees of $\pi'$ restricted to $f(\epsilon_{ik})$ over all $i \in I$, and this sum is $\sum_{i \in I} \delta_i' = d'$ (again by Lemma~\ref{lem:epsilonik}). Thus we have 
\begin{equation}\label{eq:Sk}
\chi(f(\cS_k)) = d'\chi(W_{kt}'), 
\end{equation}
where $W'_{kt} = \pi_1(\cP_{kt}')$ for $\cP_{kt}'$ a branch orbifold of color $a$.   By Lemma~\ref{lem:orbifold-color} all branch orbifolds of color~$a$ are identical.  So  $f(\cS_k)$ covers an orbifold 
identical to $\cP_{kt}'$ with total degree $d'$ for each $1 \le k \le N'$.  Thus we have 
\begin{equation}\label{eq:AA'}
d\, \chi(W_A) = \chi( Z_\cA) = \sum_{k \in I'} \chi (f(\cS_k)) = d'\; \sum_{k \in I'} \chi(W_{kt}') = 
d'\, \chi'_t
=d'\, \chi(W_{A'})
\end{equation}
where the last equality comes from $\chi'_t = \chi(\cR_t')$, $\chi(W_{A'}) = \chi(\pi_1(\cA'))$,  and the fact that $\cA'$ and~$\cR_t'$ have the same color $a$. 

Now $T$ and $T'$ have the same Euler characteristic and the same number of boundary components, so they are homeomorphic.  Moreover, we can choose a homeomorphism $\psi$ between them which extends the homeomorphism $({\beta'}^{\pm})^{-1} \circ f \circ \beta^{\pm}$ on their boundaries.
\end{proof}

The following lemma constructs pairs of homeomorphic surfaces which map to fixed pairs 
of orbifolds $\cP_{is}$ and $\cP'_{kt}$.  
Collections of surfaces of this form will be used to replace copies of~$Z_{\cA'}$.  

\begin{lemma}\label{lem:Tik}
Let $Z$ be a component of $\cZ$, and let $\epsilon_{ik}, \delta_k$ and $\delta_i'$ be as in 
Lemma~\ref{lem:epsilonik}. 
For each~$i$ and $k$ such that $\epsilon_{ik}$ is nonempty, and for each $s$ and $t$ such that 
$\cP_{is}$ and $\cP'_{kt}$ have color~$a$, 
there
exist connected surfaces $S$ and $S'$ and orbifold covering maps  
$\zeta: S \to \cP_{is}$ and $\zeta': S' \to \cP'_{kt}$ of degree $2\delta_k$ and $2\delta_i'$ respectively, such that  the following properties hold.

\begin{enumerate}
\item 
There is 
a partition 
$\partial S = \partial S^{+} \sqcup \partial S^{-}$ and homeomorphisms
$\xi^{ \pm}:   \partial S^{\pm}\to \partial Z$ such that 
$\zeta  = \pi \circ \xi^{\pm}$.

\item There is 
a partition 
$\partial S' = \partial S'^{+} \sqcup \partial S'^{-}$ and homeomorphisms
$\xi'^{ \pm}:   \partial S'^{\pm}\to f(\partial Z$) such that 
$\zeta'  = \pi \circ \xi'^{\pm}$.

\item There is a homeomorphism $\varphi : S \to S'$ whose restriction to 
$\partial S$ consists of the maps \newline $({\xi'}^{\pm})^{-1} \circ f \circ \xi^{\pm}$. 

\end{enumerate}
\end{lemma}

\begin{proof}
As in the proof of Lemma~\ref{lem:TA}, we construct covers $S \to \cP_{is} $ and 
$S' \to \cP_{kt}'$ of the required degrees using Lemma~\ref{lem:neumann}.  By construction, $S$ and $S'$ have homeomorphic boundaries.  
We show below that $\chi(S)=\chi(S')$.  It follows that $S$ and $S'$ are homeomorphic, and we choose a homeomorphism $\varphi$ between them which extends the homeomorphism 
$({\xi'}^{\pm})^{-1} \circ f \circ \xi^{\pm}$ on their boundaries.  

Since $\chi(S) = \delta_k\chi(W_{is}) = \delta_k \chi_{is}$  and $\chi(S') = \delta_i'\chi(W'_{kt}) = \delta'_i \chi'_{kt}$, it is enough to show that $\delta_k\chi_{is} = \delta_i'\chi'_{kt}$.  This
 follows from the two facts below, where $d$ and $d'$ are as in Lemma~\ref{lem:epsilonik}.
\begin{enumerate}[(i)]
\item $d\chi_s = d' \chi'_{t}$.

\item 
$d'\chi'_{kt} = \delta_k \chi_s
\text{ and }
d\chi_{is} = \delta_i' \chi'_t$.

\end{enumerate}
More precisely, facts (i) and (ii) give 
$$
\frac{d'}{d} = \frac{\chi_s}{\chi'_t} = \left(\frac{d' \chi'_{kt}}{\delta_k}\right)
\left(\frac{\delta_i'}{d\chi_{is}}\right) = \left(\frac{d'}{d}\right)\left(\frac{\delta_i'\,\chi'_{kt}}
{\delta_k  \,\chi_{is}} \right).
$$
This yields $\delta_k\chi_{is} = \delta_i'\chi'_{kt}$
as desired.

Equation~\eqref{eq:AA'} in the proof of Lemma~\ref{lem:TA} shows that $d \chi(W_A) = d'\chi(W_{A'})$. 
Now (i) follows from the fact that $\chi_s = \chi(W_A)$ and 
$\chi'_t= \chi (W_{A'})$. 
To prove (ii), we recall from the proof of Lemma~\ref{lem:epsilonik} that $\cS_k$ covers $\cA$ with degree $\delta_k$.  This, together with equation~\eqref{eq:Sk} in the proof of Lemma~\ref{lem:TA}  and the fact that $\chi(W_A) = \chi_s$ 
gives:
$$
\delta_k \chi_s = \delta_k \chi(W_A) = \chi(\cS_k) = \chi (f(\cS_k)) = d'\chi'_{kt}. 
$$
The proof of the second equation in (ii) is similar. 
\end{proof}

We are now ready to construct the new covers 
$\cM$ and $\cM'$.

\begin{proof}[Proof of Proposition~\ref{prop:new-maps}]
Suppose that each $\cO_i$ and $\cO_k'$ has $m$ branch orbifolds of color $a$.  Then we 
will construct homeomorphic surface amalgam covers $\cM$ and $\cM'$ of  $\cO_\G$ and $\cO_{\G'}$ of degrees $2m^2 D$ and $2m^2D'$ respectively. 

\subsubsection*{New homeomorphic covering spaces $\bar \cY$ and $\bar \cY'$}
Take $2m^2$ disjoint copies of $\cY$ and $\cY'$ to obtain $\bar \cY$ and~$\bar \cY'$.  More precisely, for  $1 \le p,q \le m$, 
let $\tau_{pq}^\pm : \cY \to\cY_{pq}^\pm $ and ${\tau'}_{pq}^\pm :  \cY\to {\cY'}_{pq}^\pm  $ be homeomorphisms, and set:
$$
\bar \cY = \left(\bigsqcup_{p,q=1}^m \cY_{pq}^+ \right) \; \bigsqcup \;\left(
\bigsqcup_{p,q=1}^m \cY_{pq}^- \right)
\quad \text{and} \quad
\bar \cY' =\left(\bigsqcup_{p,q=1}^m {\cY'}_{pq}^+ \right) \;\bigsqcup \;\left(
\bigsqcup_{p,q=1}^m {\cY'}_{pq}^- \right).
$$

We define the map $\bar \pi:  \bar \cY\to  \cO_\G$ as follows.  
First define an isometry $\sigma: \cO_\G \to \cO_\G$ so that if  
$1 \leq j_1 <  \dots <  j_m \leq r$ are the indices of the color $a$ orbifolds $\cR_j$ 
in $\cO_{\G}$, then $ \sigma (\cR_{j_u}) = \cR_{j_{u+1}}$, where $u+1$ is taken mod $m$, and $\sigma$ is the identity elsewhere.
  Then $\bar \pi$ is defined on $\cY_{pq}^\pm$ by
$\bar \pi = \sigma^{p}\circ \pi \circ (\tau_{pq}^\pm)^{-1}$.  Clearly the resulting map $\bar \pi: \bar\cY \to \cO_\G$ is an orbifold covering map of degree $2m^2D$.

Similarly we choose the analogous isometry $\sigma': \cO_\G' \to 
\cO_{\G}'$ and define the corresponding covering map $\bar \pi': \cY' \to \cO_\G'$ of degree $2m^2D'$, by setting $\bar \pi' = {\sigma' }^{q} \circ \pi' \circ({\tau'}_{pq}^{\pm})^{-1}$ on ${\cY'}_{pq}^\pm$.
Finally, we define a homeomorphism  $\bar f: \bar \cY \to \bar \cY'$ by setting $\bar f ={{\tau'}_{pq}^\pm} 
 \circ f \circ ({\tau_{pq}^\pm})^{-1}$ on $\cY_{pq}^\pm$. 

\subsubsection*{Construction of $\cM$ and $\cM'$}
Let $\cZ \subset \cY$ be as in Remark~\ref{remark:preimageAA'}.  
Given a component $Z$ of $\cZ \subset \cY$, let $\partial Z$ denote the branching curves it contains.  Let $Z_{pq}^\pm = \tau_{pq}^\pm(Z)$
 and $\partial Z^\pm_{pq}= \tau_{pq}^\pm(\partial Z)$ denote the corresponding objects in $\cY^\pm_{pq}$.  
 Let $\cM_\cZ$ be the subspace of $\bar \cY$ obtained by cutting out the interiors of 
  all $2m^2$ copies of $\cZ$.  More precisely,
$$ 
\cM_\cZ = 
\bar \cY\; 
\setminus \; \left(\; \bigsqcup_{Z\in \cZ} \;\bigsqcup_{p,q=1}^m\left(Z^\pm_{pq} \setminus 
\partial Z^\pm_{pq}  \right)\;\right).
$$

The space $\cM$ is obtained by gluing  a  new collection of surfaces to $\cM_\cZ$ in such a way that 
each branching curve in $\cM$ has a neighborhood homeomorphic to a neighborhood of the corresponding curve in $\cY$, and moreover there is an extension of $\bar \pi|_{\cM_{\cZ}}$ to a map $\mu$ on $\cM$, such that $\bar \pi $ and $\mu$ agree on such neighborhoods of branching curves.  The space $\cM'$ is obtained similarly.  The surfaces are glued as follows. 

Fix a component $Z$ of $\cZ$.  We glue two types of surfaces to $\cM_\cZ$ and  $f(\cM_\cZ)$
along the curves $\sqcup_{p,q=1}^m \partial 
Z_{pq}^\pm$ 
and $\sqcup_{p,q=1}^m \bar f(\partial 
Z_{pq}^\pm)$ respectively.  Let $Z = Z_\cA \sqcup Z_{\cA'}$ as in Remark~\ref{remark:preimageAA'}.

\begin{enumerate}
\item (Replacements for $Z_\cA$.)
Construct $m^2$ copies each of the surfaces $T$  and $T'$ from Lemma~\ref{lem:TA}, indexed by $1 \le p, q \le m$.  
Let $\beta_{pq}^\pm: \partial T_{{pq}} \to \partial Z$ and ${\beta'}_{pq}^\pm: \partial T'_{{pq}} \to f(\partial Z)$ denote the maps from the lemma.  Then $T_{{pq}}$ is attached to $\cM_\cZ$ 
by identifying $\partial T_{{pq}}^+$ to $\partial Z_{pq}^+$ via $\tau_{pq}^+\circ \beta^+_{pq}$, and $\partial T_{{pq}}^-$ to $\partial Z_{pq}^-$ via $\tau_{pq}^-\circ \beta^-_{pq}$.
Similarly 
$T'_{{pq}}$ is attached to $\bar f(\partial Z^\pm_{pq})$ in $\bar f(\cM_\cZ)$
 via the maps ${\tau'_{pq}}^\pm \circ \beta\pm+_{pq}$.
Note that by Lemma~\ref{lem:TA} there are covering maps $\alpha_{pq}: T_{{pq}}
\to \cA$ and $\alpha'_{pq}:T_{{pq}}' \to \cA'$.

\bigskip 

\item (Replacements for $Z_{\cA'}$.)
For each non-empty collection of curves $\epsilon_{ik}$ in $Z$, we glue in $m^2$ copies of a single surface.  Thus $Z_{\cA'}$ is replaced by the union of such surfaces for all pairs~$i, k$ such that $\epsilon_{ik}$ is nonempty.  

Fix such an $\epsilon_{ik}$. 
Construct $m^2$ copies of the surfaces $S$ and $S'$ described in Lemma~\ref{lem:Tik}, denoted by $S_{pq}$ and $S_{pq}'$, with $1\le p,q\le m$. 
 Corresponding to $p$ and $q$, choose $s= j_p$ and $t=l_q$ in the statement of 
Lemma~\ref{lem:Tik}, and let $\zeta_{pq}: S_{pq} \to  \cP_{i{j_p}}$ and $\zeta_{pq}':S'_{pq} \to 
\cP_{k{l_q}}'$ be the covering maps given by the lemma.  Let $\xi_{pq}^\pm: \partial S_{pq} \to \partial Z
$  and ${\xi'}_{pq}^\pm: \partial S'_{pq} \to f(\partial Z)
$ be the associated maps on the boundaries. 

We cannot simply attach $S_{pq}$ to $\partial Z^\pm_{pq}$ as in the previous case, because we want 
to define a 
covering map $\mu$ on $\cM$ which agrees with $\bar \pi$ on a neighborhood of each branching curve in $\bar \cY$.  Now $\mu(S_{pq})$
will have to be a single orbifold $\cP_{i{j_s}}$ for some $1 \le s \le m$.  On the other hand,~$S_{pq}$ is replacing 
a disjoint collection of surfaces in $\cY^+_{pq} \cup \cY_{pq}^-$ (namely all surfaces in 
$\tau_{pq}^{\pm}(Z_{\cA'})$ with boundary in $\tau_{pq}^{\pm}(\epsilon_{ik})$).  A priori the image of the union of these surfaces under~$\bar \pi$ could be all 
the $\cP_{ij}$ of color $a$. Then for any choice of $s$ there would be some branching curve incident to two surfaces mapping to $\cP_{i{j_s}}$, and $\mu$ would not be an orbifold covering.  
To take care of this issue, we  define the attaching maps as follows. 

Let $\gamma$ be a boundary curve in $\partial S_{pq}^+$. 
Then $\xi_{pq}^+(\gamma)$ is a branching curve of $ \cY$, which is incident to two subsurfaces  in $Z$, say $S_1$ in $Z_{\cA'}$ and $S_2$ in $Z_\cA$. Then 
$\pi(S_1) = \cP_{ij_u}$ for some $1 \le u \le m$ and $\pi' (f(S_2)) = \cP'_{kl_v}$ for some $1 \le v \le m$.  
We then attach the $\gamma$ boundary of~$S_{pq}$ to $\cY_{(p-u)(q-v)}^+\cap \cM_Z$ via the map $
\tau_{(p-u)(q-v)}^+ \circ \xi^+_{pq}$.  Note that $u$ and $v$ depend on~$\gamma$.  
Now if $\varphi$ is the homeomorphism from $S_{pq}$ to $S_{pq}'$ given by Lemma~\ref{lem:Tik}, then we attach the~$\varphi(\gamma)$ boundary of $S_{pq}'$ to $\cY_{(p-u)(q-v)} \cap f(\cM_\cZ)$ via the map  
$\bar f \circ \tau_{(p-u)(q-v)}^+ \circ \xi^+_{pq} \circ \varphi^{-1}$.
We repeat this procedure for each $\gamma$ in $\partial S_{pq}^\pm$ to attach all the boundary components of $S_{pq}$ to $\cM_Z$ and the corresponding boundary components of $S'_{pq}$ to $f(\cM_\cZ)$.  

We repeat the above procedure for each $1 \le p \le m$ and $1 \le q\le m$. This completes the replacement of surfaces corresponding to the copies of the curves $\epsilon_{ik}$.  
We now do the same for each $i$ and $k$ such that $\epsilon_{ik}$ is nonempty.  
\end{enumerate}

We claim that at the end of steps (1) and (2), each curve in $\sqcup_{p,q=1}^m \partial 
Z_{pq}^\pm$ is incident to exactly one surface from (1) and one surface from (2).  The first claim is easy.  Now suppose (without loss of generality) that $\gamma \in Z_{st}^+$ is incident to two surfaces, say along $\gamma_1$ in  $S_{p_1q_1}$ and $\gamma_2$ in $S_{p_2q_2}$. Then $\xi_{p_1q_1}^+(\gamma_1)= \xi^+_{p_2q_2}(\gamma_2)
= ({\tau^+}_{st})^{-1}(\gamma)$ in $\partial Z$, and if $u$ and $v$ are defined as in (2), then 
$s=p_1-u=p_2-u$. Thus $p_1=p_2$, and similarly $q_1 = q_2$.

Finally, we repeat steps (1) and (2) above for each component $Z$ of $\cZ$ to get $\cM$ and $\cM'$. 
Note that since the $\partial Z$'s for different $Z$'s partition $\partial \cZ$, we see that each branching curve in $\sqcup_{Z\in \cZ} \sqcup_{p,q} \partial Z_{pq}^\pm$ is incident to exactly one surface from (1) and one from (2).  In fact, if $\gamma \in \partial Z_{pq}^\pm$ for some $Z, p$ and $q$, then there exist neighborhoods $N_\cY(\gamma)$ and 
$N_\cM(\gamma)$ in $\cY$ and $\cM$ respectively, and  a homeomorphism $h_\gamma: N_\cY(\gamma)\to
N_\cM(\gamma)$, which is the identity on $\cM_\cZ \cap N_\cY(\gamma)$, maps $N_\cY(\gamma) \cap \cZ_\cA$ to the part of~$N_\cM$ coming from the surface from (1), and maps 
$N_\cY(\gamma) \cap \cZ_{\cA'}$ to the part of $N_\cM$ coming from the surface from (2).  

\subsubsection*{The covering maps $\mu$ and $\mu'$, and the homeomorphism $\tilde f$}
Define maps $\mu$  and $\mu'$ as follows, where $1 \le p,q \le m$:
$$\mu =
\begin{cases}
\bar \pi & \text{ on } \cM_\cZ\\
\alpha_{pq}& \text{ on each } T_{{pq}}\\
\zeta_{uv}& \text{ on each } S_{pq} 
 \end {cases} 
 \quad \text{ and }\quad
 \mu' =
\begin{cases}
\bar \pi' & \text{ on } f(\cM_\cZ)\\
\alpha'_{pq}& \text{ on each } T'_{{pq}}\\
\zeta_{pq}& \text{ on each } S'_{pq}. 
 \end {cases}
$$

The overlap between the domains of the various maps is exactly the set of branching curves 
$\sqcup_{Z\in \cZ} \sqcup_{p,q} \partial Z_{pq}^\pm$.   Thus it follows from the definitions of the gluing maps, the definition of $\bar \pi$, and items (i) and (ii) in 
Lemmas~\ref{lem:TA} and~\ref{lem:Tik} that $\mu$ and $\mu'$ are  well-defined.  

We already know that $\mu$ is a covering map away from the branching curves in  $\sqcup_{Z\in \cZ} \sqcup_{p,q} \partial Z_{pq}^\pm$.  If~$\gamma$ is such a branching curve, then the above description of the surfaces incident to $\gamma$, together with the fact that $\mu$ and $\bar \pi$ agree on $\gamma$, shows that there is a homeomorphism $h$ from some neighborhood of $\gamma$  in $\cM$ to a neighborhood of $\gamma$ in $\cY$ such that $\cM = \bar \pi \circ h$.  It follows that $\mu$ is an orbifold covering map in a neighborhood of each branching curve.  Thus $\mu$, and similarly $\mu'$, are orbifold covering maps.  

 Finally, define 
$\tilde f$ by setting it equal to $\bar f$ on $\cM_\cZ$, to $\psi_{pq}$ on each surface of the form $T_{{pq}}$ as in (1) above, and to
$\varphi_{pq}$  on each surface of the form $S_{pq}$ as in (2) above, where $\psi_{pq}$ and 
$\varphi_{pq}$ 
are the homeomorphisms from 
Lemmas~\ref{lem:TA} and~\ref{lem:Tik}.  Then $\tilde f: \cM \to \cM'$  is a well-defined homeomorphism, and we have 
that $\tilde f (\mu^{-1} (\cA)) = \mu'^{-1}(\cA)'$. 

If the covers $\cM$ and $\cM'$ constructed above are not connected, we consider the restriction of $\tilde f$ to a single connected component of $\cM$ to get the desired homeomorphic connected covers satisfying the hypothesis of Case 1. 
\end{proof}

This completes the proof of Proposition~\ref{prop:ness-cond}.

\section{Sufficient conditions for cycles of generalized $\Theta$-graphs} \label{sec:CycleGenThetaSuff}

In this section we establish sufficient conditions 
for the commensurability of right-angled Coxeter groups defined by cycles of generalized $\Theta$-graphs.  We show that conditions (1) and (2) from Theorem~\ref{thm:CycleGenTheta} are sufficient in Sections~\ref{sec:cond1} and~\ref{sec:cond2}, respectively.

\subsection{Sufficiency of condition (1)}\label{sec:cond1}

We now prove the following result.

 \begin{prop}\label{prop:cond1}
Let $W = W_\G$ and $W' = W_{\G'}$ be as in Theorem~\ref{thm:CycleGenTheta}.  If condition (1) holds, then~$W$ and $W'$ are commensurable. 
 \end{prop}
 
By the construction given in Section~\ref{sec:tfcover}, we have surface amalgams $\cX$ and $\cX'$ which are finite-sheeted covers of $\cO_\G$ and $\cO_{\G'}$, respectively.   To prove Proposition~\ref{prop:cond1}, we will construct homeomorphic surface amalgams $\cY$ and $\cY'$ so that $\cY$ is a finite-sheeted cover of $\cX$ and $\cY'$ is a finite-sheeted cover of $\cX'$.  It follows that $\pi_1(\cY) \cong \pi_1(\cY')$ has finite index in both $W = \pi_1(\cO_\G)$ and $W' = \pi_1(\cO_{\G'})$, hence $W$ and $W'$ are abstractly commensurable. 

Let $\Lambda$ and $\Lambda'$ be the JSJ graphs of $W$ and $W'$, respectively, and recall from Section~\ref{sec:tfcover} that the JSJ graphs of $\cX$ and $\cX'$ are the half-covers $\half(\Lambda)$ and $\half(\Lambda')$.  The surface amalgams $\cY$ and $\cY'$ will have isomorphic JSJ graphs $\Psi$ and $\Psi'$, with $\Psi$ and $\Psi'$ being half-covers of $\half(\Lambda)$ and $\half(\Lambda')$, respectively.  After introducing some notation and constants for use in the proof, in Sections~\ref{sec:notation1} and~\ref{sec:constants} respectively, we construct $\Psi$ and $\Psi'$ in Section~\ref{sec:Psi}.  We then construct common covers of the surfaces in $\cX$ and $\cX'$ in Sections~\ref{sec:T} and~\ref{sec:Tpj}.  Finally, in Section~\ref{sec:Y}, we describe how these surfaces are glued together to form homeomorphic surface amalgams $\cY$ and $\cY'$ so that there are finite-sheeted coverings $\cY \to \cX$ and $\cY' \to \cX'$.

\subsubsection{Notation}\label{sec:notation1}

We first establish some notation to be used throughout the proof of Proposition~\ref{prop:cond1}.   This includes labeling the vertices of $\half(\Lambda)$ and $\half(\Lambda')$. 

Using~(1)(a) of Theorem~\ref{thm:CycleGenTheta}, let $\cV_1,\dots,\cV_M$ be the commensurability classes of the Euler characteristic vectors $\{ v_i \mid r_i \geq 2 \}$ and $\{ v_k' \mid r_k \geq 2 \}$.  For each $1 \leq p \leq M$, let  $r_p \geq 2$ be the number of entries of each vector in $\cV_p$, let $w_p$ be the minimal integral element of the commensurability class $\cV_p$, let $\{ v_{pq} \mid 1 \leq q \leq N_p \}$ be the vectors from $\{ v_i \mid r_i \geq 2 \}$ which lie in $\cV_p$, and let $\{ v'_{pq'} \mid 1 \leq q' \leq N'_p \}$ be the vectors from $\{ v_k' \mid r_k' \geq 2 \}$ which lie in $\cV_p$.   For each $p$, $q$, and $q'$, let $R_{pq}$ and $R'_{pq'}$ be the (unique) rationals so that $v_{pq} = R_{pq}w_p$ and $v'_{pq'} = R'_{pq'}w_p$.

\begin{center}
 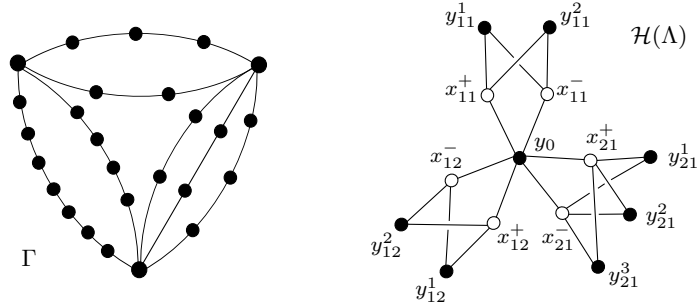
\begin{figure}[ht]
  \begin{overpic}[scale=.5, tics=5]%
{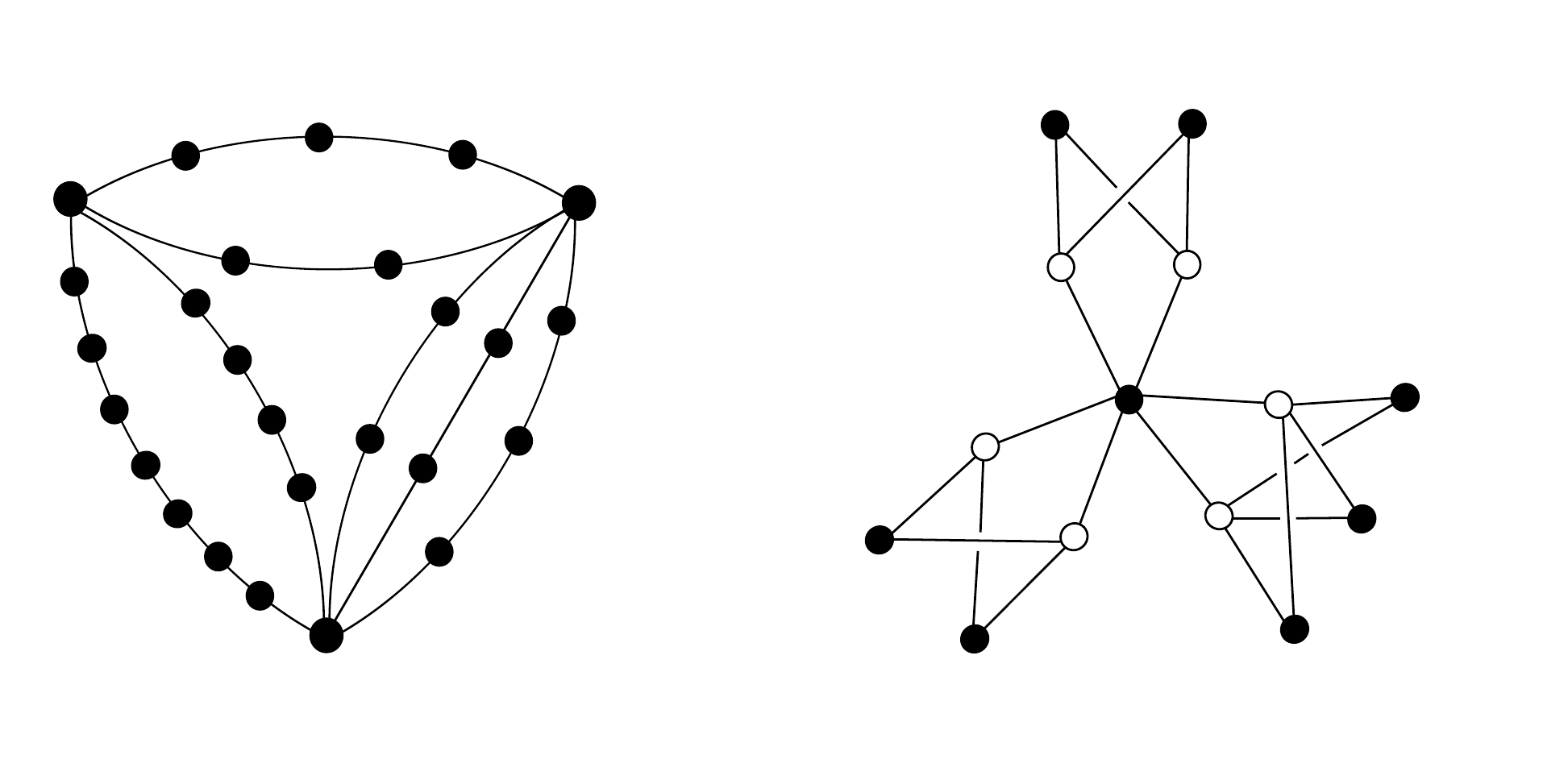}
 \put(74, 26){\scriptsize $y_0$}
 \put(5,10){\footnotesize $\G$}
 \put(62,33){\scriptsize{$x_{11}^+$}}
 \put(77,33){\scriptsize $ x_{11}^-$}
 \put(62,43){\scriptsize $y_{11}^1$}
 \put(77.5,43){\scriptsize $y_{11}^2$}
 \put(81,26.5){\scriptsize $x_{21}^+$}
 \put(75,13){\scriptsize $x_{21}^-$}
 \put(92,24){\scriptsize $y_{21}^1$}
 \put(89,16){\scriptsize $y_{21}^2$}
 \put(84,8){\scriptsize $y_{21}^3$}
 \put(69,13){\scriptsize $x_{12}^+$}
 \put(60,24){\scriptsize $x_{12}^-$}
 \put(58,6){\scriptsize $y_{12}^1$}
 \put(52,12){\scriptsize $y_{12}^2$}
 \put(87,40){\footnotesize $\half(\Lambda)$}
 \end{overpic}
 \caption{{\Small The graph $\G$ on the left is an example of a cycle of generalized $\Theta$-graphs in which there are two commensurability classes of Euler characteristic vectors for $W_{\G}$. On the right, the graph $\half(\Lambda)$ is the JSJ graph of the degree $16$ torsion-free cover $\cX$ of $\cO_\G$  constructed in Section~\ref{sec:tfcover}. The vertices of $\half(\Lambda)$ are labeled as described in Section \ref{sec:notation1}. }}
 \label{figure:suff_cond_1} 
  \end{figure}
 \end{center}

We use the notation from the previous paragraph to label the vertices of the graphs $\half(\Lambda)$ and~$\half(\Lambda')$, as follows.  An example of the labeling is given in Figure \ref{figure:suff_cond_1}.  The graph $\half(\Lambda)$ has two Type~1 vertices for every $r_i \geq 2$, hence has $2 \sum_{p=1}^M N_p$ Type~1 vertices.  We label these as:
\[
V_1(\half(\Lambda)) = \{ x_{pq}^+, x_{pq}^- \mid 1 \leq p \leq M, 1 \leq q \leq N_p \}.
\]
The central vertex of $\half(\Lambda)$, which is of Type 2, we denote by~$y_0$.  This vertex is adjacent to every Type 1 vertex, so has valence $2 \sum_{p=1}^M N_p$.  For each $p$, the graph $\half(\Lambda)$ has $r_pN_p$ additional Type $2$ vertices, all of valence $2$, which we label so that
\[
V_2(\half(\Lambda)) = \{ y_0 \} \cup \{ y_{pq}^{j} \mid 1 \leq p \leq M, 1 \leq q \leq N_p, 1 \leq j \leq r_{p}\}.
\]
Each of the vertices $y_{pq}^j$ is adjacent to both $x_{pq}^+$ and $x_{pq}^-$.   

Similarly, $\half(\Lambda')$ has $2 \sum_{p=1}^M N'_p$ vertices of Type~1, given by
\[
V_1(\half(\Lambda')) = \{ x_{pq'}^+, x_{pq'}^- \mid 1 \leq p \leq M, 1 \leq q' \leq N'_p \},
\]
and $1 + \sum_{p=1}^M r_p N'_p$ vertices of Type~2, given by
\[
V_2(\half(\Lambda')) = \{ y_0' \} \cup \{ y_{pq'}^{j} \mid 1 \leq p \leq M, 1 \leq q' \leq N'_p, 1 \leq j \leq r_{p}\},
\]
so that $y_0'$ is adjacent to every Type 1 vertex, and $y_{pq'}^j$ is adjacent to $x_{pq'}^+$ and $x_{pq'}^-$.

\subsubsection{Constants}\label{sec:constants}

We next define quite a few constants for use in the proof of Proposition~\ref{prop:cond1}, continuing notation from Section~\ref{sec:notation1}.
We will need the following lemma. 

\begin{lemma}\label{lem:pq}  Assume $W = W_\G$ and $W' = W_{\G'}$ are as in Theorem~\ref{thm:CycleGenTheta}, and that condition (1) holds.  Then for each $1 \leq p \leq M$,
\[
\chi(W_{A'})\cdot \left( \sum_{q=1}^{N_p} R_{pq}\right)= \chi(W_A)  \cdot\left(\sum_{q'=1}^{N'_p} R'_{pq'}\right). 
\]
\end{lemma}

\begin{proof}
For any vector $v$, write $\overline{v}$ for the sum of the entries of $v$.  Using this and the notation of Theorem~\ref{thm:CycleGenTheta}, an easy calculation shows that for each $i \in I$ and each $k \in I'$,
\[
\chi(W_{\Theta_i}) = \sum_{j=1}^{r_i} \chi(W_{\beta_{ij}}) = \overline{v_i} \quad \mbox{and} \quad \chi(W_{\Theta_k'}) = \sum_{l=1}^{r_k'} \chi(W_{\beta_{kl}}) = \overline{v_k'}.
\]
Now combining this with the notation established in Section~\ref{sec:notation1}, we have by condition~(1)(b) of Theorem~\ref{thm:CycleGenTheta} that for each $1 \leq p \leq M$, 
\[
\chi(W_{A'}) \cdot\left( \sum_{q=1}^{N_p} \overline{v_{pq}} \right) = \chi(W_{A'}) \cdot\left( \sum_{q=1}^{N_p} R_{pq}\overline{w_p} \right) = 
\chi(W_A) \cdot\left( \sum_{q'=1}^{N'_p} R'_{pq'}\overline{w_p} \right) = \chi(W_A) \cdot\left( \sum_{q'=1}^{N'_p} \overline{v'_{pq'}} \right).
\]
The result follows.  
\end{proof}

Next, for each $1 \leq p \leq M$, each $1 \leq q \leq N_p$, and each $1 \leq q' \leq N_p'$, let 
\begin{equation}\label{eqn:kpq}
k_{pq} = 16|R_{pq}| \quad \mbox{and} \quad k'_{pq'} = 16|R'_{pq'}|.
\end{equation}
Then by Lemma~\ref{lem:pq}, we may define constants $B_p$ for $1 \leq p \leq M$ by 
\[
B_p = |\chi(W_{A'})| \cdot \left(\sum_{q=1}^{N_p} k_{pq} \right)= |\chi(W_A)| \cdot \left( \sum_{q'=1}^{N'_p} k'_{pq'}\right). 
\]
We also define
\begin{equation}\label{eqn:B}
B = \sum_{p=1}^M B_p.
\end{equation}

\begin{lemma}\label{lem:constants1}  With the same hypotheses as Lemma~\ref{lem:pq}:
\begin{enumerate} 
\item each $k_{pq}$ and $k'_{pq'}$ is a positive integer divisible by $4$; 
\item each $B_p$ is a positive integer, so $B$ is a positive integer; and 
\item for each $p$, we have $B_p \geq N_p$.
\end{enumerate}
\end{lemma}

\begin{proof}
By Definition~\ref{def:euler},
the entries in the Euler characteristic vectors $v_{pq}$ and $v'_{pq'}$ and the rationals $\chi(W_A)$ and $\chi(W_{A'})$ all have denominator at most $4$.  Since $w_p$ is the minimal integral element of the commensurability class $\cV_p$, which contains the vectors $v_{pq}$ and $v_{pq'}$, the rationals $R_{pq}$ and $R'_{pq'}$ have denominator at most $4$ as well. The lemma follows easily from this.
\end{proof} 

We now define $K$ to be the product of all of the $k_{pq}$ and $k'_{pq'}$, that is,
\begin{equation}\label{eqn:K}
K = \prod_{p=1}^M \left[ \left( \prod_{q=1}^{N_p} k_{pq}\right) \left( \prod_{q'=1}^{N'_p} k'_{pq'}\right) \right]. 
\end{equation}
We also define 
\begin{equation}\label{eqn:D} 
d_{pq} = \frac{K}{k_{pq}}, \qquad d'_{pq'} = \frac{K}{k'_{pq'}}, \qquad D = K|\chi(W_{A'})|, \quad \mbox{and} \quad D' = K|\chi(W_A)|.  
\end{equation}
Lemma~\ref{lem:constants1} now implies:

\begin{lemma}  The constants $K$, $d_{pq}$, $d'_{pq'}$, $D$, and $D'$ are all positive integers.
\qed
\end{lemma}

\subsubsection{Construction of the half-covers $\Psi$ and $\Psi'$}\label{sec:Psi}

We now construct the graphs $\Psi$ and $\Psi'$, and show that they are isomorphic.

Let $B$ be as defined at~\eqref{eqn:B} above, and recall from Lemma~\ref{lem:constants1} that $B$ is a positive integer.  
The graphs $\Psi$ and $\Psi'$ will both have $2B$ Type~1 vertices.  In $\Psi$, we use that $B = \sum_{p=1}^M \left( \sum_{q=1}^{N_p} |\chi(W_{A'})|k_{pq}\right)$ to put
\[
V_1(\Psi) = \{ x_{pq}^{i,+}, x_{pq}^{i,-} \mid 1 \leq p \leq M, 1 \leq q \leq N_p, 1 \leq i \leq |\chi(W_{A'})|k_{pq} \}.
\]
In $\Psi'$, we use that $B = \sum_{p=1}^M \left( \sum_{q'=1}^{N'_p} |\chi(W_A)|k'_{pq'}\right)$ to put 
\[
V_1(\Psi') = \{ x_{pq'}^{i',+}, x_{pq'}^{i',-} \mid 1 \leq p \leq M, 1 \leq q' \leq N'_p, 1 \leq i' \leq |\chi(W_A)|k'_{pq'} \}.
\]

We now describe the Type 2 vertices of $\Psi$ and $\Psi'$, and the edges.  The graph $\Psi$ has a distinguished Type~2 vertex $y_1$ of valence $2B$, which is adjacent to every Type 1 vertex.  There are $\sum_{p=1}^M B_p r_p$ additional Type 2 vertices in $\Psi$, each of valence~2, so that:
\[
V_2(\Psi) = \{ y_1 \} \cup \{ y_{pq}^{ij} \mid 1 \leq p \leq M, 1 \leq q \leq N_p, 1 \leq i \leq |\chi(W_{A'})|k_{pq}, 1 \leq j \leq r_p\}.
\]
For $1 \leq j \leq r_{p}$, the vertex $y_{pq}^{ij}$ is adjacent to $x_{pq}^{i,+}$ and $x_{pq'}^{i,-}$.  Similarly, the graph $\Psi'$ has $1 + \sum_{p=1}^M B_p r_p$ Type~2 vertices, given by
\[
V_2(\Psi') = \{ y_1' \} \cup \{ y_{pq'}^{i'j} \mid 1 \leq p \leq M, 1 \leq q' \leq N'_p, 1 \leq i' \leq |\chi(W_A)|k'_{pq'}, 1 \leq j \leq r_p\},
\]
with $y_1'$ adjacent to every Type 1 vertex, and $y_{pq'}^{i'j}$ adjacent to $x_{pq'}^{i',+}$ and $x_{pq'}^{i',-}$.

\begin{lemma}\label{lem:iso_Psi}  The graphs $\Psi$ and $\Psi'$ are isomorphic.
\end{lemma}

\begin{proof}  By construction, it suffices to show that for each $p$, the $B_p$ vertices $\{ x_{pq}^{i,+} \mid 1 \leq q \leq N_p, 1 \leq i \leq |\chi(W_{A'})|k_{pq} \}$ in $\Psi$ and the $B_p$ vertices $\{ x_{pq'}^{i',+} \mid 1 \leq q' \leq N'_p, 1 \leq i' \leq |\chi(W_A)|k'_{pq'}\}$ in $\Psi'$ have the same collection of valences.  But  each of these vertices in fact has the same valence, namely $1 + r_p$, since each $x_{pq}^{i,+}$ (respectively, $x_{pq'}^{i',+}$) is adjacent to the central vertex $y_1$ (respectively, $y_1'$), and to the $r_p$ vertices $\{ y_{pq}^{ij} \mid 1 \leq j \leq r_p \}$ (respectively,  $\{ y_{pq'}^{i'j} \mid 1 \leq j \leq r_p \}$).  The result follows.
\end{proof}

The next lemma is easily verified.

\begin{lemma}\label{lem:Psi}  There is a half-covering $\Psi \to \half(\Lambda)$ induced by $y_1 \mapsto y_0$, $x_{pq}^{i,\pm} \mapsto x_{pq}^{\pm}$, and $y_{pq}^{ij} \mapsto y_{pq}^j$, and similarly for $\Psi' \to \half(\Lambda')$.
\end{lemma}

\subsubsection{Covering the central surfaces}\label{sec:T}

Denote by $S_{\cA}$ the surface in $\cX$ that covers the orbifold $\cA$ with degree 16, and by $S_{\cA'}$ the corresponding surface in $\cX'$ (see Section~\ref{sec:S_A} for the constructions of these surfaces).   In this section, we construct a common cover $T$ of the surfaces $S_\cA$ and $S_{\cA'}$.  

Recall that the boundary components of $S_\cA$ (respectively, $S_{\cA'}$) are in bijection with the Type 1 vertices of $\half(\Lambda)$ (respectively, $\half(\Lambda')$).  Now label the boundary components of $S_\cA$ and $S_{\cA'}$ using the notation for the Type 1 vertices of these graphs from Section~\ref{sec:notation1} above.  Let $d_{pq}$, $d_{pq'}$, $D$, and~$D'$ be the positive integers defined at~\eqref{eqn:D} above.

\begin{lemma}\label{lem:T} There exists a connected surface $T$ with $2B$ boundary components which covers $S_\cA$ with degree $D$ and $S_{\cA'}$ with degree $D'$.  
\end{lemma}

\begin{proof} Notice first that
\[
D\cdot\chi(S_\cA) = \left(K|\chi(W_{A'})|\right)\cdot\left( 16 \chi(W_A) \right) =   \left(K|\chi(W_A)|\right)\cdot\left( 16 \chi(W_{A'}) \right) = D' \cdot\chi(S_{\cA'}),
\]
so these covering degrees are compatible with the existence of $T$.

Next we consider the number of boundary components.  Recall from Section~\ref{sec:S_A} that $S_\cA$ is obtained by gluing together the connected surface $S_A$, which has $2N$ boundary components, and a connected surface $S_\beta$ with two boundary components for every branch $\beta$ in $\G$ whose vertices are contained in $A$.  There is one such $S_\beta$ for every $r_i = 1$, thus $S_\cA$ has two boundary components for every $r_i \geq 2$.  Hence $S_\cA$ has $2\sum_{p=1}^{M}N_p$ boundary components.  Similarly, $S_{\cA'}$ has $2\sum_{p=1}^{M}N'_p$ boundary components. 

Now by Lemma~\ref{lem:constants1}, we have $B_p \geq N_p$ (respectively, $B_p \geq N_p'$) for $1 \leq p \leq M$.  Since $B = \sum_{p=1}^M B_p$, it follows that the number of boundary components of $S_\cA$ (respectively, $S_{\cA'}$) has the same parity as $2B$, and is less than or equal to $2B$.  

Now let $T$ be a connected surface with Euler characteristic $\chi(T) = D \cdot \chi(S_\cA) = D' \cdot \chi(S_{\cA'})$ and~$2B$ boundary components.  To complete the proof, by Lemma~\ref{lem:neumann} it suffices to specify the degrees by which the boundary components of $T$ will cover the boundary components of $S_\cA$ and~$S_{\cA'}$.  

For this, we label the $2B$ boundary components of $T$ in two different ways, first by the Type~1 vertices of $\Psi$, and second by the Type 1 vertices of $\Psi'$ (see Section~\ref{sec:Psi} above).  We specify that for $1 \leq i \leq  |\chi(W_{A'})|k_{pq}$, the boundary component $x_{pq}^{i,+}$ of $T$ covers the boundary component $x_{pq}^+$ of~$S_\cA$ with degree $d_{pq} = K/{k_{pq}}$, and that for $1 \leq i' \leq  |\chi(W_A)|k'_{pq'}$, the boundary component $x_{pq'}^{i',+}$ of $T$ covers the boundary component $x_{pq'}^+$ of $S_{\cA'}$ with degree $d'_{pq'} = K/{k'_{pq'}}$.  Hence the union $\cup_{i} x_{pq}^{i,+}$ covers $x_{pq}^+$ with total degree $|\chi(W_{A'})|k_{pq}d_{pq} = |\chi(W_{A'})|K = D$ and the union $\cup_{i'} x_{pq'}^{i',+}$ covers $x_{pq'}^+$ with total degree $|\chi(W_A)|k'_{pq'}d'_{pq'} = |\chi(W_A)|K = D'$.  Similarly, the union  $\cup_{i} x_{pq}^{i,-}$ covers $x_{pq}^- \subset S_\cA$ with degree $D$ in total and degree $d_{pq}$ on each component, and $\cup_{i'} x^{i',-}_{pq'} \subset T$ covers $x^-_{pq'} \subset S_{\cA'}$ with degree $D'$ in total and degree $d'_{pq'}$ on each component.   This completes the proof.
\end{proof}

\subsubsection{Covering the other surfaces}\label{sec:Tpj}

For each $p$, $q$, $q'$, and $j$, let $S_{pq}^j$ (respectively, $S_{pq'}^j$) be the surface in $\cX$ (respectively, $\cX'$) corresponding to the vertex $y_{pq}^j \in V_2(\half(\Lambda))$ (respectively, $y_{pq'}^j \in V_2(\half(\Lambda'))$).  Recall that each $S_{pq}^j$ and $S_{pq'}^j$ has 2 boundary components.  In this section we construct a surface~$T_p^j$ which is a common cover of the surfaces $S_{pq}^j$ and~$S_{pq'}^j$.  

\begin{lemma}\label{lem:Tpj}  For each $1 \leq p \leq M$ and each $1 \leq j \leq r_p$, there exists a connected surface $T_p^j$ with 2 boundary components so that:
\begin{enumerate}
\item for all $1 \leq q \leq N_p$, the surface $T_p^j$ covers $S_{pq}^j$ with degree $d_{pq}$ in total, and degree $d_{pq}$ on each boundary component; and
\item  for all $1 \leq q' \leq N'_p$, the surface $T_p^j$ covers $S_{pq'}^j$ with degree $d'_{pq'}$ in total, and degree~$d'_{pq'}$ on each boundary component.
\end{enumerate}
\end{lemma}

\noindent The constants  $d_{pq}$ and $d'_{pq'}$ in this statement are as defined at~\eqref{eqn:D} above.

\begin{proof}  We prove (1); the proof of (2) is similar.  Write $w_p^j$ and $v_{pq}^j$ for the $j$th entry of the vectors~$w_p$ and $v_{pq}$,  respectively.  Let $T_p^j$ be a connected surface with 2 boundary components and Euler characteristic $\chi(T_p^j) = -Kw_p^j$, where $K$ is the positive integer defined at~\eqref{eqn:K} above.  Since each~$S_{pq}^j$ also has 2 boundary components, by Lemma~\ref{lem:neumann} it suffices to show that $\chi(T_p^j) = d_{pq}\chi(S^j_{pq})$.  Now~$S_{pq}^j$ is a 16-fold cover of an orbifold with Euler characteristic~$v_{pq}^j$, so  
\[ d_{pq}\chi(S_{pq}^j) = 16d_{pq}v_{pq}^j = 16 d_{pq}R_{pq}w_p^j = -k_{pq}d_{pq}w_p^j = -K w_p^j = \chi(T_p^j)\]
where the constant $k_{pq}$ is as defined at~\eqref{eqn:kpq} above.  This completes the proof.
\end{proof}

\subsubsection{Constructions of surface amalgams $\cY$ and $\cY'$}\label{sec:Y}

We finish the proof of Proposition~\ref{prop:cond1} by constructing homeomorphic surface amalgams $\cY$ and $\cY'$ which cover $\cX$ and $\cX'$, respectively.

Recall that the surface amalgam $\cY$ will have JSJ graph $\Psi$, and $\cY'$ will have JSJ graph $\Psi'$, where~$\Psi$ and $\Psi'$ are the isomorphic graphs from Section~\ref{sec:Psi} above.  The surfaces in $\cY$ and $\cY'$ will be those constructed in Sections~\ref{sec:T} and~\ref{sec:Tpj} above, as we now explain.  In $\cY$, put $S_{y_1}=T$, and in $\cY'$, put $S_{y_1'}=T$, where $T$ is the surface with $2B$ boundary components from Lemma~\ref{lem:T}.  Now for each $p$, $q$, $q'$, $i$, $i'$, and $j$, if $y = y_{pq}^{ij} \in V_2(\Psi)$ write $S_{pq}^{ij}$ for the surface $S_y$ in $\cY$, and similarly, if $y' = y_{pq'}^{i'j} \in V_2(\Psi')$ write $S_{pq'}^{i'j}$ for the surface $S_{y'}$ in $\cY'$.  We then put $S_{pq}^{ij} = T_p^j$ in $\cY$ and $S_{pq'}^{i'j} = T_p^j$ in $\cY'$, where $T_p^j$ is the surface with 2 boundary components constructed in Lemma~\ref{lem:Tpj} above. 
 
To glue these surfaces together to form a surface amalgam $\cY$ which covers $\cX$, we label the two boundary components of $S_{pq}^{ij}$ by the two vertices of $\Psi$ which are adjacent to $y_{pq}^{ij}$, namely $x_{pq}^{i,+}$ and $x_{pq}^{i,-}$.   Now recall from the proof of Lemma~\ref{lem:T} that if we label the boundary components of $T = S_{y_1}$ by the Type 1 vertices of $\Psi$, then the boundary component $x_{pq}^{i,\pm}$ of $S_{y_1}$ covers the boundary component $x_{pq}^\pm$ of $S_\cA \subset \cX$ with degree $d_{pq}$.  Also, by Lemma~\ref{lem:Tpj}, the covering from $T_{p}^j = S_{pq}^{ij}$ to the surface $S_{pq}^j \subset \cX$ has degree $d_{pq}$ on each boundary component.  Since the degrees of these coverings $S_{y_1} \to S_\cA$ and $ S_{pq}^{ij} \to S_{pq}^j$ are equal on each boundary component which has the same label, we may glue together boundary components of the collection $\{ S_{y_1} \} \cup \{ S_{pq}^{ij} \}$ which have the same labels to obtain a surface amalgam $\cY$ which covers $\cX$.  The construction of $\cY'$ which covers $\cX'$ is similar.  

Now all of the surfaces in $\cY$ and $\cY'$ are compact, hence the coverings $\cY \to \cX$ and $\cY' \to \cX'$ are finite-sheeted (in fact, they have degrees $D$ and $D'$ respectively, but we will not need this).  The following lemma then completes the proof of Proposition~\ref{prop:cond1}.

\begin{lemma}  The surface amalgams $\cY$ and $\cY'$ are homeomorphic.
\end{lemma}

\begin{proof} We note first that $\cY$ and $\cY'$ have isomorphic JSJ graphs $\Psi$ and $\Psi'$,  by Lemma~\ref{lem:iso_Psi}, so it suffices to consider the surfaces in $\cY$ and $\cY'$.  By construction, $\cY$ and $\cY'$  have the same central surface $S_{y_1} = T = S_{y_1'}$.  Now for each $1 \leq p \leq M$, consider the $B_p r_p$ surfaces $\{ S_{pq}^{ij} \}$ in $\cY$ and the $B_p r_p$ surfaces $\{ S_{pq'}^{i'j} \}$ in $\cY'$.  Since $S_{pq}^{ij} = T_p^j = S_{pq'}^{i'j}$ for each $q$, $q'$, $i$, and $i'$, we have by construction that each of the $B_p$ pairs of branching curves $\{x_{pq}^{i,+}, x_{pq}^{i,-}\}$ in $\cY$ (respectively, $\{x_{pq'}^{i',+}, x_{pq'}^{i',-}\}$ in $\cY'$) is incident to the same collection of $1 + r_p$ surfaces $\{T\} \cup \{ T_p^j  \}$.  Hence $\cY$ and $\cY'$ are homeomorphic.
\end{proof}

This completes the proof of Proposition~\ref{prop:cond1}.

\subsection{Sufficiency of condition (2)}\label{sec:cond2} 

We now prove the following result.
 
 \begin{prop}\label{prop:cond2}
Let $W=W_{\Gamma}$ and $W'=W_{\Gamma'}$ be as in Theorem~\ref{thm:CycleGenTheta}.  If condition (2) holds, then 
$W$ and $W'$ are commensurable. 
 \end{prop}
 
 \begin{center}
 \begin{figure}[t]
  \begin{overpic}[width=125mm]
{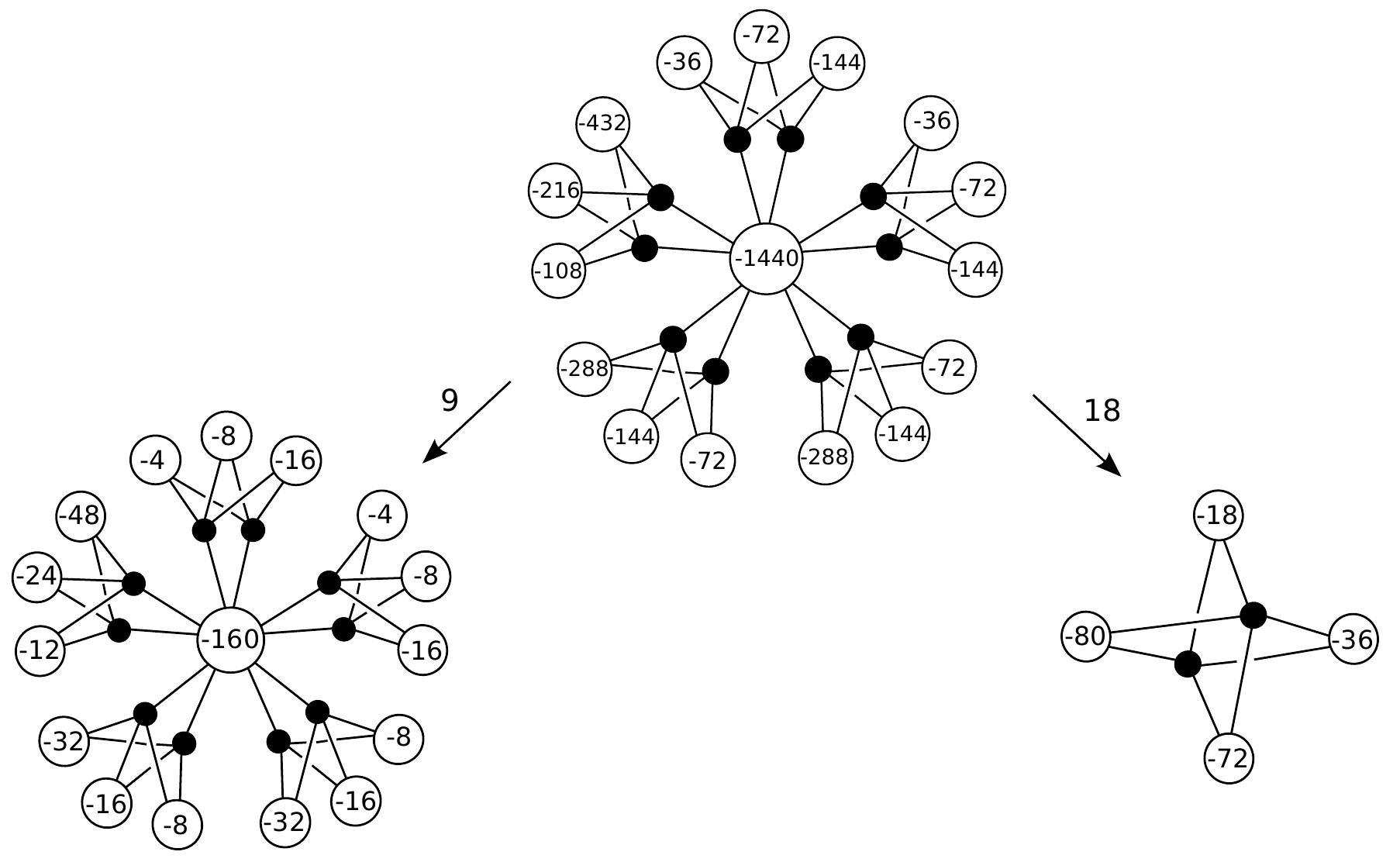}
\put(0,30){$\cX$}
\put(35,55){$\cZ$}
\put(95,25){$\cW$}
 \end{overpic}
 \caption{{\Small Illustrated above is an example of the covers described in Section \ref{sec:cond2}, on the level of the JSJ graphs. The white vertices represent surfaces with specified Euler characteristic; the valence of the vertex is the number of boundary components of the surface. The black vertices represent branching curves that are glued to boundary components of adjacent surfaces. The covering maps between surface amalgams restrict to half-coverings between the JSJ graphs. In this example, the space $\cX$ does not cover the space $\cW$, and there is no minimal surface amalgam within this commensurability class.   }}
 \label{figure:suff_cond_2} 
  \end{figure}
 \end{center}

 By the construction described in Section \ref{sec:tfcover}, there are surface amalgams $\cX$ and $\cX'$ which form degree $16$ covers of the orbicomplexes $\cOg$ and $\cOgp$, respectively. To prove that $W$ and $W'$ are abstractly commensurable, we prove there are finite-sheeted covers $\cZ \rightarrow \cX$ and $\cZ' \rightarrow \cX'$ and a surface amalgam $\cW$ so that $\cZ$ and $\cZ'$ finitely cover $\cW$.  An example of the construction given in this section appears in Figure \ref{figure:suff_cond_2}.

 We will describe these covering spaces and maps in terms of their JSJ graphs and half-covers as defined in Section \ref{sec:half-coverings}. We set notation in Section \ref{sec:suff2_notation}. The surface amalgam $\cW$ is described in Section \ref{sec:suff2_cW}. The surface amalgam $\cZ$ and the finite cover $\cZ \rightarrow \cX$ is given in Section \ref{sec:suff2_cZ}. Finally, the finite cover $\cZ \rightarrow \cW$ is given in Section \ref{sec:suff2_covers}. 
The construction for $\cX'$ is analogous.   
  
  \subsubsection{Notation}\label{sec:suff2_notation}
  For ease of notation, suppose the elements of the multiset of Euler characteristic vectors $\{v_i \mid i \in I\}$ in the statement of Theorem \ref{thm:CycleGenTheta} are labeled $v_1, \ldots, v_n$. By assumption, there exists $r \geq 2$ so that each vector $v_i$ for $1 \leq i \leq n$ has $r$ entries. Let $\cV$ denote the commensurability class of the vectors $\{v_i \mid 1 \leq i \leq n\}$. 
  
  The orbicomplex $\cOg$ contains a central orbifold $\cA$ with $n$ non-reflection edges $e_1, \ldots, e_n$. Attached to the non-reflection edge $e_i$ is a collection of $r$ branch orbifolds $\cP_{i1}, \ldots, \cP_{ir}$, with $\cP_{ij} = \cP_{\beta_{ij}}$ and $\pi_1(\cP_{ij}) = W_{ij} = W_{\beta_{ij}}$ for each $1 \leq i \leq n$ and $1 \leq j \leq r$. By definition, for $1 \leq i \leq n$ the Euler characteristic vector $v_i$ is given by
 \[
  v_i = (\chi_{i1}, \ldots, \chi_{ir}) = (\chi(W_{i1}), \ldots, \chi(W_{ir})),
  \]
  and $w$ is the reordering of 
\[ \left(\sum_{j=1}^n \chi_{j1}, \sum_{j=1}^n \chi_{j2}, \ldots, \sum_{j=1}^n \chi_{jr}, \chi(W_A)\right)  
\]
so that its entries are in non-increasing order.

 The surface amalgam $\cX$ has a central surface $S_{\cA}$ with Euler characteristic $16\chi(W_A)$ and $2n$ boundary components. Attached to each of the $n$ pairs of boundary components of $S_{\cA}$ are surfaces $S_{i1}, \ldots ,S_{ir}$ for $1\leq i \leq n$ where each $S_{ij}$ has two boundary components and Euler characteristic $16\chi(W_{ij})$, and forms a degree $16$ cover of $\cP_{ij}$. 
 
 Let $\Lambda$ be the JSJ graph of $W$. As described in Section \ref{sec:X}, the JSJ graph of  $\cX$ is the half-cover~$\half(\Lambda)$. 
 The graph $\cH(\gL)$ has $2n$ vertices of Type 1. Label these vertices as 
 $$ V_1(\cH(\gL)) = \{ x_i^{+}, x_i^{-} \mid 1 \leq i \leq n \}. $$
 There is a central vertex of Type 2 associated to the surface $S_{\cA}$; label this vertex as $y_0$. This vertex is adjacent to every Type 1 vertex, hence has valence $2n$. For each $i$, the graph $\cH(\gL)$ has~$r$ additional Type 2 vertices, associated to the surfaces $S_{ij}$ for $1 \leq j \leq r$. Label these vertices accordingly, so that 
 \[V_2(\cH(\gL)) = \{ y_0\} \cup \{ y_{ij} \mid 1 \leq i \leq n, 1 \leq j \leq r \}.  \]
 Each vertex $y_{ij}$ is adjacent to $x_i^+$ and $x_i^-$. 
 
 For $1 \leq i \leq n$,  define 
 \[
  \widetilde{v}_i = 16v_i 
  = (\chi(S_{i1}), \dots, \chi(S_{ir})) \in 2\Z^r_{<0}.
\]  
Let $\ww = 16w$, so that $\ww$ is the reordering of 
 \[
    \left(\sum_{j=1}^n \chi(S_{j1}), \sum_{j=1}^n \chi(S_{j2}), \ldots, \sum_{j=1}^n \chi(S_{jr}), \chi(S_{\cA})\right) \in 2\Z^{r+1}_{<0} 
   \]
 so that its entries are in non-increasing order. Suppose that in this ordering, $\chi(S_{\cA})$ is the $k$th entry of $\ww$. Let $\widehat{w} \in \Z^r$ be the vector $\ww$ with the $k$th entry deleted. Then 
  \[
   \hw = \wv_1 + \dots + \wv_n \in 2\Z^r_{<0}.
  \]
  
  By condition (2)(b) in Theorem~\ref{thm:CycleGenTheta}, the vectors $w$ and $w'$ are commensurable. Let 
  \begin{eqnarray*}
   w_0 &=& (w_1, \ldots, w_{r+1}) \in \Z^{r+1}
  \end{eqnarray*}
 be the minimal integral element in the commensurability class of $w$ and $w'$. Since all entries of $w$ and $w'$ are negative, all entries of $w_0$ are positive.  Finally, let $\hwn$ be the vector $w_0$ minus its $k$th entry, so 
     \begin{eqnarray*}
      \hwn &=& (w_1, \ldots, w_{k-1}, w_{k+1}, \ldots, w_{r+1}) \in \Z^r.
     \end{eqnarray*}
  
  \subsubsection{The surface amalgam $\cW$}\label{sec:suff2_cW}
      
    The surface amalgam $\cW$ has JSJ graph $\Phi$ with two Type~1 vertices labeled 
    \[ V_1(\Phi) = \{x^+, x^- \}. \] 
    The graph $\Phi$ has $r+1$ Type $2$ vertices labeled 
    \[ V_2(\Phi) = \{y_1, \ldots, y_{r+1} \}. \]
    Each Type $2$ vertex is adjacent to both $x^+$ and $x^-$. 
    
    To construct the surface amalgam $\cW$, for $1 \leq i \leq r+1$, let $W_i$ be the surface with Euler characteristic $-2w_i$ and 
    2 boundary components $C_{i1}$ and $C_{i2}$; the surface $W_i$ is associated to the vertex $y_i$. Identify the curves $C_{i1}$ for $1 \leq i \leq r+1$ to create a single curve $C_1$ associated to the vertex $x^+$, and identify the curves $C_{i2}$ for $1 \leq i \leq r+1$ to create a single curve $C_2$ associated to the vertex $x^-$. This forms a surface amalgam $\cW$ which has two singular curves $C_1$ and $C_2$ and JSJ graph $\Phi$. 
    
    \begin{lemma} \label{lem:suff2_JSJ_cover}
     The graph $\half(\gL)$ forms a half-cover of the graph $\Phi$.
    \end{lemma}
    \begin{proof}
     The cover is induced by the maps $x_i^+ \mapsto x^+$, $x_i^- \mapsto x^-$, and $y_{ij} \mapsto y_j$ for $1\leq i \leq n$ and $1 \leq j \leq r$, and $y_0 \mapsto y_{r+1}$. 
    \end{proof}

    \begin{remark} In view of Lemma~\ref{lem:suff2_JSJ_cover}, note that 
$\cX$ does not necessarily cover $\cW$;    an example of this is given in Figure \ref{figure:suff_cond_2}. To resolve this, we construct a finite cover $\cZ \rightarrow \cX$ so that~$\cZ$ has JSJ graph $\half(\gL)$ and finitely covers $\cW$.
    \end{remark}

    \subsubsection{The surface amalgam $\cZ$ and the finite covering map $\cZ \rightarrow \cX$.}\label{sec:suff2_cZ}  
    
    We will need the following lemma.

      \begin{lemma} \label{lemma_suff2_wu}
       Let $u \in \Z^r$ be the minimal integral element in the commensurability class  $\cV$. There exists a positive integer $K$ so that $\hwn = Ku$. 
      \end{lemma}
      \begin{proof}
       Since $\wv_i = 16v_i$, the set of integer vectors $\{\wv_i, \hw, \hwn\mid 1 \leq i \leq n\}$ is also contained in $\mathcal{V}$. Thus there exists an integer $K$ so that $\hwn = Ku$. Since all entries of $\hwn$ and $u$ are positive, $K$ is a positive integer. 
      \end{proof}

      Let $\cZ$ be the following surface amalgam with JSJ graph $\half(\gL)$. Associated to the Type 2 vertex~$y_0$ of $\half(\gL)$, the space $\cZ$ contains one central surface $\widetilde{S_{\cA}}$ with $2n$ boundary components and Euler characteristic $\chi(\widetilde{S_{\cA}}) = K\chi(S_{\cA})$.      
      Attached to the $i$th pair of boundary curves of $\widetilde{S_{\cA}}$, which are associated to the vertices $x_i^+, x_i^- \in \half(\gL)$  for $1 \leq i \leq n$, there are $r$ surfaces $\widetilde{S_{ij}}$ for $1 \leq j \leq r$, where~$\widetilde{S_{ij}}$ has 2 boundary components and Euler characteristic $K\chi(S_{ij})$; these surfaces are associated to the vertices $y_{ij}$ in $\half(\gL)$.

      \begin{lemma} \label{lemma_suff2_zoverx}
       The surface amalgam $\cZ$ forms a degree $K$ cover of the surface amalgam $\cX$, where~$K$ is the positive integer guaranteed by Lemma \ref{lemma_suff2_wu}.        
      \end{lemma}
      \begin{proof}
       By Lemma \ref{lem:neumann}, the surface $\widetilde{S_{\cA}}$ forms a degree $K$ cover of the surface $S_{\cA}$ so that the degree restricted to each boundary component of $\widetilde{S_{\cA}}$ over the corresponding boundary component of $S_{\cA}$ is equal to $K$. Similarly, by Lemma \ref{lem:neumann}, each surface $\widetilde{S_{ij}}$ forms a degree $K$ cover of $S_{ij}$ so that the degree restricted to each boundary component of $\widetilde{S_{ij}}$ over the corresponding boundary component of $S_{ij}$ is equal to $K$. Thus, since the covering maps agree along the gluings, $\cZ$ forms a degree $K$ cover of $\cX$. 
      \end{proof}

  \subsubsection{The finite cover $\cZ \rightarrow \cW$}\label{sec:suff2_covers}
 For $1 \leq i \leq n$, define 
\[
   z_i = K\wv_i 
   = (\chi(\widetilde{S_{i1}}), \dots, \chi(\widetilde{S_{ij}}))\;\text{ and }\;
z= K\ww,
\]   
where $K$ is the positive integer from Lemma \ref{lemma_suff2_wu}.
Note that the entries of $z$ are in non-increasing order. 
If $\chi(\widetilde{S_{\cA}})$ is the $k$th entry of $z$, let $\hz$ be the vector $z$ with the $k$th entry deleted. Then
$$
    \hz = K\hw 
    = z_1 + \dots + z_n \in 2\Z^r_{<0}.
$$
   Since $w_0$ is the minimal integral element in the commensurability class of $\ww \in 2\Z^{r+1}_{<0}$, there exists a positive integer $M$ so that $\ww=-2Mw_0$. Hence, $z = MK(-2w_0)$. Let $D = MK \in \N$. Then
$$
    z= D(-2w_0) \;\text{ and }\;
    \hz =D(-2\hwn).
  $$

 We compute $D$ a second way, and use this to prove that $\cZ$ finitely covers $\cW$. 
 
 \begin{lemma} \label{lem:suff2_D}
  Let $d_i$ be positive integers so that $\wv_i = d_i(-2u)$, where $u$ is the minimal integral element in the commensurability class $\mathcal{V}$. Then $D = \sum_{i=1}^n d_i$. 
 \end{lemma}
 \begin{proof}
   By construction of the covering map and the definition of $K$, we have
$$
  z_i = K\wv_i = K(d_i(-2u)) = d_i(-2\hwn).
$$ Hence
$$
  \hz = z_1 + \dots + z_n = \left(\sum_{i=1}^n d_i\right) (-2\hwn).
 $$
 So, combining the two equations for $\hz$ given above, we have that $\sum_{i=1}^n d_i = D$.
 \end{proof}

 \begin{lemma}
  The surface amalgam $\cZ$ forms a degree $D$ cover of the surface amalgam $\cW$.
 \end{lemma}
 \begin{proof}
    Suppose for ease of notation that $\chi(\widetilde{S_{\cA}})$ is the $(r+1)$st entry of the vector $z$. Then  for all $1 \leq i \leq n$ and all $1 \leq j \leq r$, we have that $\chi(\widetilde{S_{ij}}) = d_i \chi(W_j)$, where the surface $W_j$ in $\cW$ has Euler characteristic $w_j$ by definition, and $d_i$ is defined in Lemma \ref{lem:suff2_D}. By Lemma~\ref{lem:neumann}, $\widetilde{S_{ij}}$ covers~$W_j$ by degree $d_i$ so that the degree restricted to each of the boundary components is equal to $d_i$. So, $\bigcup_{i=1}^n \widetilde{S_{ij}}$ forms a degree $D$ cover of $W_j$ for all $1 \leq j \leq r$. In addition, $\chi(\widetilde{S_{\cA}}) = D\chi(W_{r+1})$, and the~$n$ pairs of boundary components of $\widetilde{S_{\cA}}$ can be partitioned so that the $i$th pair of boundary curves of~$\widetilde{S_{\cA}}$ covers the pair of boundary curves of $W_{r+1}$ by degree $d_i$, for $1 \leq i \leq n$, with  $\sum_{i=1}^n d_i = D$. Thus by Lemma~\ref{lem:neumann}, the space $\cZ$ forms a $D$-fold cover of $\cW$, concluding the proof. 
 \end{proof}

\section{Geometric amalgams and RACGs}\label{sec:GeomAmalgamsRACGs}

In this section we further investigate the relationship between geometric amalgams of free groups (see Definition~\ref{def:surface-amalgam}) and right-angled Coxeter groups.  In Section~\ref{sec:GeomAmalgamsTrees} we prove Theorem~\ref{thm:GeomAmalgamsRACGs} from the introduction, which states that  the fundamental groups of surface amalgams whose JSJ graphs are trees are commensurable to right-angled Coxeter groups (with defining graphs in~$\cG$).  
As a consequence, we obtain the commensurability classification of geometric 
amalgams of free groups whose JSJ graphs are trees of diameter at most four 
(Corollary~\ref{cor:GeomAmalgamsCommens}).  Then in Section~\ref{sec:GeomAmalgamsNonTrees} we give an example of a geometric amalgam of free groups which is not quasi-isometric (hence not commensurable) to any Coxeter group, or indeed to any group which is finitely generated by torsion elements.

\subsection{Geometric amalgams over trees}\label{sec:GeomAmalgamsTrees}
Let $\cX(T)$ be a surface amalgam whose JSJ graph is a tree~$T$.  To prove Theorem~\ref{thm:GeomAmalgamsRACGs}, we construct a surface amalgam $\cX(\Lambda)$ which finitely covers $\cX(T)$, and which looks like one of the covers constructed in 
Section~\ref{sec:tfcover}.  We then follow the construction from Section~\ref{sec:tfcover} backwards 
to construct an orbicomplex $\cO$ with fundamental group a right-angled Coxeter group whose JSJ graph is $T$.  It follows that the fundamental group of $\cX(T)$ is commensurable to a right-angled Coxeter group with JSJ graph $T$.

\begin{proof}[Proof of Theorem~\ref{thm:GeomAmalgamsRACGs}]
Let $\cX(T)$ be as above.   By Corollary~\ref{cor:PosGenus} we may assume that every surface associated to a Type 2 vertex of $\cX(T)$ has positive genus.  Let $\Lambda = \half(T)$ be the graph from Definition~\ref{def:halfT}, with associated half-covering map $\theta: \Lambda \to T$.   We construct a surface amalgam over~$\Lambda$ which covers $\cX(T)$ with degree 8. 

Let $y$ and $y'$ be Type 2 vertices in $T$ and $\Lambda$ respectively, with $\theta(y') =y$. Recall that for each edge~$e$ incident to $y$ in $T$, there are  two edges $e', e''$ incident to $y'$ in $\Lambda$, such that $\theta(e') = \theta(e'') =e$. In particular, the valence of $y'$ is twice that of $y$.  Suppose $S_y$ is the surface associated to $y$, with genus~$g$ and $b$ boundary components.  
The surface $S_{y'}$ associated to $y'$ is defined to be the degree 8 cover of $S_{y}$ which has 
$2b$ boundary components, such that each boundary component of $S_y$ is covered by two boundary curves of $S_{y'}$, each mapping with degree 4.  Such a cover exists by Lemma~\ref{lem:neumann}, since 8 and $2b$ are both even.  This induces a bijection between the boundary components of $S_{y'}$ and the edges of $\Lambda$ incident to $y'$.  
Now form $\cX(\Lambda)$ by gluing the surfaces associated to the Type~2 vertices of $\Lambda$ according to the adjacencies in $\Lambda$.  The covering maps on the individual surfaces  induce a degree 8 cover $\cX(\Lambda) \to \cX(T)$. 

If $y, y', S_y, g, b$ and $S_{y'}$ are as in the previous paragraph, and $g'$ is the genus of $S_{y'}$, then we have $8\chi(S_y) = \chi (S_{y'})$, or equivalently, $8(2-2g-b)= 2-2g'-2b$.  Simplifying, we get:
\begin{equation}\label{eq:g'}
g' = 3b + 8g -7. \end{equation}
Observe that $g'$ is even if and only if $b$ is odd, and that $g' > 3b$, since $g$ was assumed to be positive.

Next we construct an orbicomplex $\cO$ 
as a graph of spaces with underlying graph $T$
such that~$\cX(\Lambda)$ is a degree 16 cover of $\cO$.  
If $y$ is a terminal Type 2 vertex of $T$, then the vertex space associated to~$y$ 
is the orbifold $\cP_y$ defined as follows. Let $S_{y'}\subset \cX(\Lambda)$ be as above.  
It follows from Equation~\eqref{eq:g'}, since $b=1$ in this case, that $g'$ is even.  Moreover, $g'$ is positive.  Then $S_{y'}$ forms a degree 16 cover of $\cP_y$, an orbifold with underlying space a right-angled hyperbolic polygon with $\frac{g'}2+4 \ge 5$ sides, as described in 
Section~\ref{sec:tfcover}.  The orbifold $\cP_y$ has $\frac{g'}{2}+3 \ge 4$ reflection edges and one non-reflection edge of length~1.

Now let $y$ be a non-terminal Type 2 vertex of $T$ with valence $b$.
 Then the surface $S_{y'}$ of $\cX(\Lambda)$ has $2b$ boundary components and genus $g' > 3b$ such that $g'$ and $b$ have opposite parity. 
We will realize $S_{y'}$ as a cover of an orbifold $\cA_y$ which is the union of 
orbifolds $\cQ_A$ and $\cP$, identified to each other along a non-reflection edge of each.
To do this, decompose $S_{y'}$ into two surfaces $S_{1}$ and $S_{2}$ by cutting along a pair of non-separating curves whose union separates $S_{y'}$, and
 such that $S_{1}$ has $2b+2$ boundary components and genus $3(b+1) - 7 <g'$, and $S_{2}$ has 2 boundary components and genus $h$ satisfying 
$g' = 1 + 3(b+1)-7 + h$.  Thus $h = g'-3b +3$.  Since $g'$ and $b$ have opposite parity, $h$ is even, and since $g'>3b$, we know that $h$ is positive.

 By the construction in Section~\ref{sec:tfcover}, the surface 
$S_{1}$ can be realized as a degree 16 cover of  an orbifold~$\cQ_A$ with underlying space a right-angled $2(b+1)$-gon, so that $\cQ_A$ has $b+1$ non-reflection edges, each of length 1, and 
$b+1$ reflection edges.  
Since $h$ is positive and even, the surface $S_{2}$ forms a degree 16 cover of an  orbifold $\cP$ with $\frac h 2 +4 \ge 5$ sides, as in the terminal vertex case. The orbifold $\cP$ has $\frac{h}{2} +3$ reflection edges and one non-reflection edge of length 1. Form the orbifold $\cA_y$ by gluing the non-reflection edge of $\cP$ to one of the non-reflection edges of $\cQ_A$.  Now let $\cA_y$ be the vertex space associated to $y$.  

We now identify the orbifold vertex spaces described above according to the incidence relations of~$T$ along non-reflection edges to get an orbicomplex $\cO$. 
Arguments similar to those from Section~\ref{sec:orbicomplex} show that the fundamental group of $\cO$ is a right-angled Coxeter group.   
Since the covering maps from the surfaces in $\cX$ to the orbifolds in $\cO$ agree along their intersections, the space $\cX$ finitely covers $\cO$.  The theorem follows. 
\end{proof}

\begin{proof}[Proof of Corollary~\ref{cor:GeomAmalgamsCommens}.]
 If $T$ is a tree, then by Theorem~\ref{thm:GeomAmalgamsRACGs}, any geometric amalgam of free groups with JSJ graph $T$ is commensurable to a right-angled Coxeter group $W$ with JSJ graph $T$.  If, moreover, $T$ has diameter at most 4, then by 
Remark~\ref{rmk:TreeJSJ}, the defining graph of $W$ is 
 either a generalized $\Theta$-graph or a cycle of generalized $\Theta$-graphs. Corollary~\ref{cor:GeomAmalgamsCommens} follows.
 \end{proof}

\subsection{Geometric amalgams not over trees}\label{sec:GeomAmalgamsNonTrees} 

In the case of surface amalgams whose JSJ graphs are not trees, the situation is more complicated.  
Since $\half(T)$ is not a tree, clearly some geometric amalgams of free groups whose JSJ graphs are not trees are commensurable to right-angled Coxeter groups.  On the other hand, the following example shows that there exist geometric amalgams of free groups which are not quasi-isometric (and hence not commensurable) to any Coxeter group, or indeed to  any group which is finitely generated by torsion elements.  

\begin{example}\label{ex:non-Coxeter-amalgam}
Let $G$ be a geometric amalgam of free groups with JSJ graph $\Lambda$ as in
Figure~\ref{fig:non-Coxeter-amalgam}, where the Type 1 vertices are white and the Type 2 vertices black.

\begin{figure}[htp]
\begin{overpic}[scale =0.33]%
{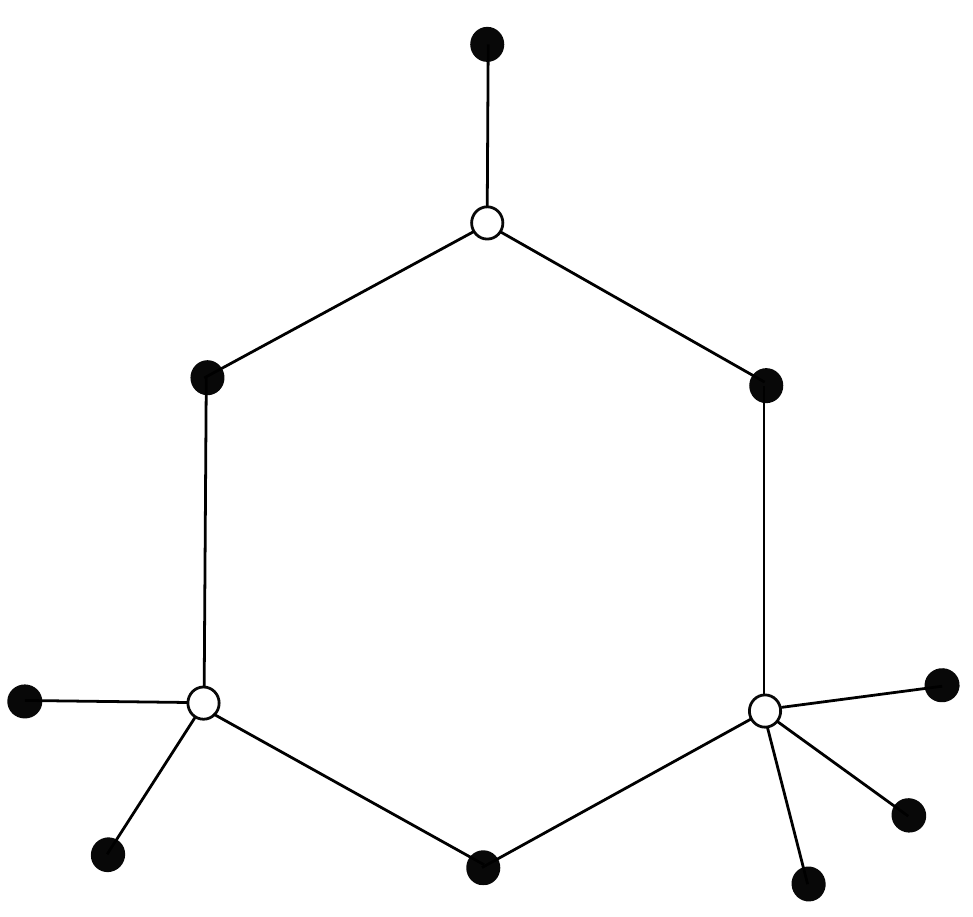}
\put(-30,40){{$ \Lambda=$}}
\put(54,72){{\scriptsize$ a_0$}}
\put(82,54){{\scriptsize$a_1$}}
\put(82,27){{\scriptsize$a_2$}}
\put(45,-5){{\scriptsize$a_3$}}
\put(8 ,27){{\scriptsize$a_4$}}
\put(8,54){{\scriptsize $a_5$}}
\end{overpic}
\caption{}
\label{fig:non-Coxeter-amalgam}
\end{figure}

Suppose $G$ is quasi-isometric to a group $H$ which is finitely generated by torsion elements.  Then~$H$ is 
also 1-ended and hyperbolic, and $G$ and $H$ have the same JSJ tree $\cT$.  Moreover, 
the JSJ graph of $H$ is a finite tree, say $\Psi$.  (The graph~$\Psi$ is a tree since any group with a nontrivial loop in its JSJ graph admits a nontrivial homomorphism to $\Z$, but no such homomorphism from $H$ to $\Z$ exists.)  Now the JSJ tree $\cT$ half-covers $\Psi$.  
We show that this is a contradiction, by showing that if $\theta:\cT \to T$ is a half-covering, where $T$ is any tree, then $T$ must be infinite.  In the following argument, all indices are taken mod $6$.

For each vertex $a_i$ of $\Lambda = G \backslash \cT$, by abuse of notation write $Ga_i$ for the $G$-orbit of some lift of~$a_i$ in $\cT$.    The following are straightforward consequences of the definition of a half-cover.  
\begin{enumerate}
\item 
Each vertex in $ \theta(Ga_i) \subset V(T)$ is adjacent to at least one vertex from each of $\theta(Ga_{i-1})$ and $\theta(Ga_{i+1})$. \item For all $i \neq j$, we have $\theta( Ga_i)\cap \theta(Ga_j) = \emptyset$.
\end{enumerate}

To show that $T$ must be infinite, we define a map $\nu: [0,\infty) \to T$ as follows.  Choose $\nu(0)$ to be an element of $\theta(Ga_0)$, and then for each $i>0$, inductively choose $\nu(i)$ to be an element of $\theta(Ga_i)$ adjacent to $\theta(Ga_{i-1})$, which exists by (1) above.  The adjacency condition can now be used to extend $\nu$ to 
$[0, \infty)$.  By (2), we conclude that for each $i>0$, the vertices $\nu(i-1)$ and $\nu(i+1)$ are distinct (i.e. there is no ``folding'').  This, together with the fact that $T$ is a tree, implies that~$\nu$ is injective, and it follows that $T$ is infinite.  
\end{example}

\section{Discussion}\label{sec:discussion}

In this section we discuss some obstructions to extending our commensurability results.

A natural strategy that we considered is that used by Behrstock--Neumann in~\cite{behrstock-neumann} for quasi-isometry classification.  
In Theorem 3.2 of~\cite{behrstock-neumann}, it is shown that two non-geometric graph manifolds $M$ and $M'$ which have isomorphic Bass--Serre trees $T$ and~$T'$ are quasi-isometric, by first proving that there is a minimal graph $\Lambda''$ which is \emph{weakly covered} by both of the quotient graphs $\Lambda = \pi_1(M) \backslash T$ and $\Lambda' = \pi_1(M') \backslash T'$.  The notion of weak covering just requires the graph morphism to be locally surjective, unlike the half-coverings we introduce in Section~\ref{sec:half-coverings}, which must be locally bijective at vertices of Type 1.  The minimal graph $\Lambda''$ determines a $3$-manifold~$M''$ which is covered by both $M$ and $M'$, and a bilipschitz map $M \to M'$ is constructed by showing that there are bilipschitz maps $M \to M''$ and $M' \to M''$.

We now give an example to explain why such a strategy cannot be pursued in our setting.  Let $G = G(\Lambda)$ and $G' = G(\Lambda')$ be geometric amalgams of free groups with JSJ graphs $\Lambda$ and~$\Lambda'$ as shown in Figure~\ref{fig:NoHalfCover}, with Type 1 vertices white and Type 2 vertices black.  
\begin{figure}[htp]
\begin{overpic}[scale =0.4]%
{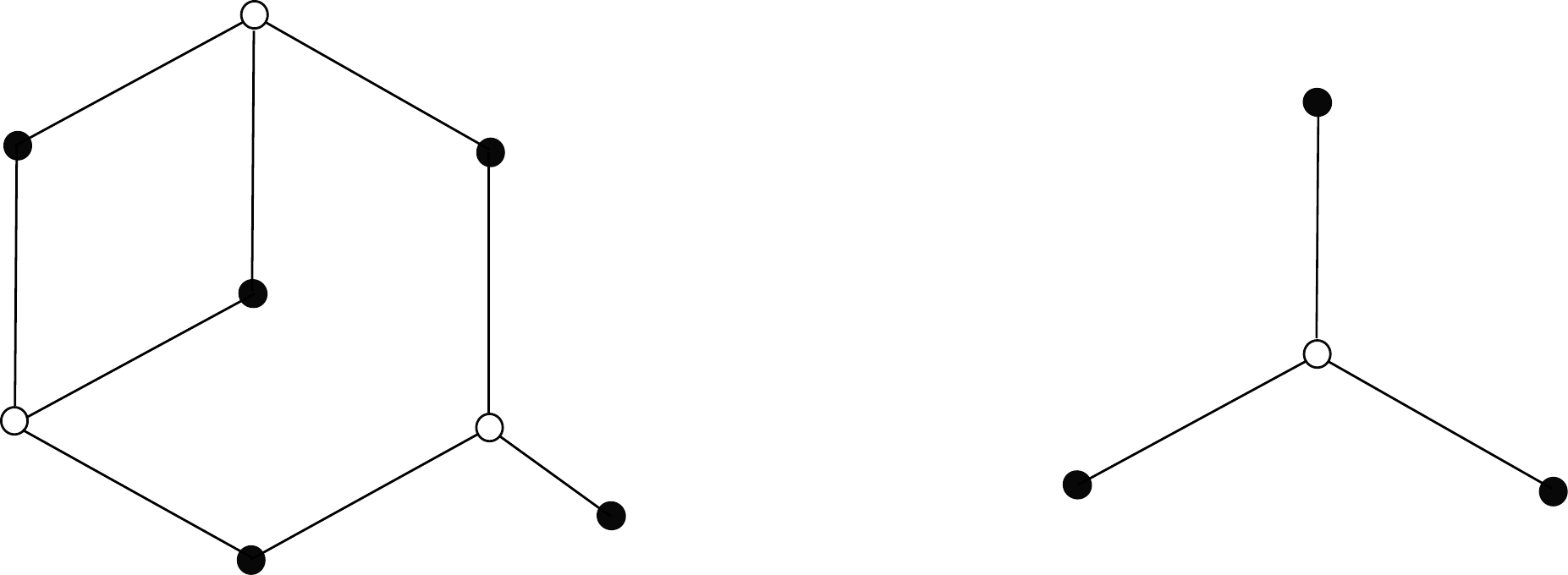}
\put(-15,18){{$ \Lambda=$}}
\put(55,18){{$ \Lambda'=$}}
\put(15,-2){{\scriptsize$y_1$}}
\put(18,16){{\scriptsize$y_2$}}
\put(-4,27){{\scriptsize $y_3$}}
\put(33,27){{\scriptsize $y_4$}}
\put(41,3){{\scriptsize $y_5$}}
\end{overpic}
\caption{}
\label{fig:NoHalfCover}
\end{figure}
Since all Type~1 vertices in $\Lambda$ and $\Lambda'$ are of valence 3, the groups $G$ and $G'$ have isomorphic JSJ trees.  Now the graph~$\Lambda'$ does not half-cover any other graph, so if any graph is half-covered by both $\Lambda$ and~$\Lambda'$, then there is a half-covering $\theta:\Lambda \to \Lambda'$.  The map $\theta$ must send each of the triples of  Type~2 vertices $\{y_1, y_2, y_3\}$, $\{y_2, y_3, y_4\}$, and $\{y_1, y_4, y_5\}$ of $\Lambda$ to the three distinct Type 2 vertices of~$\Lambda'$, but this is impossible.  Hence there is no (minimal) graph which is half-covered by both $\Lambda$ and $\Lambda'$.  

Instead, we may approach commensurability by considering common finite half-covers of $\Lambda$ and~$\Lambda'$.  More precisely, we consider common finite half-covers of finite, connected, bipartite graphs $\Lambda$ and $\Lambda'$ in which all Type 1 vertices have valence $\geq 3$, such that $\Lambda$ and $\Lambda'$ are both half-covered by the same (infinite) bipartite tree in which all Type 2 vertices have countably infinite valence.  With these assumptions, any corresponding geometric amalgams of free groups $G = G(\Lambda)$ and $G' = G(\Lambda')$ will have isomorphic JSJ trees (a necessary condition for commensurability).  Using similar methods to those of Leighton~\cite{leighton} (see also Neumann~\cite{neumann-on-leighton}), we can construct a common finite half-cover $\Lambda''$ of any two such graphs $\gL$ and $\gL'$.  

However, a half-covering of  JSJ graphs does not necessarily induce a topological covering map of associated surface amalgams.  For example, in Figure~\ref{figure:suff_cond_2}, the JSJ graph of $\pi_1(\cX)$ does half-cover the JSJ graph of $\pi_1(\cW)$, but by considering Euler characteristics, one sees that the surface amalgam~$\cX$ does not cover the surface amalgam $\cW$.  Hence half-coverings of graphs may not induce finite-index embeddings of associated geometric amalgams of free groups.  Indeed, our results distinguishing commensurability classes can be used to construct examples where for certain surface amalgams $\cX$ over $\Lambda$ and $\cX'$ over $\Lambda'$, there is no surface amalgam $\cX''$ over a common finite half-cover $\Lambda''$ so that~$\cX''$ covers both $\cX$ and $\cX'$.  Such examples exist even when $\gL$ and $\gL'$ are isomorphic graphs.
A priori, we see no reason why there would be a bound on the size of common half-covers so that in determining whether two groups are commensurable, it is enough to look at the set of common half-covers of their JSJ graphs up to a given size and determine whether suitable surface amalgams over these graphs exist.  

Finally we remark that the proofs of our necessary and sufficient conditions for cycles of generalized $\Theta$-graphs make substantial use of $\Lambda$ and $\Lambda'$ both having a distinguishable central vertex, so it is not clear how to generalize the commensurability classification to arbitrary pairs of JSJ graphs.

\bibliographystyle{plain}
\bibliography{refs} 

\end{document}